\documentclass{jmlr} 

\usepackage{algorithm,algorithmic}
\usepackage{amssymb, multirow, paralist, color}
\newtheorem{thm}{Theorem}
\newtheorem{prop}{Proposition}
\newtheorem{cor}[thm]{Corollary}
\newtheorem{ass}{Assumption}
\usepackage{enumitem}
\usepackage{textcase,booktabs}
\usepackage[T1]{fontenc}

\def \S {\mathbf{S}}
\def \A {\mathcal{A}}
\def \X {\mathcal{X}}

\def \R {\mathbb{R}}

\def \w {\mathbf{w}}
\def \v {\mathbf{v}}

\def \x {\mathbf{x}}

\def \E {\mathrm{E}}

\def \x {\mathbf{x}}

\def \a {\mathbf{a}}
\def \diag {\mbox{diag}}
\def \b {\mathbf{b}}

\def \z {\mathbf{z}}
\def \s {\mathbf{s}}

\def \y {\mathbf{y}}
\def \u {\mathbf{u}}

\def \g {\mathbf{g}}

\def \P {\mathcal{P}}

\def \xh {\widehat{\x}}

\def \C {\mathbf C}

\def \y {\mathbf{y}}
\def \E {\mathrm{E}}
\def \x {\mathbf{x}}
\def \g {\mathbf{g}}

\def \z {\mathbf{z}}
\def \u {\mathbf{u}}

\def \w {\mathbf{w}}

\def \R {\mathbb{R}}
\def \S {\mathcal{S}}

\def \A {\mathcal{A}}

\def \v {\mathbf{v}}

\def \a {\mathbf{a}}
\def \b {\mathbf{b}}

\def \s {\mathbf{s}}

\def \C {\mathcal{C}}

\def \xh {\widehat{\x}}

\def \X {\mathcal{X}}

\def \P {\mathbb{P}}

\begin{document}

\title[SO for DC Functions and Non-smooth Non-Convex Regularizers]{Stochastic Optimization for DC Functions and Non-smooth Non-convex Regularizers with Non-asymptotic Convergence}
\author{\Name{Yi Xu}$^\dagger$\Email{yi-xu@uiowa.edu}\\
\Name{Qi Qi}$^\dagger$\Email{qi-qi@uiowa.edu}\\
 \Name{Qihang Lin}$^\ddagger$ \Email{qihang-lin@uiowa.edu}\\
 \Name{Rong Jin}$^\natural$\Email{jinrong.jr@alibaba-inc.com}\\
\Name{Tianbao Yang}$^\dagger$\Email{tianbao-yang@uiowa.edu}\\
   \addr $^\dagger$Department of Computer Science, The University of Iowa, Iowa City, IA 52242, USA  \\
   \addr$^\ddagger$Department of Management  Sciences,  The University of Iowa, Iowa City, IA 52242, USA\\
   \addr$^\natural$Machine Intelligence Technology, Alibaba Group, Bellevue, WA 98004, USA\\
}

\maketitle
\vspace*{-0.5in}
 \begin{center}First version: November 28, 2018\\ Revised version: February 4, 2019~\footnote{In the revised version, we present some improved complexity results for non-smooth and non-convex regularizers and for functions with known H\"{o}lder continuity  parameter $\nu\in(0,1]$ by a simple change of an algorithmic parameter.}\end{center}

\begin{abstract}
Difference of convex (DC) functions cover a broad family of non-convex and possibly  non-smooth and non-differentiable  functions, and have wide applications in machine learning and statistics. Although deterministic algorithms for DC functions have been extensively studied, stochastic optimization that is more suitable for learning with big data remains under-explored. In this paper, we propose new stochastic optimization algorithms and study their first-order convergence theories  for solving a broad family of DC functions.  We improve the existing algorithms and theories of stochastic optimization for DC functions  from both practical and theoretical perspectives. On the practical side, our algorithm is more user-friendly without requiring a large mini-batch size and more efficient by saving unnecessary computations. On the theoretical side,  our convergence analysis does not necessarily require the involved functions to be smooth with Lipschitz continuous gradient. Instead, the convergence rate  of the proposed stochastic algorithm is automatically adaptive to the H\"{o}lder continuity of the gradient of one component function. Moreover, we extend the proposed stochastic algorithms for DC functions to solve problems with a general  non-convex non-differentiable  regularizer, which  does not necessarily have a DC decomposition but enjoys an efficient proximal mapping.  To the best of our knowledge, this is the first work that gives the first non-asymptotic convergence for solving non-convex optimization whose objective has a general non-convex non-differentiable  regularizer. 
\end{abstract}

\section{Introduction}
In this paper, we consider a  family of non-convex non-smooth optimization problems that can be written in the following form:
\begin{align}\label{eqn:P1}
 \min_{\x\in\R^d} g(\x)  + r(\x) - h(\x),
\end{align}
where $g(\cdot)$ and $h(\cdot)$ are real-valued lower-semicontinuous  convex functions, $r(\cdot)$ is a proper lower-semicontinuous function. We include the component $r$ in order to capture non-differentiable functions that usually plays the role of regularization, e.g., the indicator function of a convex set $\X$ where $r(\x) = \delta_{\X}(\x)$ is zero if $\x\in\X$ and infinity otherwise, and a non-differential regularizer such as the convex $\ell_1$ norm $\|\x\|_1$ or the  non-convex $\ell_0$ norm and $\ell_p$ norm $\|\x\|^p_p$ with $p\in(0,1)$. We do not necessarily impose smoothness condition  on $g(\x)$ or $h(\x)$ and the convexity condition on $r(\x)$. 

A special class of the problem~(\ref{eqn:P1}) is the one with $r(\x)$ being  a convex function - also known as difference of convex (DC) functions.  We would like to mention that even the family of DC functions is broader enough to cover many interesting non-convex problems that are well-studied, including an additive composition of a smooth non-convex function and a non-smooth convex function, weakly convex functions, etc. We postpone this discussion to Section~\ref{sec:pre} after we formally introduce the definitions of smooth functions and weakly convex functions.

In the literature, deterministic algorithms for DC problems have been studied extensively since its introduction by Pham Dinh Tao in 1985 and are continuously receiving attention from the community~\citep{pmlr-v80-khamaru18a,Wen2018}. Please refer to~\citep{LeThi2018} for a survey on this subject.  Although stochastic optimization (SO) algorithms for the special cases of DC functions mentioned above (smooth non-convex functions, weakly convex functions) have been well studied recently~\citep{davis2017proximally,sgdweakly18,modelweakly18,Drusvyatskiy2018,chen18stagewise,DBLP:journals/corr/abs/1805.05411,DBLP:conf/icml/Allen-Zhu17,chen18stagewisekatyusha,DBLP:conf/icml/ZhuH16,DBLP:conf/cdc/ReddiSPS16,Reddi:2016:SVR:3045390.3045425,zhang2018convergence}, a comprehensive study  of SO algorithms with a broader applicability to the DC functions and the problem~(\ref{eqn:P1}) with a non-smooth  non-convex regularizer $r(\x)$ still remain rare. The papers by~\cite{pmlr-v54-nitanda17a} and~\cite{pmlr-v70-thi17a} are the most related works dedicated to the stochastic optimization of special DC functions. \cite{pmlr-v70-thi17a} considered a special class of DC problem whose objective function consists of a large sum of nonconvex smooth functions and a regularization term  that can be written as a DC function. They reformulated the problem into~(\ref{eqn:P1}) such that $h$ is a sum of $n$ convex functions, and $g$ is a quadratic function and $r$ is the first component of the DC decomposition of the regularizer. Regarding algorithm and convergence, they proposed a stochastic variant of the classical DCA (Difference-of-Convex Algorithm) and established an asymptotic convergence result for finding a critical point. To our knowledge, the paper by \cite{pmlr-v54-nitanda17a} is the probably the first result that gives non-asymptotic convergence for finding an approximate critical point of a special class of DC problem, in which both $g$ and $h$ can be stochastic functions and $r=0$. 
Their algorithm consists of multiple stages of  solving a convex objective that is constructed by linearizing $h(\x)$ and adding a quadratic regularization. 
However, their algorithm and convergence theory have the following drawbacks. First, at each stage, they need to compute an unbiased stochastic gradient denoted by $\v(\x)$ of $\nabla h(\x)$ such that $\E[\|\v(\x) - \nabla h(\x)\|^2]\leq \epsilon^2$, where $\epsilon$ is the accuracy level imposed on the returned solution in terms of the gradient's norm. In reality, one has to resort to mini-batching technique by using a large number of samples to ensure this condition, which is impractical and not user-friendly. An user has to worry about what is the size of the mini-batch in order to find a sufficiently accurate solution while keeping the computational costs minimal. Second, for each constructed convex subproblem, their theory requires running a stochastic algorithm that solves each subproblem to the accuracy level of $\epsilon$, which could waste a lot of computations at earlier stages.  Third, their convergence analysis requires that $r(\x)=0$ and $g(\x)$ is a smooth function with a Lipchitz continuous gradient. 

\paragraph {Our Contributions - I.} In Section~\ref{sec:DC}, we propose new stochastic optimization algorithms and establish their convergence results for solving the DC class of the problem~(\ref{eqn:P1}) that improves the algorithm and theory in~\cite{pmlr-v54-nitanda17a} from several perspectives. It is our intension to address the aforementioned drawbacks of their algorithm and theory. In particular, (i) our algorithm only requires unbiased stochastic (sub)-gradients of $g(\x)$ and $h(\x)$ without a requirement on the small variance of the used stochastic (sub)-gradients; (ii) we do not need to solve each constructed subproblem to the accuracy level of $\epsilon$. Instead, we allow the accuracy  for solving each constructed subproblem to grow slowly without sacrificing the overall convergence rate; (iii) we improve the convergence theory significantly. First, our convergence analysis does not require $g(\x)$ to be smooth with a Lipchitz continuous gradient. Instead, we only require either $g(\x)+r(\x)$ or $h(\x)$ to be differentiable with a H\"{o}lder continuous gradient, under the former condition $h(\x)$ can be a non-smooth non-differentiable function  and under the later condition $r(\x)$ and $g(\x)$ can be non-smooth non-differentiable functions. Second, the convergence rate is automatically adaptive to the H\"{o}lder continuity of the involved function without requiring the knowledge of the H\"{o}lder continuity to run the algorithm. Third, when adaptive stochastic gradient method is employed to solve each subproblem, we establish an adaptive convergence similar to existing theory of AdaGrad for convex problems~\citep{duchi2011adaptive,SadaGrad18} and weakly convex problems~\citep{chen18stagewise}. 

\paragraph{Our Contributions - II. }Moreover, in Section~\ref{sec:nsncr} we extend our algorithm and theory to the more general class of non-convex non-smooth problem~(\ref{eqn:P1}), in which $r(\x)$ is a general non-convex non-differentiable regularizer that enjoys an efficient proximal mapping. Although such kind of non-smooth non-convex regularization has been considered in literature~\citep{Attouch2013,Bolte:2014:PAL:2650160.2650169,Bot2016,Li:2015:APG:2969239.2969282,YuZMX15,leiyangpg18,Liu2018,doi:10.1080/02331934.2016.1253694,DBLP:conf/aaai/ZhongK14}, existing results are restricted to deterministic optimization and asymptotic or local convergence analysis. In addition, most of them consider a special case of our problem with $g - h$ being a smooth non-convex function. To the best of our knowledge, this is the first work of stochastic optimization  with a  non-asymptotic first-order convergence result for tackling the non-convex objective~(\ref{eqn:P1}) with a non-convex non-differentiable regularization and a smooth function $g$ and a possibly non-smooth function $h$ with a H\"{o}lder continuous gradient.  Our algorithm and theory are based on using the Moreau envelope  of $r(\x)$ that can be written as a DC function, which then reduces to the problem that is studied in Section~\ref{sec:DC}. By using the algorithms and their convergence results established in  Section~\ref{sec:DC} and carefully controlling the approximation parameter,  we establish the first non-asymptotic convergence of stochastic optimization for solving the original non-convex problem with a  non-convex non-differentiable regularizer. This non-asymptotic convergence result can be also easily extended to the deterministic optimization, which itself is novel and could be interesting to a broader community.  A summary of our results is presented in Table~\ref{tab:2}. 
\begin{table*}[t]
		\caption{Summary of results presented in this paper for finding a (nearly) $\epsilon$-critical point of the problem~(\ref{eqn:P1}), where $g$ and $h$ are assumed to be convex.  
		HC refers to H\"{o}lder continuous gradient condition; SM refers to the smooth condition; CX means convex; NC means non-convex and NS means non-smooth; LP denotes Lipchitz continuous function; LB means lower bounded over $\R^d$; FV means finite-valued over $\R^d$; FVC means finite-valued over a compact set. $\nu\in(0,1]$ denotes the power constant of the involved function's H\"{o}lder continuity.  $n$ denotes the total number of components in a finite-sum problem. SPG denotes stochastic proximal gradient algorithm. SVRG denotes stochastic variance reduced gradient algorithm. AdaGrad denotes adaptive stocahstic gradient method. AG denotes accelerated gradient methods. Complexity for SPG and AdaGrad means iteration complexity, and for SVRG and AG means gradient complexity. 
		 }
		\centering
		\label{tab:2}
		{\begin{tabular}{l|l|l|ll}
			\toprule
			$g$ & $h$ & $r$ &Algorithms for subproblems&Complexity\\
			\midrule
			-& HC &CX&SPG, AdaGrad&$O(1/\epsilon^{4/\nu})$\\
			SM& HC &CX&SVRG&$\widetilde O(n/\epsilon^{2/\nu})$\\
			HC& - &CX, HC&SPG, AdaGrad&$O(1/\epsilon^{4/\nu})$\\
			SM& - &CX, HC&SVRG&$\widetilde O(n/\epsilon^{2/\nu})$\\
			\midrule
			SM&HC&NC, NS, LP&SPG&$O(1/\epsilon^{5(1+1/\nu)/2})$\\
			SM&HC&NC, NS, FV, LB&SPG&$O(1/\epsilon^{2(1+2/\nu)})$\\
			SM&HC&NC, NS, LP& SVRG, AG&$\widetilde  O(n/\epsilon^{3(1+1/\nu)/2})$\\
			SM&HC&NC, NS, FV, LB&SVRG, AG&$\widetilde  O(n/\epsilon^{4(1+2/\nu)/3})$\\
			SM&HC&NC, NS, FVC &SVRG, AG&$\widetilde  O(n/\epsilon^{4(1+2/\nu)/3})$\\
     	\bottomrule
		\end{tabular}}
	\end{table*}

\begin{table*}[t]
		\caption{Summary of improved complexities when $\nu$ is known
		 }
		\centering
		\label{tab:new}
		{\begin{tabular}{l|l|l|ll}
			\toprule
			$g$ & $h$ & $r$ &Algorithms for subproblems&Complexity\\
			\midrule
			-& HC &CX&SPG, AdaGrad&$O(1/\epsilon^{(1+3\nu)/\nu})$\\
			SM& HC &CX&SVRG&$\widetilde O(n/\epsilon^{(1+\nu)/\nu})$\\
			HC& - &CX, HC&SPG, AdaGrad&$O(1/\epsilon^{(1+3\nu)/\nu})$\\
			SM& - &CX, HC&SVRG&$\widetilde O(n/\epsilon^{(1+\nu)/\nu})$\\
						\midrule
			SM&HC&NC, NS, LP&SPG&$O(1/\epsilon^{4+1/\nu})$\\
			SM&HC&NC, NS, FV, LB&SPG&$O(1/\epsilon^{4+2/\nu)})$\\
			SM&HC&NC, NS, LP& SVRG, AG&$\widetilde  O(n/\epsilon^{2+1/\nu})$\\
			SM&HC&NC, NS, FV, LB&SVRG, AG&$\widetilde  O(n/\epsilon^{2+2/\nu})$\\
			SM&HC&NC, NS, FVC &SVRG, AG&$\widetilde  O(n/\epsilon^{2+2/\nu})$\\
     	\bottomrule
		\end{tabular}}
	\end{table*}



\section{Preliminaries}\label{sec:pre}
In this section, we present some preliminaries. Let $\|\cdot\|_p$ denote the standard $p$-norm with $p\geq 0$. 
For a non-convex function $f(\x): \R^d\rightarrow \R$, let $\hat\partial f(\x)$ denote the Fr\'{e}chet subgradient and $\partial f(\x)$ denote the limiting subgradient, i.e., 
\begin{align*}
\hat\partial f(\bar\x)  & = \left\{\v\in\R^d: \lim_{\x\rightarrow\bar \x}\inf \frac{f(\x) - f(\bar \x) - \v^{\top}(\x - \bar \x)}{\|\x - \bar\x\|}\geq 0\right\},\\
\partial f(\bar\x) &  =  \{\v\in\R^d: \exists \x_k \xrightarrow[]{f} \bar\x, v_k\in\hat\partial f(\x_k), \v_k\rightarrow \v\},
\end{align*} 
where the notation $\x\xrightarrow[]{f} \bar \x$ means that $\x\rightarrow \bar\x$ and $f(\x)\rightarrow f(\bar\x)$. It is known that $\hat\partial f(\x)\in\partial f(\x)$.  If $f(\cdot)$ is differential at $\x$, then $\hat\partial f(\x) = \{\nabla f(\x)\}$. Moreover, if $f(\x)$ is continuously differentiable on a neighborhood of $\x$, then $\partial f(\x) = \{\nabla f(\x)\}$. When $f$  is convex, the Fr\'{e}chet and the limiting subgradient reduce to the subgradient in the sense of convex analysis: $\partial f(\x) = \{\v\in\R^d: f(\x)\geq f(\y) + \v^{\top}(\x - \y), \forall\y\in\R^d\}$.  For simplicity, we use $\|\cdot\|$ to denote the Euclidean norm (aka. $2$-norm) of a vector. Let $\text{dist}(\S_1, \S_2)$ denote the distance between two sets. 

A function $f(\x)$ is smooth with a $L$-Lipchitz continuous gradient if it is differentiable and the following inequality holds
\begin{align*}
\|\nabla f(\x)  - \nabla f(\y)\|\leq L\|\x - \y\| ,\forall\x, \y.
\end{align*}
A differentiable function $f(\x)$ has $(L, \nu)$-H\"{o}lder continuous gradient if there exists $\nu\in(0,1]$ such that 
\begin{align*}
\|\nabla f(\x)  - \nabla f(\y)\|\leq L\|\x - \y\|^\nu, \forall \x, \y.
\end{align*}
Next, let us characterize the critical points of the considered problem~(\ref{eqn:P1}) that are standard in the literature~\citep{10.1007/978-3-642-45610-7_3,Horst1999,LeThi2018,doi:10.1080/02331934.2016.1253694}, and introduce the convergence measure for an iterative optimization algorithm. First, let us consider the DC problem:
\begin{align}\label{eqn:GP}
\min_{\x\in\R^d} f(\x):=g(\x) - h(\x)
\end{align}
where $g: \R^d\rightarrow \R\cup\{\infty\}$ is a proper lower semicontinuous convex function and $h:\R^d\rightarrow R$ is convex. If $\bar\x$ is a local minimizer of $f(\x)$, then $\partial h(\bar\x)\subset\partial g(\bar\x)$.  Any point $\bar\x$ that satisfies the condition $\partial h(\bar\x)\subset\partial g(\bar\x)$ is called a stationary point of~(\ref{eqn:GP}) and any point $\bar\x$ such that $\partial h(\bar\x)\cap\partial g(\bar\x)\neq\emptyset$ is called a critical point of~(\ref{eqn:GP}). If $h(\x)$ is further differentiable, the stationary points and the critical points coincide. For an iterative optimization algorithm, it is hard to find an exactly critical point in a finite-number of iterations. Therefore, one is usually concerned with finding an $\epsilon$-critical point $\x$ that satisfies 
\begin{align}
 \text{dist}(\partial h(\x), \partial g(\x))\leq \epsilon.
\end{align}
Similarly, we can extend the above definitions of stationary and critical points to the general problem~(\ref{eqn:P1}) with $r(\x)$ being a proper and lower semi-continuous (possibly non-convex) function~\citep{doi:10.1080/02331934.2016.1253694}. In particular, $\bar\x$ is called a stationary point of the considered problem~(\ref{eqn:P1}) if it satisfies  $\partial h(\x)\subset \hat\partial (g+r)(\x)$,  and any point $\bar \x$  such that $\partial h(\bar\x)\cap \hat\partial (g+r)(\bar\x)\neq\emptyset$ is called a critical point of~(\ref{eqn:P1}). When $g$ is differentiable,  $\hat\partial (g+r)(\bar\x) = \nabla g(\x) + \hat\partial r(\x)$~\citep{RockWets98}[Exercise 8.8], and when both $g$ and $r$ are convex and their domains cannot be separated $\hat\partial (g+r)(\bar\x) = \partial g(\x) + \partial r(\x)$~\citep{RockWets98}[Corollary 10.9].  
An $\epsilon$-critical point of~(\ref{eqn:P1}) is a point $\x$ that satisfies  $ \text{dist}( \partial h(\x), \hat\partial (g+ r)(\x))\leq \epsilon$. It is notable that when $g + r$ is non-differentiable, finding an $\epsilon$-critical point could become a challenging task for an iterative algorithm even under the condition that $r$ is  a convex function. Let us consider the example of $g=h=0, r=|x|$. As long as $x\neq 0$, we have $\text{dist}(0, \partial|x|) = 1$. To address this challenge when $g + r$ is non-differentiable, we introduce the notion of nearly $\epsilon$-critical points. In particular, a point $\x$ is called a nearly $\epsilon$-critical point of the problem~(\ref{eqn:P1}) if there exists $\bar\x$ such that 
\begin{align}
\|\x - \bar\x\|\leq O(\epsilon), \quad  \text{dist}( \partial h(\bar\x), \hat\partial(g+ r)(\bar\x))\leq \epsilon.
\end{align}
A similar notion of nearly critical points for non-smooth and non-convex optimization problems have been utilized in several recent works~\citep{davis2017proximally,sgdweakly18,modelweakly18,chen18stagewise}.

\paragraph{Examples and Applications of DC functions.} Before ending this section, we present some examples of DC functions and their applications in machine learning and statistics. 

{\it Example 1: Additive composition of a smooth loss function and a non-smooth convex regularizer.} Let us consider 
\begin{align*}
\min_{\x\in\R^d}g(\x) + r(\x),
\end{align*}
where $r(\x)$ is a convex function and $g(\x)$ is an $L$-smooth function. For an $L$-smooth function, it is clear that $\hat g(\x) = g(\x) + \frac{L}{2}\|\x\|^2$ is a convex function. Therefore, the above objective function can be written as $\hat g(\x) + r(\x) - \frac{L}{2}\|\x\|^2$ - a DC function.

{\it Example 2: Weakly convex functions.} Weakly convex functions have been recently studied in numerous papers~\citep{davis2017proximally,sgdweakly18,modelweakly18,chen18stagewise,zhang2018convergence}. A function $f(\x)$ is called $\rho$-weakly convex if $f(\x) + \frac{\rho}{2}\|\x\|^2$ is a convex function. More generally, $f(\x)$ is called $\rho$-relative convex with respect to a strongly convex function $\omega(\x)$ if $f(\x) + \rho\omega(\x)$ is convex~\citep{zhang2018convergence}. It is obvious that a weakly convex function $f(\x)$ is a DC function. Examples of weakly convex functions can be found in deep neural networks with a smooth active function and a non-smooth loss function~\citep{chen18stagewise},  robust learning~\citep{DBLP:journals/corr/abs-1805-07880}, robust phase retrieval~\citep{modelweakly18}, etc. 

{\it Example 3: Non-Convex Sparsity-Promoting Regularizers.}
Many non-convex sparsity-promoting regularizers in statistics can be written as a DC function, including log-sum penalty (LSP)~\citep{Candades2008}, minimax concave penalty (MCP)~\citep{cunzhang10}, smoothly clipped absolute deviation (SCAD)~\citep{CIS-172933},  capped $\ell_1$ penalty~\citep{Zhang:2010:AMC:1756006.1756041}, transformed $\ell_1$ norm~\cite{DBLP:journals/corr/ZhangX14}. For detailed DC composition of these regularizers, please refer to~\citep{Wen2018,DBLP:conf/icml/GongZLHY13}.  It is notable that for LSP, MCP and SCAD, the second function in their DC decomposition can be a smooth function. In particular, if one consider regression or classification with LSP, MCP, SCAD or transformed $\ell_1$ norm regularizer, the problem is a special case of~(\ref{eqn:P1}) with $r(\x)$ being a convex function and $h(\x)$ being a smooth convex function.  Here we give one example by considering learning with MCP as a regularization, where the problem is
\begin{align}
\min_{\x\in\R^d}\frac{1}{n}\sum_{i=1}^n\ell(\x^{\top}\a_i, b_i) +\underbrace{ \lambda\sum_{i=1}^d\int^{|w_i|}_0\left[1 - \frac{x}{\theta\lambda}\right]_+dx}\limits_{P(\x)},
\end{align}
where $(\a_i, b_i), i=1,\ldots, n$ denote a set of data points (feature vector and label pairs), $\ell(\cdot, \cdot)$ is a convex loss function with respect to its first argument, $\theta>0$ is a constant and $\lambda>0$ is a regularization parameter. We can write $P(\x)$ as a difference of two convex functions
\begin{align*}
P(\x) = \lambda\|\x\|_1 - \underbrace{\lambda\sum_{i=1}^d\int_{0}^{|x_i|}\min\{1, \frac{x}{\theta\lambda}\} dx}\limits_{h(\x)},
\end{align*}
where $h(\x)$ is continuously differentiable with $\frac{1}{\theta}$-Lipchitz continuous gradient~\cite{}. 

{\it Example 4: Least-squares Regression with $\ell_{1-2}$ Regularization}. Recently, a non-convex regularization in the form of $\lambda(\|\x\|_1 - \|\x\|_2)$ was proposed for least-squares regression or compressive sensing~\citep{Yin2015MinimizationO}, which is naturally a DC function.  

{\it Example 5: Positive-Unlabeled (PU) Learning}. A standard learning task is to find a model denoted by $\x$ that minimizes the expected risk based on a convex surrogate loss $\ell$, i.e., 
\begin{align*}
\min_{\x\in\R^d}\E_{\z, y}[\ell(\x; \z, y)],
\end{align*}
where $\z\in\R^m$ denotes the feature vector of a random data and $y\in\{1, -1\}$ denotes its corresponding label. In practice one observes a finite set of i.i.d training data $\{\z_i, y_i\}, i=1\ldots, n$, which leads to the well-known empirical risk (ERM) minimization problem, i.e.,  $\min_{\x\in\R^d}\frac{1}{n}\sum_{i=1}^n\ell(\x; \z_i, y_i)$. However, if only positive data $\{\z_i, +1,  i=1, \ldots, n_+\},$ are observed, ERM becomes problematic. A remedy to address this challenge is to use unlabeled data for computing an unbiased estimation of $\E_{\z, y}[\ell(\x; \z, y)]$. In particular, the objective in the following problem is an unbiased risk~\citep{NIPS2017_6765}:
\begin{align*}
\min_{\x\in\R^d}\frac{\pi_p}{n_+}\sum_{i=1}^{n_+}\left(\ell(\x; \z_i, +1) - \ell(\x; \z_i, -1)\right) + \frac{1}{n_u}\sum_{j=1}^m\ell(\x; \z^u_j, -1),
\end{align*}
where $\{\z_i^u, i=1, \ldots, n_u\}$ is a set of unlabeled data, and $\pi_p = \Pr(y=1)$ is the prior probability of the positive class. It is obvious that if $\ell(\x; \cdot)$ is a convex loss function in terms of $\x$, the above objective function  is a DC function. In practice, an estimation of $\pi_p$ is used.

\paragraph{Examples of Non-Convex Non-Smooth Regularizers.} Finally, we present some  examples of non-convex non-smooth regularizers $r(\x)$  that cannot be written as a DC function or whose DC composition is unknown. Thus, the algorithms and theories presented in Section~\ref{sec:DC} are not directly applicable, but the algorithms discussed in Section~\ref{sec:nsncr} are applicable when the proximal mapping of each component of $r(\x)$ is efficient to compute.  Examples include $\ell_0$ norm (i.e., the number of non-zero elements of a vector) and  $\ell_p$ norm regularization for $p\in(0,1)$ (i.e., $\sum_{i=1}^d|x_i|^p$), whose proximal mapping can be efficiently computed~\citep{Attouch2013,Bolte:2014:PAL:2650160.2650169}. For another example, let us consider a penalization approach for tackling  non-convex  constraints. Consider a non-convex optimization problem with domain constraint $\x\in\C$, where $\C$ is a non-convex set. Directly handling a non-convex constrained problem could be difficult. An alternative solution is to convert the domain constraint into a penalization $\frac{\lambda}{2}\|\x - \P_\C(\x)\|^2$ with $\lambda>0$ in the objective, where $\P_\C(\cdot)$ denotes the projection of a point to the set $\C$. Note that when $\C$ is a non-convex set, $r(\x)=\frac{\lambda}{2}\|\x - \P_\C(\x)\|^2$ is a non-convex non-smooth function in general, and its proximal mapping enjoys a closed-form solution~\citep{DBLP:journals/mp/LiP16}.

As a final remark, it is worth mentioning that even if $r(\x)$ can be written as a DC function such that the two components in its DC decomposition are both non-smooth non-differntiable (e.g., $\ell_{1-2}$ regularization, capped $\ell_1$ norm $\sum_{i=1}^d\min(|x_i|, \theta)$), the theory  presented in Section~\ref{sec:nsncr} can be still useful to derive a non-asymptotic first-order convergence in terms of finding a close critical point, while the theory in Section~\ref{sec:DC} is not directly applicable.


\section{New Stochastic Algorithms of DC functions}\label{sec:DC}
In this section, we present new stochastic algorithms for solving the problem~(\ref{eqn:P1}) when $r(\x)$ is a convex function and their convergence results. We assume both $g(\x)$ and $h(\x)$ have a large number of components such that computing a stochastic gradient is much more efficient than computing a deterministic gradient. Without loss of generality, we assume $g(\x) = \E_\xi[g(\x; \xi)]$ and $h(\x) = \E_{\varsigma}[h(\x; \varsigma)]$, and consider the following problem:
\begin{align}\label{eqn:PS}
\min_{\x\in\R^d}  F(\x) : = \E_\xi[g(\x; \xi)] + r(\x) -  \E_{\varsigma}[h(\x; \varsigma)].
\end{align}
where $g, h: \R^d\rightarrow\R^d$ are real-valued lower-semicontinuous convex functions and $r$ is a proper lower-semicontinuous convex function. 
It is notable that a special case of this problem is the finite-sum form: 
\begin{align}
\min_{\x\in\R^d}  F(\x):=\frac{1}{n_1}\sum_{i=1}^{n_1} g_i(\x) + r(\x) - \frac{1}{n_2}\sum_{j=1}^{n_2}h_j(\x),
\end{align}
which will allows us to develop faster algorithms for smooth functions by using variance reduction techniques.

Since we do not necessarily impose any smoothness assumption on $g(\x)$ and $h(\x)$, we will postpone the particular assumptions for these functions in the statements of later theorems.  For all algorithms presented below, we assume that the proximal mapping of $r(\x)$  can be efficiently computed, i.e., the solution to the following problem  can be easily computed for any $\eta>0$:
\begin{align*}
\min_{\x\in\R^d} \frac{1}{2\eta}\|\x - \y\|^2 + r(\y).
\end{align*}
A basic assumption that will be used in the analysis is the following. 
\begin{ass}\label{ass:0}
For a given initial solution $\x_1\in\text{dom}(r)$, assume that there exists $\Delta>0$ such that $F(\x_1) - \inf_{\x\in\R^d}F(\x)\leq \Delta$. 
\end{ass}


The basic idea of the proposed algorithm is similar to the stochastic algorithm proposed in~\citep{pmlr-v54-nitanda17a}. The algorithm consists of multiple stages of solving convex problems. At the $k$-th stage ($k\geq 1$), given a point $\x_k$, a convex majorant function $F^\gamma_{\x_k}(\x)$ is constructed as following such that $F^{\gamma}_{\x_k}(\x) \geq F(\x), \forall \x$ and $F^{\gamma}_{\x_k}(\x_k)  = F(\x_k)$:
\begin{align}
F^{\gamma}_{\x_k}(\x) &=  g(\x)  + r(\x) - (h(\x_k) + \partial h(\x_k)^{\top}(\x - \x_k)) + \frac{\gamma}{2}\|\x - \x_k\|^2,
\end{align}
where $\gamma>0$ is a constant parameter. Then a stochastic algorithm is employed to optimize the convex majorant function. The key difference from the previous work lies at how to solve each convex majorant function. An important change introduced to our design is to make the proposed algorithms more efficient and more practical. Roughly speaking, we only require solving each function $F^\gamma_{\x_k}(\x)$ up to an accuracy level of $c/k$ for some constant $c>0$, i.e., finding a solution $\x_{k+1}$ such that 
\begin{align}\label{eqn:acc}
\E[F^{\gamma}_{\x_k}(\x_{k+1})  - \min_{\x\in\R^d} F^{\gamma}_{\x_k}(\x)  ]\leq \frac{c}{k}.
\end{align}
 In contrast, the algorithm and anlysis presented in~\citep{pmlr-v54-nitanda17a} requires solving each convex problem up to an accuracy level of $\epsilon$, which is the expected accuracy level on the final solution. This change not only makes our algorithms more efficient by saving a lot of unnecessary computations but also more practical without requiring $\epsilon$ to run the algorithm.

\begin{algorithm}[t]
    \caption{A Stagewise Stochastic DC Algorithm: SSDC-$\A$}\label{alg:meta}
    \begin{algorithmic}[1]
        \STATE \textbf{Initialize:}  $\x_1\in\text{dom}(r)$
                \FOR {$k = 1,\ldots, K$}
        \STATE Let $F_{k}(\x)=F^\gamma_{\x_k}=g(\x)  + r(\x)  - h(\x_k)- \partial h(\x_k)^{\top}(\x - \x_k)+\frac{\gamma}{2}\| \x-\x_{k}\|^2$
        \STATE $\x_{k+1} = \A(F^\gamma_{\x_k}, \Theta_k)$ \hfill $\diamond$  $\Theta_k$ denotes algorithm dependent parameters
        \ENDFOR
    \end{algorithmic}
\end{algorithm}

We present a meta algorithm in Algorithm~\ref{alg:meta}, in which $\mathcal A$ refers to an appropriate stochastic algorithm for solving each convex majorant function. The Step 4 means that   $\mathcal A$ is employed for finding $\x_{k+1}$ such that~(\ref{eqn:acc}) is satisfied (or a more fine-grained  condition is satisfied for a particular algorithm as discussed later), where $\Theta_k$ denotes the algorithm dependent parameters (e.g., the number of iterations). 
There are three issues that deserve further discussion in order to fully understand the proposed algorithm. First, how many outer iterations $k$ is needed to ensure finding a (nearly) $\epsilon$-stationary point of the original problem under the condition that~(\ref{eqn:acc}) is satisfied for each problem. Second, how to ensure the condition~(\ref{eqn:acc}) to be satisfied for a stochastic algorithm. Third, what is the overall complexity (iteration complexity or gradient complexity) taking into account the complexity of the stochastic algorithm  $\mathcal A$ for solving each convex majorant function. Note that the last two issues are closely related to the particular algorithm employed. We emphasize that the last two issues are important not only in theory but also in practice.  Related factors such as how to initialize the algorithm  $\mathcal A$, how to set the step size and  how many iterations for each call of $\mathcal A$ suffice have dramatic effect on the practical performance. Next, we first present a general convergence analysis of Algorithm~\ref{alg:meta} under the condition that~(\ref{eqn:acc}) is satisfied for solving each problem. Then, we present several representative stochastic algorithms for solving each convex majorant function and derive their overall iteration complexities.  

Our convergence analysis also has its merits compared with the previous work~\citep{pmlr-v54-nitanda17a}. We will divide our convergence analysis into three parts. First, in subsection~\ref{sec:general} we introduce a general convergence measure without requiring any smoothness assumptions of involved functions and conduct a convergence analysis of the proposed algorithm. Second, we analyze different stochastic algorithms and their convergence results in subsection~\ref{sec:algs}, including an adaptive convergence result for using AdaGrad. Finally, we discuss the implications of these convergence results for solving the original problem in terms of finding a (nearly) $\epsilon$-stationary point in subsection~\ref{sec:stationary}. 

\subsection{A General Convergence Result}\label{sec:general}
For any $\gamma>0$, define
\begin{align*}
P_\gamma(\z) & = \arg\min_{\x\in\R^d} g(\x) + r(\x) -(h(\z) +  \partial h(\z)^{\top}(\x - \z)) + \frac{\gamma}{2}\|\x - \z\|^2,\\
 G_\gamma(\z) & = \gamma (\z - P_\gamma(\z)).
\end{align*}
It is notable that $P_\gamma(\z)$ is well defined since the above problem is strongly convex. The following proposition shows that when $\z = P_\gamma(\z)$, then $\z$ is a critical point of the original problem. 
\begin{prop}\label{prop:1}
If $\z =P_\lambda(\z)$, then $\z$ is a critical  point of the problem $\min_{\x\in\R^d}g(\x)  + r(\x) - h(\x)$. 
\end{prop}
\begin{proof}
According to the first-order optimality condition, we have
\begin{align*}
0\in \partial (g(P_\lambda(\z))+ r(P_\lambda(\z)))  -\partial h(\z) + \gamma(P_\gamma(\z) - \z).
\end{align*}
Since $\z =P_\gamma(\z)$, we have
\begin{align*}
0\in \partial (g+r)(\z)   -\partial h(\z),
\end{align*}
which implies that $\z$ is a critical  point of the original minimization problem. 
\end{proof}
The above proposition implies that $\|G_\gamma(\z)\|= \gamma\|P_\gamma(\z) - \z\|$ can serve as a  measure of convergence of an algorithm for solving the considered minimization problem. In subsection~\ref{sec:stationary}, we will discuss how the convergence in terms of $ \gamma\|P_\gamma(\z) - \z\|$ implies that the standard convergence measure in terms of the (sub)gradient norm of the original problem. The following theorems are the main results of this subsection. 

\begin{thm}\label{thm:2}
Suppose Assumption~\ref{ass:0} holds and  there exists an stochastic algorithm $\mathcal A$  that when applied to $F^\gamma_{\x_k}(\x)$ can find a solution $\x_{k+1}$ satisfying~(\ref{eqn:acc}), then with a total of $K$ stages we have
\begin{align*}
 \E\bigg[\|G_\gamma(\x_{\tau})\|^2\bigg]&\leq \frac{2\gamma\Delta}{K} + \frac{2\gamma c(1+\log (K))}{K},
\end{align*}
where 
$\tau\in\{1,\ldots, K\}$ is uniformly sampled. 
\end{thm}
{\bf Remark:} It is clear that when $K\rightarrow\infty$, $\gamma\|\x_{\tau} - P_\gamma(\x_{\tau})\|\rightarrow 0$ in expectation, implying the convergence to a critical point. Note that the $\log (K) $ factor will lead to an iteration complexity of $O(\log(1/\epsilon)/\epsilon^4)$ for using stochastic (sub)gradient method. This seems to be slightly worse than that presented in~\citep{pmlr-v54-nitanda17a} by a logarithmic factor. However, practically  our algorithms can perform  much better. This is because that if we simply run a stochastic algorithm $\mathcal A$ at each stage to find a solution $\x_{k+1}$ such that $E[F^{\gamma}_{\x_k}(\x_{k+1})  - \min_{\x\in\R^d} F^{\gamma}_{\x_k}(\x)  ]\leq \frac{c}{K}$, one can obtain a convergence upper bound of $O(1/K)$ without a logarithmic factor. However,  the stochastic algorithm $\mathcal A$ will need much more iterations at each stage, leading to a worse performance in practice.

A simple way to get rid of such a logarithmic factor without sacrificing the practical performance is by exploiting non-uniform sampling under  a slightly stronger condition of the problem. 

\begin{thm}\label{thm:3}
Suppose  there exists an stochastic algorithm $\mathcal A$  that when applied to $F^\gamma_{\x_k}(\x)$ can find a solution $\x_{k+1}$ satisfying~(\ref{eqn:acc}), and there exists $\Delta>0$ such that $\E[F(\x_k) - \min_{\x}F(\x)]\leq \Delta$ for all $k\in\{1,\ldots, K\}$, then with a total of $K$ stages we have
\begin{align*}
 \E\bigg[\|G_\gamma(\x_{\tau})\|^2\bigg]&\leq \frac{2\gamma(\Delta+c)(\alpha+1)}{K},
\end{align*}
where 
$\tau\in\{1,\ldots, K\}$ is sampled according to probabilities $p(\tau=k) = \frac{k^\alpha}{\sum_{k=1}^Kk^\alpha}$ with $\alpha\geq 1$. 
\end{thm}
{\bf Remark: } Compared to Theorem~\ref{thm:2}, the condition $\E[F(\x_k) - \min_{\x}F(\x)]\leq \Delta$ for all $k\in\{1,\ldots, K\}$ is slightly stronger than Assumption~\ref{ass:0}. However, it can be easily satisfied if $\x_k\in\text{dom}(r)$ resides in a bounded set  (e.g., when $r(\x)$ is the indicator function of a bounded set), or  if $\E[F(\x_k)]$ is non-increasing (e.g.,  when using variance-reduction methods for the case that $g(\x)$ is smooth).

\begin{proof}[of Theorem~\ref{thm:3}] The proof of Theorem~\ref{thm:2} can be obtained by a slight change of the following proof. 
Define the following notations.  
\begin{align*}
\z_k = P_\gamma(\x_k)=\arg\min_{\x\in\R^d} F_k(\x) := \underbrace{g(\x) + r(\x) - \partial h(\x_k)^{\top}(\x - \x_k)}\limits_{f_k(\x)} + \frac{\gamma}{2}\|\x - \x_k\|^2.
\end{align*}
 By the assumption of~(\ref{eqn:acc}), we have $\E[F_k(\x_{k+1}) -F_k(\z_k)]\leq \epsilon_k = c/k$.  By the strong convexity of $F_k$, we have $F_k(\x_k)\geq F_k(\z_k) + \frac{\gamma}{2}\|\x_k - \z_k\|^2$. Thus we have
\begin{align}\label{eqn:c1}
\E[f_k(\x_{k+1})   + \frac{\gamma}{2}\|\x_{k+1} - \x_k\|^2]&\leq F_k(\x_k) - \frac{\gamma}{2}\|\x_k - \z_k\|^2 + \epsilon_k\nonumber\\
& =  g(\x_k)  + r(\x_k)- \frac{\gamma}{2}\|\x_k - \z_k\|^2 +  \epsilon_k.
\end{align}
Rearranging the terms, we have  
\begin{align*}
 \E\bigg[\frac{\gamma}{2}\|\z_{k} - \x_k\|^2\bigg]&\leq \E[g(\x_k)  + r(\x_k) - f_k(\x_{k+1})]  +  \epsilon_k\\
 &\leq \E[g(\x_k) + r(\x_k) - g(\x_{k+1}) - r(\x_{k+1}) + \partial h(\x_k)^{\top}(\x_{k+1} - \x_k)]  + \epsilon_k\\
 &\leq \E[g(\x_k)  + r(\x_k )- g(\x_{k+1}) -r(\x_{k+1}) + h(\x_{k+1}) - h(\x_k)]  + \epsilon_k\\
 &=  \E[F(\x_k) - F(\x_{k+1})] + \epsilon_k,
\end{align*}
where the last inequality follows the convexity of $h(\cdot)$. Multiplying both sides by $w_k = k^\alpha$ and taking summation over $k=1,\ldots, K$, we have
\begin{align}\label{eqn:c2}
 \E\bigg[\frac{\gamma}{2}\sum_{k=1}^Kw_k\|\z_{k} - \x_k\|^2\bigg]&\leq \E\bigg[ \sum_{k=1}^Kw_k(F(\x_k) - F(\x_{k+1}))\bigg] +\sum_{k=1}^Kw_k \epsilon_k,
\end{align}
The second term in the R.H.S of the above inequality can be easily bounded using simple calculus. For the first term, we use similar analysis as that in the proof of Theorem 1 in~\citep{chen18stagewise}:
\begin{align*}
&\sum_{k=1}^{K} w_k (F(\x_k) - F(\x_{k+1}))= \sum_{k=1}^{K} (w_{k-1}F(\x_{k}) - w_kF(\x_{k+1})) + \sum_{k=1}^{K}(w_k - w_{k-1})F(\x_{k})\\
&= w_0 F(\x_1) - w_{K}F(\x_{K+1}) +\sum_{k=1}^{K}(w_k - w_{k-1})F(\x_{k})\\
& =\sum_{k=1}^{K}(w_s - w_{s-1})(F(\x_{k}) - F(\x_{K+1}))\leq \sum_{k=1}^{K}(w_k - w_{k-1})(F(\x_k) - \min_{\x}F(\x)),
\end{align*}
where we use $\w_0=0$.
Taking expectation on both sides, we have
\begin{align*}
 \E\bigg[ \sum_{k=1}^Kw_k(F(\x_k) - F(\x_{k+1}))\bigg] \leq  \sum_{k=1}^{K}(w_k - w_{k-1})\E[(F(\x_k) - \min_{\x}F(\x))]\leq \Delta w_{K}
\end{align*}
Then, we have
\begin{align*}
 \E\bigg[\frac{\gamma}{2}\|\z_{\tau} - \x_\tau\|^2\bigg]&\leq \frac{\Delta (\alpha+1)}{K} + \frac{c(\alpha+1)}{K},
\end{align*}
which can complete the proof by multiplying both sides by $2\gamma$. 
The result in Theorem~\ref{thm:2} for the uniform sampling can be easily derived from the equality~(\ref{eqn:c2}) by using the fact $\sum_{k=1}^K1/k\leq (1+\log K)$.
\end{proof}

\subsection{Convergence Results of Different Stochastic Algorithms}\label{sec:algs}
In this section, we will present the convergence results of Algorithm~\ref{alg:meta} for employing different stochastic algorithms for minimizing $F_k(\x)$ at each stage. In particular, we consider three representative algorithms, namely stochastic proximal subgradient (SPG) method~\citep{DBLP:conf/colt/DuchiSST10,DBLP:conf/icml/ZhaoZ15}, adaptive stochastic gradient (AdaGrad)  method~\citep{duchi2011adaptive,SadaGrad18}, and proximal stochastic gradient method with variance reduction (SVRG)~\citep{DBLP:journals/siamjo/Xiao014}. SPG is a simple stochastic method, AdaGrad allows us to derive adaptive convergence to the history of learning, and SVRG allows us to leverage the finite-sum structure and the smoothness of the problem to improve the convergence rate.

\paragraph{Stochastic  Proximal  Subgradient Method.}
We make the additional assumptions about the problem for developing SPG. 
\begin{ass}\label{ass:1}
Assume one of the following conditions hold: 
\begin{itemize}
\item[(i)] $g(\x)$ is $L$-smooth and there exists $G>0$ such that $\E[\|(\nabla g(\x; \xi) - \partial h(\x; \varsigma)) - \E[\nabla g(\x; \xi) - \partial h(\x; \varsigma)]\|^2]\leq G^2$, where $\partial h(\x)$ denotes a subgradient such that $ \E_{\varsigma}[\partial h(\x; \varsigma)] =  \partial h(\x)$. 
\item[(ii)]  there exists $G>0$ such that $\E[\|\partial g(\x; \xi)\|^2]\leq G^2$, $\E[\| \partial h(\x; \varsigma) \|^2\}]\leq G^2$ for $\x\in\text{dom}(r)$, and either $r = \delta_{\X}(\x)$ for a closed convex set $\X$ or $\|\partial r(\x)\|\leq G$ for $\x\in\text{dom}(r)$.
\end{itemize}
\end{ass}
{\bf Remark:} The first assumption is typically used in the analysis of stochastic gradient method when the involved function is smooth~\citep{DBLP:conf/icml/ZhaoZ15}, and the second assumption is typically used when the involved function is non-smooth~\citep{DBLP:conf/colt/DuchiSST10}. Note that the condition  $\partial r(\x)^{\top}(\x - \y)\geq 0, \forall \x, \y\in\text{dom}(r)$ is to capture the indicator function of a convex set. When $r(\x)$ is the indicator function of a convex set $\X$, we have $\text{dom}(r)= \X$  and $\partial r(\x)$ corresponds to the normal cone of $\X$, implying $\partial r(\x)^{\top}(\x - \y)\geq 0, \forall \x, \y\in\X$. 

Denote by $F^\gamma_{\x_1}(\x) = g(\x)  + r(\x) - \partial h(\x_1)^{\top}(\x - \x_1)  + \frac{\gamma}{2}\|\x - \x_1\|^2$. We present the SPG algorithm in Algorithm~\ref{alg:sgd} with two options to handle smooth and non-smooth $g$ separately. The constraint $\|\x - \x_1\|\leq 3G/\gamma$ at Step 5 is added to accommodate the proximal mapping of $r(\x)$ when $g(\x)$ is non-smooth. When using the subgradient of $r(\x)$ instead of the proximal mapping of $r(\x)$ in the update or  $r(\x)$ is the indicator function of a bounded convex set, the constraint $\|\x - \x_1\|\leq  3G/\gamma$ can be removed. 
\begin{algorithm}[t]
    \caption{{SPG}$(F^\gamma_{\x_1}, \x_1, T)$}\label{alg:sgd}
\begin{algorithmic}[1]
   \STATE Set step size $\eta_t$ according to Proposition~\ref{prop:sgd}, $\Omega=\{\x\in\text{dom}(r): \|\x - \x_1\|\leq 3G/\gamma\}$
    \FOR{$t=1,\ldots, T$}
    \STATE Compute  stochastic subgradients $\partial g(\x_t; \xi_t)$ and $\partial h(\x_1; \varsigma_t)$
        \STATE Option 1: $$\x_{t+1} =\arg\min_{\x} \x^{\top}(\partial g(\x_t; \xi_t) - \partial h(\x_1; \varsigma_t)) +r(\x) + \frac{\gamma}{2}\|\x - \x_1\|^2 + \frac{1}{2\eta_t}\|\x- \x_{t}\|^2$$
            \STATE Option 2: $$\x_{t+1} =\arg\min_{\x\in\Omega} \x^{\top}(\partial g(\x_t; \xi_t) - \partial h(\x_1; \varsigma_t) )  + r(\x) + \frac{\gamma}{2}\|\x - \x_1\|^2+ \frac{1}{2\eta_t}\|\x- \x_{t}\|^2$$
    \ENDFOR
    \STATE \textbf{Output}: $\widehat{\x}_T = \sum_{t=2}^{T+1} t\x_t/\sum_{t=2}^{T+1}t$  (Option 1) or  $\widehat{\x}_T = \sum_{t=1}^{T} t\x_t/\sum_{t=1}^{T}t$ (Option 2)

\end{algorithmic}
\end{algorithm}

\begin{prop}\label{prop:sgd} Suppose Assumption~\ref{ass:1}(i) hold, then by setting $\eta_t = 3/(\gamma(t+1))$ and $\gamma\geq 3L$,  Algorithm~\ref{alg:sgd} with Option 1 guarantees that 
\begin{align*}
\E[F^\gamma_{\x_1}(\xh_T) - F^\gamma_{\x_1}(\x_*)]\leq\frac{4\gamma\|\x_* - \x_1\|^2}{3T(T+3)} + \frac{6G^2}{(T+3)\gamma}.
\end{align*}
Suppose Assumption~\ref{ass:1}(ii) hold, then by setting $\eta_t = 4/(\gamma t)$,   Algorithm~\ref{alg:sgd} with Option 2 guarantees that 
\begin{align*}
\E\bigg[F^\gamma_{\x_1}(\xh_T) - F^\gamma_{\x_1}(\x_*)\bigg]\leq  \frac{\gamma\|\x_* - \x_1\|^2}{4T(T+1)} +  \frac{28G^2}{\gamma (T+1)},
\end{align*}
where $\x_* = \arg\min_{\x}F^\gamma_{\x_1}(\x)$. 
\end{prop}
We present a proof the above Proposition in the Appendix, which mostly follows existing analysis of SPG or related algorithms. By applying the above results (e.g., the second result in Proposition~\ref{prop:sgd}) to the $k$-th stage, we have
\[
\E\bigg[F_k(\x_{k+1}) - F_k(\z_k)\bigg]\leq  \frac{\gamma\|\z_k - \x_k\|^2}{4T_k(T_k+1)} +  \frac{28G^2}{\gamma (T_k+1)},
\]
where $T_k$ denotes the number of iterations used by SPG for the $k$-th stage.   One might directly use the above result to argue that  the condition~(\ref{eqn:acc}) holds by assuming that $\|\x_k - \z_k\|$ is bounded, which is true in the non-smooth case due to the domain constraint $\x\in\Omega$ in the update. In the smooth case, the upper bound is not directly available for setting $T_k$ such that  the condition~(\ref{eqn:acc}) holds. Fortunately, when we apply the above result in the convergence analysis of Algorithm~\ref{alg:meta}, we can utilize the strong convexity of $F_k$ to cancel the term $O(\frac{\gamma\|\z_k - \x_k\|^2}{T_k(T_k+1)})$ by setting $T_k$ to be larger than a constant.

Let us summarize the convergence of Algorithm~\ref{alg:meta} when using SPG to  solve each subproblem. 
\begin{thm}\label{thm:sgd}
Suppose Assumption~\ref{ass:1} (i) holds and Algorithm~\ref{alg:sgd} is employed for solving $F_k$ with parameters given in Proposition~\ref{prop:sgd} and with $\gamma\geq 3L$ and  $T_k = 3Lk/\gamma+3$, and there exists $\Delta>0$ such that $\E[F(\x_k) - \min_{\x}F(\x)]\leq \Delta$ for all $k\in\{1,\ldots, K\}$, then with a total of $K$ stages Algorithm~\ref{alg:meta} guarantees 
\begin{align*}
 \E\bigg[\|G_\gamma(\x_\tau)\|^2\bigg]&\leq \frac{8\gamma\Delta(\alpha+1)}{K} + \frac{32G^2\gamma(\alpha+1)}{LK}.
\end{align*}
Similarly, Suppose Assumption~\ref{ass:1} (ii) holds and Algorithm~\ref{alg:sgd} is employed for solving $F_k$ with parameters given in Proposition~\ref{prop:sgd} and with $T_k = k/\gamma +1$, and there exists $\Delta>0$ such that $\E[F(\x_k) - \min_{\x}F(\x)]\leq \Delta$ for all $k\in\{1,\ldots, K\}$, then with a total of $K$ stages Algorithm~\ref{alg:meta} guarantees 
\begin{align*}
 \E\bigg[\|G_\gamma(\x_{\tau})\|^2\bigg]&\leq \frac{8\gamma\Delta(\alpha+1)}{K} + \frac{448G^2 \gamma (\alpha+1)}{K},
\end{align*}
where 
$\tau\in\{1,\ldots, K\}$ is sampled according to probabilities $p(\tau=k) = \frac{k^\alpha}{\sum_{k=1}^Kk^\alpha}$ with $\alpha\geq 1$. 
\end{thm}
{\bf Remark:} Let us consider the iteration complexity of using SPG for finding a solution that satisfies $\E\bigg[\|G_\gamma(\x_{\tau})\|^2\bigg]\leq \epsilon^2$. In both cases, we need a total number of stages $K = O(\gamma/\epsilon^2)$ and total iteration complexity  $\sum_{k=1}^KT_k = \sum_{k=1}^KO(k/\gamma +1 )= O(\gamma/\epsilon^4)$. 
One can also derive similar results for using the uniform sampling under Assumption~\ref{ass:0}, which are  worse by a logarithmic factor. 

\paragraph{AdaGrad.}
AdaGrad~\citep{duchi2011adaptive} is an important algorithm in the literature of stochastic optimization, which uses adaptive step size for each coordinate. It has potential benefit of speeding up the convergence when the cumulative growth of stochastic gradient is slow. Next, we show that AdaGrad can be leveraged to solve each convex majorant function and yield adaptive convergence for the original problem. Similar to previous analysis of AdaGrad~\citep{duchi2011adaptive,SadaGrad18}, we make the following assumption. 
\begin{ass}\label{ass:new}
      For any $\x\in\text{dom}(r)$,  there exists $G>0$ such that $\|\partial g(\x; \xi)\|_\infty\leq G$ and $\|\partial h(\x; \varsigma)\|_\infty\leq G$, either $\partial r(\x)^{\top}(\x - \y)\geq 0, \forall \x, \y\in\text{dom}(r)$ or $\|\partial r(\x)\|\leq G_r$ for $\x\in\text{dom}(\x)$.
\end{ass}
\begin{algorithm}[t]
    \caption{\textsc{AdaGrad}($F^\gamma_{\x_1}, \x_1, \eta$)} \label{alg:adagrad}
    \begin{algorithmic}[1]
    \STATE \textbf{Initialize:} $t=1$, $\g_{1:0}=[]$, $H_0\in\R^{d\times d}$, $\Omega=\{\x\in\text{dom}(r): \|\x -\x_1\|\leq \frac{2\sqrt{d}G + G_r}{\gamma}\}$
    \WHILE{{$T$ does not satisfy the condition in Proposition~\ref{lem:adagrad}}}
    \STATE Compute a stochastic subgradient $\g_t $ for $g(\x_t) - \partial h(\x_1)^{\top}\x_t$
    \STATE Update $g_{1:t} = [g_{1:t-1}, \g_t]$, $s_{t,i}  = \|g_{1:t,i}\|_2$
    \STATE Set $H_t = H_0 + \diag(\s_t)$ and $\psi_t(\x) = \frac{1}{2}(\x-\x_1)^{\top}H_t(\x-\x_1)$
    \STATE Let $\x_{t+1}=\arg\min\limits_{\x\in\Omega} \x^{\top}\left(\frac{1}{t}\sum_{\tau=1}^t\g_\tau\right)  + r(\x)+ \frac{\gamma}{2}\|\x - \x_1\|^2 + \frac{1}{t\eta}\psi_t(\x)$
    \ENDWHILE
    \STATE \textbf{Output}:  $\xh_T=\sum_{t=1}^{T}\x_t/T$
    \end{algorithmic}
    \end{algorithm}
The convergence analysis of using AdaGrad is build on the following proposition about the convergence AdaGrad for minimizing $F^\gamma_{\x}$, which is attributed to~\cite{SadaGrad18}.
\begin{prop}\label{lem:adagrad}
  Let $H_0=2G I$ with $2G\geq\max_t \|\g_t\|_\infty$,  and iteration number $T$ be the smallest integer satisfying  $T\geq M\max\{a(2G+ \max_i\|g_{1:T,i}\|),  \sum_{i=1}^d\|g_{1:T,i}\|/a, G_r\|\x_1 - \x_{T+1}\|/\eta\}$ for any $a>0$. 
    Algorithm~\ref{alg:adagrad} returns an averaged solution $\xh_T$ such that
    \begin{align}\label{eqn:Es}
        \E[ F^\gamma_{\x_1}(\xh_T)-F^\gamma_{\x_1}(\x_*)] \leq \frac{1}{2aM\eta}\|\x_1-\x_*\|^2 +\frac{(a+1)\eta}{M},
    \end{align}
    where $\x_*=\arg\min_{\x}F^\gamma_{\x_1}(\x)$, $g_{1:t}=(\g_1,\ldots, \g_t)$ and $g_{1:t,i}$ denotes the $i$-th row of $g_{1:t}$.
\end{prop}


The convergence result of Algorithm~\ref{alg:meta} by using AdaGrad to solve each problem is described by following theorem.
\begin{thm} \label{thm:nadagrad}
    Suppose  Assumption~\ref{ass:new} hold and Algorithm~\ref{alg:adagrad} is employed for solving $F_k$ with  $\eta_{k} = c/ \sqrt{k}$, $T_k=\lceil M_k \max\{a(2G+ \max_i\|g^k_{1:T_k,i}\|),  \sum_{i=1}^d\|g^k_{1:T_k,i}\|/a, G_r\|\x_1^k - \x_{T_k+1}^k\|/\eta_k\}\rceil$ and $M_k\eta_k \geq  4/(a\gamma) $,  and there exists $\Delta>0$ such that $\E[F(\x_k) - \min_{\x}F(\x)]\leq \Delta$ for all $k\in\{1,\ldots, K\}$,  then with a total of $K$ stages Algorithm~\ref{alg:meta} guarantees  
\begin{align*}
    \E[\|G_\gamma(\x_\tau)\|^2]     \leq& \frac{8\gamma\Delta(\alpha+1)}{K} + \frac{4\gamma^2c^2a(a+1)(\alpha+1)}{K},
\end{align*}
    where $g^k_{1:t, i}$ denotes the cumulative stochastic gradient of the $i$-th coordinate at the $k$-th stage, and $\tau\in\{1,\ldots, K\}$ is sampled according to probabilities $p(\tau=k) = \frac{k^\alpha}{\sum_{k=1}^Kk^\alpha}$ with $\alpha\geq 1$. 
\end{thm}
{\bf Remark:} It is obvious that the total number of iterations $\sum_{k=1}^KT_k$ is adaptive to the data. Next, let us present more discussion on the iteration complexity. Note that $M_k = O(\sqrt{k})$. By the boundness of stochastic gradient $\|g^k_{1:T_k, i}\|\leq O(\sqrt{T_k})$, therefore $T_k$ in the order of $O(k)$ will satisfy the condition in Theorem~\ref{thm:nadagrad}. Thus in the worst case, the iteration complexity for finding $ \E[\|G_\gamma(\x_\tau)\|^2]\leq \epsilon^2$ is in the order of $\sum_{k=1}^K O(k)\leq O(1/\epsilon^4)$. We can show the potential advantage of adaptiveness similar to that in~\citep{chen18stagewise}. In particular, let us consider $r=\delta_{\X}$  and thus $G_r=0$ in the above result.  When the cumulative growth of stochastic gradient is slow, e.g., assuming $\|g^k_{1:T_k,i}\|\leq O({T_k}^{\beta})$ with $\beta<1/2$. Then $T_k = O(k^{1/(2(1-\beta))})$ will work, and then the total number of iterations $\sum_{k=1}^K T_k \leq K^{1+1/(2(1-\alpha))}\leq O(1/\epsilon^{2+1/(1-\alpha)})$,  which is better than $O(1/\epsilon^4)$. 

\paragraph{SVRG.} Next, we present SVRG for solving each subproblem when it has a  finite-sum form and $g$ is a smooth function. In particular, we consider the following problem:
\begin{align}\label{eqn:fs}
\min_{\x\in\R^d}  F(\x):=\frac{1}{n_1}\sum_{i=1}^{n_1} g_i(\x) + r(\x) - \frac{1}{n_2}\sum_{j=1}^{n_2}h_j(\x),
\end{align}
and make the following assumption. 
\begin{ass}\label{ass:4}
      Assume $g_i(\x)$ is $L$-smooth function.  
\end{ass}
It is notable that the smoothness of $h$ is not necessary for developing the SVRG algorithm since at each stage we linearize $h(\x)$. Each subproblem is of the following form: 
\begin{align*}
\min_{\x\in\R^d}  F^\gamma_{\x_1}(\x):=\frac{1}{n_1}\sum_{i=1}^{n_1} g_i(\x) + r(\x)  - h(\x_1)- \frac{1}{n_2}\sum_{j=1}^{n_2}\partial h_j(\x_1)^{\top}(\x - \x_1) + \frac{\gamma}{2}\|\x - \x_1\|^2.
\end{align*}

\begin{algorithm}[t]
\caption{SVRG($F^\gamma_{\x_1}, \x_1$, $T$, $S$)}\label{alg:SVRG}
\begin{algorithmic}[1]
\STATE \textbf{Input}:  $\x_1\in\text{dom}(r)$,  the number of inner initial iterations $T_1$, and the number of outer loops $S$.
\STATE $\bar\x^{(0)}=\x_1$
\FOR{$s=1,2,\ldots,S$}
\STATE $\bar\g_s=\nabla g(\bar\x^{(s-1)}) - \partial h(\x_1)$, $\x_0^{(s)} =\bar\x^{(s-1)}$
\FOR{$t=1,2,\ldots,T$}
\STATE Choose $i_t\in\{1,\ldots,n_1\}$ uniformly at random.
\STATE $\nabla_t^{(s)}=\nabla g_{i_t}(\x_{t-1}^{(s)})-\nabla g_{i_t}(\bar\x^{(s-1)})+\bar\g_s.$
\STATE
$\x_t^{(s)} = \arg\min_{\x} \langle \nabla_t^{(s)}, \x-\x_{t-1}^{(s)}\rangle+\frac{1}{2\eta}\|\x-\x_{t-1}^{(r)}\|_2^2 +r(\x) + \frac{\gamma}{2}\|\x - \x_1\|^2.$
\ENDFOR
\STATE $\bar\x^{(s)}=\frac{1}{T}\sum_{t=1}^{T}\x_{t}^{(s)}$
\ENDFOR
\STATE  \textbf{Output: } $\bar\x^{(S)}$
\end{algorithmic}
\end{algorithm}
We use the proximal SVRG proposed in~\citep{DBLP:journals/siamjo/Xiao014} to solve the above problem, which is presented in Algorithm~\ref{alg:SVRG}. Its convergence result for solving $F^\gamma_{\x_1}$ is given below. 
\begin{prop}\label{prop:svrg}
By setting $\eta<1/(4L)$ and $T$ is large enough such that $\rho = \frac{1}{\gamma\eta (1 - 4L\eta)T} + \frac{4L\eta(T+1)}{(1-4L\eta)T}<1$, then 
\begin{align*}
\E[F^\gamma_{\x_1}(\bar\x^{(S)}) - F^\gamma_{\x_1}(\x_*)]\leq \rho^S[F^\gamma_{\x_1}(\x_1) - F^\gamma_{\x_1}(\x_*)].
\end{align*}
In particular, if we set $\eta=0.05/L$,  $T \geq \max(2, 200L/\gamma)$, we have
\begin{align*}
\E[F^\gamma_{\x_1}(\bar\x^{(S)}) - F^\gamma_{\x_1}(\x_*)]\leq 0.5^S[F^\gamma_{\x_1}(\x_1) - F^\gamma_{\x_1}(\x_*)].
\end{align*}
\end{prop}
{\bf Remark:} The gradient complexity of SVRG is $(n + T)S$, where $n = n_1 + n_2$. 

\begin{thm} \label{thm:svrg}
    Suppose  Assumption~\ref{ass:0} and Assumption~\ref{ass:4} hold,  and Algorithm~\ref{alg:SVRG} is employed for solving $F_k$ with  $\eta_{k} = 0.05/L$, $T_k\geq \max(2, 200L/\gamma)$, $S_k =\lceil \log_2(k)\rceil$, then with a total of $K$ stages Algorithm~\ref{alg:meta} guarantees  
\begin{align*}
    \E[\|G_\gamma(\x_\tau)\|^2]     \leq& \frac{12\gamma\Delta(\alpha+1)}{K},
\end{align*}
    where $\tau\in\{1,\ldots, K\}$ is sampled according to probabilities $p(\tau=k) = \frac{k^\alpha}{\sum_{k=1}^Kk^\alpha}$ with $\alpha\geq 1$. 
\end{thm}
{\bf Remark:} For finding a solution such that $\E[\|G_\gamma(\x_\tau)\|^2]  $, the total number of stages $K= O(\gamma/\epsilon^2)$ and the total gradient complexity is $\widetilde O((n\gamma + L)/\epsilon^2)$. 

\subsection{Convergence results for finding a (nearly) $\epsilon$-critical point}\label{sec:stationary}
In this subsection, we will discuss the convergence results of the proposed algorithms for finding a nearly $\epsilon$-critical point. The key is to connect the convergence in terms of $\|G_\gamma(\x)\|$ to the convergence in terms of (sub)gradient. To this end, we present the following result. 
\begin{prop}\label{prop:5}
If  $g(\x)+r(\x)$ is differentiable and has $L_{g+r}$-H\"{o}lder continuous gradient, we have
\begin{align*}
\text{dist}(\partial h(\x), \nabla  g(\x) + \nabla r(\x)) &\leq \frac{L_{g+r}}{\gamma^\nu}\|G_\gamma(\x)\|^\nu +  \|G_\gamma(\x)\|.
\end{align*}
If $h(\x)$ is differentiable and has $L_h$-H\"{o}lder continuous gradient, we have
\begin{align*}
\text{dist}(\nabla h(\x_+), \partial  g(\x_+) + \partial r(\x_+)) &\leq \frac{L_h}{\gamma^\nu}\|G_\gamma(\x)\|^\nu +  \|G_\gamma(\x)\|.
\end{align*}
where $\x_+ = P_\gamma(\x)$ and $G_\gamma(\x) = \gamma(\x - \x_+)$. 
\end{prop}
\begin{proof}From the proof of Proposition~\ref{prop:1}, we have
\begin{align*}
0\in \partial  g(\x_+)   + \partial r(\x_+)-\partial h(\x) + \gamma(\x_+ - \x),
\end{align*}
When $g(\x)+r(\x)$ is differentiable and has $L$-H\"{o}lder continuous gradient, there exists $\v\in\partial h(\x)$ such that 
\begin{align*}
\|\nabla  g(\x)  + \nabla r(\x)  - \v\| &=  \|\nabla g(\x_+) +\nabla r(\x_+) - \nabla g(\x) - \nabla r(\x)\| +  \gamma\|\x - \x^+\|\\
&\leq L_{g+r}\|\x - \x_+\|^\nu + \|G_\gamma(\x)\| =  \frac{L_{g+r}}{\gamma^\nu}\|G_\gamma(\x)\|^\nu +  \|G_\gamma(\x)\|.
   \end{align*}
Similarly, we can prove the case when $h$ is differentiable and had H\"{o}lder continuous gradient. 
\end{proof}
Next, we present a general result that can be used to derive the convergence rate for the proposed algorithms. 
\begin{thm}\label{thm:6}
Assume Algorithm~\ref{alg:meta}  returns  a solution $\x_\tau$ such that 
\begin{align*}
    \E[\|G_\gamma(\x_\tau)\|^2]     \leq& O\left(\frac{1}{K}\right)
\end{align*}
under appropriate conditions. Then if  $g(\x)+r(\x)$ is differentiable and has $L$-H\"{o}lder continuous gradient with $\nu\in(0,1]$, we have
\begin{align*}
\E[\text{dist}(\partial h(\x_\tau), \nabla  g(\x_\tau) + \nabla r(\x_\tau))] &\leq O\left(\frac{1}{K^{\nu/2}} + \frac{1}{K^{1/2}}\right).
\end{align*}
If $h(\x)$ is differentiable and has $L$-H\"{o}lder continuous gradient with $\nu\in(0,1]$, we have
\begin{align*}
\E[\|\x_\tau - \z_\tau\|]\leq O\left(\frac{1}{K^{1/2}}\right), \quad \E[\text{dist}(\nabla h(\z_\tau), \partial  g(\z_\tau) + \partial r(\z_\tau))] &\leq O\left(\frac{1}{K^{\nu/2}} + \frac{1}{K^{1/2}}\right).
\end{align*}
where $\z_\tau = P_\gamma(\x_\tau)$.
\end{thm}
{\bf Remark:} In the above results, the value of $\gamma$ is set to be a constant. We can see that the convergence of  SPG and AdaGrad can be automatically  adaptive to the H\"{o}lder continuous of the involved functions without requiring the value of $L$ and $\nu$ for running the algorithm. Both algorithms have an iteration complexity (in the worst-case) of $O(1/\epsilon^{4/\nu})$ for finding a (nearly) $\epsilon$-critical point. When the problem has a finite-sum structure~(\ref{eqn:fs}) and $g(\x)$ is smooth, SVRG has a gradient complexity of $O(n/\epsilon^{2/\nu})$ for finding a (nearly) $\epsilon$-critical point. 

\subsection{Improved Convergence results when $\nu$ is known}\label{sec:DC:new:improve}
In this subsection, we present some improved convergence results when $\nu$ is known. The key idea is by using an increasing sequence of regularization parameters $\gamma$ that depend on $\nu$. 
\begin{lemma}\label{lem:DC:new1}
Suppose Assumption~\ref{ass:1} (i) holds and Algorithm~\ref{alg:sgd} is employed for solving $F_k$ with parameters given in Proposition~\ref{prop:sgd} and with $\gamma_k= 3Lk^{\frac{1-\nu}{ 1+ \nu}}$ and  $T_k = 3Lk/\gamma_k+3$, and there exists $\Delta>0$ such that $\E[F(\x_k) - \min_{\x}F(\x)]\leq \Delta$ for all $k\in\{1,\ldots, K\}$, then with a total of $K$ stages Algorithm~\ref{alg:meta} guarantees 
\begin{align*}
\E\bigg[\|\z_{\tau} - \x_\tau\|^2\bigg]&\leq \frac{8 \Delta (\alpha +1)}{3LK^{\frac{2}{1+\nu}}} + \frac{32G^2(\alpha +1) }{3K^{\frac{2}{1+\nu}}L^2},
\end{align*}
    where $\tau\in\{1,\ldots, K\}$ is sampled according to probabilities $p(\tau=k) = \frac{k^\alpha}{\sum_{k=1}^Kk^\alpha}$ with $\alpha\geq 1$. 
Similarly, Suppose Assumption~\ref{ass:1} (ii) holds and Algorithm~\ref{alg:sgd} is employed for solving $F_k$ with parameters given in Proposition~\ref{prop:sgd} and with $\gamma_k= ck^{\frac{1-\nu}{ 1+ \nu}}$ and  $T_k = k/\gamma_k +1$, and there exists $\Delta>0$ such that $\E[F(\x_k) - \min_{\x}F(\x)]\leq \Delta$ for all $k\in\{1,\ldots, K\}$, then with a total of $K$ stages Algorithm~\ref{alg:meta} guarantees 
\begin{align*}
 \E\bigg[\|\z_{\tau} - \x_\tau\|^2\bigg]&\leq \frac{8\Delta(\alpha+1)}{cK^{\frac{2}{1+\nu}}} + \frac{448G^2 (\alpha+1)}{cK^{\frac{2}{1+\nu}}},
\end{align*}
where 
$\tau\in\{1,\ldots, K\}$ is sampled according to probabilities $p(\tau=k) = \frac{k^\alpha}{\sum_{k=1}^Kk^\alpha}$ with $\alpha\geq 1$. 
\end{lemma}
\begin{cor}\label{cor:7}
Suppose  the same conditions as in Lemma~\ref{lem:DC:new1} hold.  If  $g(\x)+r(\x)$ is differentiable and has $L$-H\"{o}lder continuous gradient with $\nu\in(0,1]$, by setting $K=O(1/\epsilon^{\frac{1+\nu}{\nu}})$ we have
\begin{align*}
\E[\text{dist}(\partial h(\x_\tau), \nabla  g(\x_\tau) + \nabla r(\x_\tau))] &\leq O(\epsilon),
\end{align*}
and the total iteration complexity is $\sum_{k=1}^KT_k =O(1/\epsilon^{\frac{1+3\nu}{\nu}})$. 

If $h(\x)$ is differentiable and has $L$-H\"{o}lder continuous gradient with $\nu\in(0,1]$, by setting $K=O(1/\epsilon^{\frac{1+\nu}{\nu}})$ we have
\begin{align*}
\E[\|\x_\tau - \z_\tau\|]\leq O\left(\epsilon^{\frac{1}{\nu}}\right), \quad \E[\text{dist}(\nabla h(\z_\tau), \partial  g(\z_\tau) + \partial r(\z_\tau))] &\leq O(\epsilon), 
\end{align*}
where $\z_\tau = P_\gamma(\x_\tau)$. The total iteration complexity is $\sum_{k=1}^KT_k =O(1/\epsilon^{\frac{1+3\nu}{\nu}})$. 
\end{cor}
{\bf Remark:} The iteration complexity $O(1/\epsilon^{\frac{1+3\nu}{\nu}})$ in the above Corollary is an improved one compared with $O(1/\epsilon^{4/\nu})$ derived from Theorem~\ref{thm:6} for SPG.
\begin{proof}
By the analysis of Proposition~\ref{prop:5}, if $g(\x)+r(\x)$ is differentiable and has $L$-H\"{o}lder continuous gradient with $\nu\in(0,1]$,  we have
\begin{align*}
\E[\text{dist}(\partial h(\x_\tau), \nabla  g(\x_\tau) + \nabla r(\x_\tau))] &\leq  \E[L_{g+r}\|\z_\tau - \x_\tau\|^\nu +  \gamma_\tau \|\z_\tau - \x_\tau\| ]\\
&\leq L_{g+r}\E[\|\z_\tau - \x_\tau\|^\nu] +  \gamma_K \E[\|\z_\tau - \x_\tau\|]\\
&\leq L_{g+r}(\E[\|\z_\tau - \x_\tau\|^2])^{\nu/2}+  \gamma_K (\E[\|\z_\tau - \x_\tau\|^2])^{1/2},
\end{align*}
where we use the concavity of $x^{s}$ for $s<1$ and $x\geq 0$.  By setting $K=O(1/\epsilon^{\frac{1+\nu}{\nu}})$ and $\gamma_K = O(K^{\frac{1-\nu}{ 1+ \nu}})$, and using the results in Lemma~\ref{lem:DC:new1}, we have
\begin{align*}
\E[\text{dist}(\partial h(\x_\tau), \nabla  g(\x_\tau) + \nabla r(\x_\tau))] &\leq  O(\epsilon).
\end{align*}
Similarly, if $h(\x)$ is differentiable and has $L$-H\"{o}lder continuous gradient with $\nu\in(0,1]$, by the analysis of Proposition~\ref{prop:5} we have
\begin{align*}
\E[\text{dist}(\nabla h(\z_\tau), \partial  g(\z_\tau) + \partial r(\z_\tau))] &\leq  \E[L_{h}\|\z_\tau - \x_\tau\|^\nu +  \gamma_\tau \|\z_\tau - \x_\tau\| ]\\
&\leq L_{h}\E[\|\z_\tau - \x_\tau\|^\nu] +  \gamma_K \E[\|\z_\tau - \x_\tau\|].
\end{align*}
We can finish the proof using similar analysis. 
\end{proof}

Next, we develop the results for AdaGrad by using the following lemma.
\begin{lemma}\label{lem:DC:new2}
    Suppose  Assumption~\ref{ass:new} hold and Algorithm~\ref{alg:adagrad} is employed for solving $F_k$ with $\gamma_k= k^{\frac{1-\nu}{ 1+ \nu}}$,  $\eta_{k} = c/ \sqrt{\gamma_k k}$, $T_k=\lceil M_k \max\{a(2G+ \max_i\|g^k_{1:T_k,i}\|),  \sum_{i=1}^d\|g^k_{1:T_k,i}\|/a, G_r\|\x_1^k - \x_{T_k+1}^k\|/\eta_k\}\rceil$ and $M_k\eta_k \geq  4/(a\gamma_k) $ with $c$ and $a$ being  constants,  and there exists $\Delta>0$ such that $\E[F(\x_k) - \min_{\x}F(\x)]\leq \Delta$ for all $k\in\{1,\ldots, K\}$,  then with a total of $K$ stages Algorithm~\ref{alg:meta} guarantees  
\begin{align*}
     \E\bigg[\|\z_{\tau} - \x_\tau\|^2\bigg]&\leq \frac{8\Delta (\alpha +1)}{K^{\frac{2}{ 1+ \nu}}} + \frac{4a(a+1)c^2(\alpha +1) }{K^{\frac{2}{ 1+ \nu}}}.
\end{align*}
    where $g^k_{1:t, i}$ denotes the cumulative stochastic gradient of the $i$-th coordinate at the $k$-th stage, and $\tau\in\{1,\ldots, K\}$ is sampled according to probabilities $p(\tau=k) = \frac{k^\alpha}{\sum_{k=1}^Kk^\alpha}$ with $\alpha\geq 1$. 
\end{lemma}
\begin{cor}\label{cor:8}
Suppose  the same conditions as in Lemma~\ref{lem:DC:new2} hold.  If  $g(\x)+r(\x)$ is differentiable and has $L$-H\"{o}lder continuous gradient with $\nu\in(0,1]$, by setting $K=O(1/\epsilon^{\frac{1+\nu}{\nu}})$ we have
\begin{align*}
\E[\text{dist}(\partial h(\x_\tau), \nabla  g(\x_\tau) + \nabla r(\x_\tau))] &\leq O(\epsilon),
\end{align*}
and the total iteration complexity is $\sum_{k=1}^KT_k =O(1/\epsilon^{\frac{1+3\nu}{\nu}})$. 

If $h(\x)$ is differentiable and has $L$-H\"{o}lder continuous gradient with $\nu\in(0,1]$, by setting $K=O(1/\epsilon^{\frac{1+\nu}{\nu}})$ we have
\begin{align*}
\E[\|\x_\tau - \z_\tau\|]\leq O\left(\epsilon^{\frac{1}{\nu}}\right), \quad \E[\text{dist}(\nabla h(\z_\tau), \partial  g(\z_\tau) + \partial r(\z_\tau))] &\leq O(\epsilon), 
\end{align*}
where $\z_\tau = P_\gamma(\x_\tau)$. The total iteration complexity is $\sum_{k=1}^KT_k =O(1/\epsilon^{\frac{1+3\nu}{\nu}})$. 
\end{cor}
{\bf Remark:} The worst-case iteration complexity can be derived as following. According to the setting $M_k = O(\sqrt{k/\gamma_k})  = O(k^{\frac{v}{1+v}})$. Since $ \max\{a(2G+ \max_i\|g^k_{1:T_k,i}\|),  \sum_{i=1}^d\|g^k_{1:T_k,i}\|/a=O(\sqrt{T_k})$ and $ G_r\|\x_1^k - \x_{T_k+1}^k\|/\eta_k =O(\sqrt{k/\gamma_k})  $, we have that $T_k = O(k^{\frac{2v}{1+v}})$. Then $\sum_k T_k \leq O(1/\epsilon^{\frac{1+3\nu}{\nu}})$. Nevertheless, the convergence of SSDC-AdaGrad is adaptive to the data.

\begin{lemma}\label{lem:DC:new3}
Suppose  Assumption~\ref{ass:0} and Assumption~\ref{ass:4} hold,  and Algorithm~\ref{alg:SVRG} is employed for solving $F_k$ with $\gamma_k= ck^{\frac{1-\nu}{ 1+ \nu}}$ with a constant $c>0$,  $\eta_{k} = 0.05/L$, $T_k\geq \max(2, 200L/\gamma_k)$, $S_k =\lceil \log_2(k)\rceil$, then with a total of $K$ stages Algorithm~\ref{alg:meta} guarantees  
\begin{align*}
 \E\bigg[\|\z_{\tau} - \x_\tau\|^2\bigg]&\leq \frac{12\Delta (\alpha +1)}{cK^{\frac{2}{ 1+ \nu}}}.
 \end{align*}
where $\tau\in\{1,\ldots, K\}$ is sampled according to probabilities $p(\tau=k) = \frac{k^\alpha}{\sum_{k=1}^Kk^\alpha}$ with $\alpha\geq 1$.  
\end{lemma}

\begin{cor}\label{cor:9}
Suppose  the same conditions as in Lemma~\ref{lem:DC:new3} hold.  If  $g(\x)+r(\x)$ is differentiable and has $L$-H\"{o}lder continuous gradient with $\nu\in(0,1]$, by setting $K=O(1/\epsilon^{\frac{1+\nu}{\nu}})$ we have
\begin{align*}
\E[\text{dist}(\partial h(\x_\tau), \nabla  g(\x_\tau) + \nabla r(\x_\tau))] &\leq O(\epsilon),
\end{align*}
and the total gradient complexity is $O(n/\epsilon^{\frac{1+\nu}{\nu}})$. 

If $h(\x)$ is differentiable and has $L$-H\"{o}lder continuous gradient with $\nu\in(0,1]$, by setting $K=O(1/\epsilon^{\frac{1+\nu}{\nu}})$ we have
\begin{align*}
\E[\|\x_\tau - \z_\tau\|]\leq O\left(\epsilon^{\frac{1}{\nu}}\right), \quad \E[\text{dist}(\nabla h(\z_\tau), \partial  g(\z_\tau) + \partial r(\z_\tau))] &\leq O(\epsilon), 
\end{align*}
where $\z_\tau = P_\gamma(\x_\tau)$. The total gradient complexity is $O(n/\epsilon^{\frac{1+\nu}{\nu}})$. 
\end{cor}

The proofs of Corollary~\ref{cor:8} and \ref{cor:9}  are similar to that of Corollary~\ref{cor:7} and hence are omitted. The proofs of Lemma~\ref{lem:DC:new2} and Lemma~\ref{lem:DC:new3} are presented in the Appendix.

\section{Tackling Non-Smooth Non-Convex Regularization}\label{sec:nsncr}
In this section, we consider the more general class of problems where $r(\x)$ is  non-smooth and non-convex that is not necessarily a DC function (e.g., $\ell_0$ norm). Even if $r(\x)$ is a DC function such that both components in its DC decomposition $r(\x) = r_1(\x) - r_2(\x)$ are non-differentiable functions without H\"{o}lder continuous gradients (e.g., $\ell_{1-2}$ regularization, capped $\ell_1$ norm), the theories presented in this section are useful to derive non-asymptotic convergence results in terms of finding an $\epsilon$-critical point. Please note that in this case the results presented in section~\ref{sec:stationary} are not applicable. 
In particular, we consider the following class of problems:
\begin{align}\label{eqn:Pf}
\min_{\x\in\R^d} \underbrace{g(\x) - h(\x)}\limits_{f(\x)}+ r(\x),
\end{align}
where $g(\x)$ and $h(\x)$ are real-valued lower-semicontinuous convex functions, $g$ has a Lipchitz continuous gradient, and $r(\x)$ is a proper  non-smooth and non-convex lower-semicontinuous function.  Both $g$ and $h$ can be stochastic functions as in~(\ref{eqn:PS}).  We assume  $r(\x)$ is simple such that its proximal mapping exisits and  can be efficiently computed, i.e., 
\begin{align*}
\text{prox}_{\mu r}(\x)=\arg\min_{\y\in\R^d}\frac{1}{2\mu}\|\y - \x\|^2 + r(\y).
\end{align*}
The problem is challenging due to the presence of  non-smooth non-convex functions $r$. To tackle this function,  we introduce the Moreau envelope of $r$:
\begin{align*}
r_{\mu}(\x) = \min_{\y\in\R^d}\frac{1}{2\mu}\|\y - \x\|^2 + r(\y).
\end{align*}
where $\mu>0$. A nice property of the Moreau envelope function is that it can be written as a DC function: 
\begin{align*}
r_{\mu}(\x)= \frac{1}{2\mu}\|\x\|^2  - \underbrace{\max_{\y\in\R^d}\frac{1}{\mu}\y^{\top}\x - \frac{1}{2\mu}\|\y\|^2 - r(\y)}\limits_{R_\mu(\x)},
\end{align*}
where $R_\mu(\x)$ is a convex function because it is the max of convex functions of $\x$. We also have several nice properties about the Moreau envelope function that will be useful for our analysis.  
\begin{lemma}\label{lem:prox}
\begin{align}
&\frac{1}{\mu}\text{prox}_{\mu r}(\x)\subseteq \partial R_\mu(\x)\\
&\forall\v\in \text{prox}_{\mu r}(\x),~\frac{1}{\mu}(\x  - \v)\subseteq \hat \partial r(\v),
\end{align}
where $\hat\partial$ denotes the Fr\'{e}chet subdifferentiale. 
\end{lemma}
The proof of the first fact can be found in~\citep{Liu2018}[Eqn. 7], and the second fact follows~\citep{rockafellar-1970a}[Theorem 10.1]. Given the Moreau envelope of $r$, the key idea is to solve the following DC problem:
\begin{align}\label{eqn:New}
\min_{\x\in\R^d} g(\x) - h(\x)+  \frac{1}{2\mu}\|\x\|^2 - R_{\mu}(\x) = \underbrace{g(\x)  +  \frac{1}{2\mu}\|\x\|^2}\limits_{\hat g(\x)} - \underbrace{(h(\x) + R_{\mu}(\x))}\limits_{\hat h(\x)}.
\end{align}
By carefully controlling the value of $\mu$ and combining the results presented in previous section, we are able to derive non-asymptotic convergence results for the original problem. It is worth mentioning that using  the Moreau envelope of $r$ and its DC decomposition for handling non-smooth non-convex function is first proposed in~\citep{Liu2018}. However, their algorithms are deterministic and convergence results are only asymptotic. To formally state our non-asymptototic convergence results, we make the following assumptions. 
\begin{ass}\label{ass:last}
 Assume one of the following conditions holds: 
\begin{itemize}
\item[(i)]  $r$ is Lipchitz continuous.
\item[(ii)] $r$ is lower bounded and finite-valued over $\R^d$.
\item[(ii)] $f(\x) +  r_{\mu}(\x)$ is level bounded for a small $\mu<1$, and $r$ is finite-valued on a compact set, and lower bounded over $\R^d$. 
\end{itemize}
\end{ass}
{\bf Remark:} The above assumptions on $r$ capture many interesting non-convex non-smooth regularizers. For example,  $\ell_{1-2}$ regularization and capped $\ell_1$ norm satisfy Assumption~\ref{ass:last} (i). The $\ell_0$ norm satisfies Assumption~\ref{ass:last} (ii). A coersive function $r$ usually satisfies  Assumption~\ref{ass:last} (iii), e.g., $\ell_p$ norm $r(\x) = \sum_{i=1}^d|x|^p$ for $p\in(0,1)$. 


When employing the presented algorithms in last section to solve the problem~(\ref{eqn:New}), we let $r =   \frac{1}{2\mu}\|\x\|^2$, $g=g$ and $h= \hat h$. It is also notable that the new component $ R_{\mu}(\x)$ in $\hat h$ is deterministic, whose subgradient can be computed according to Lemma~\ref{lem:prox}. Thus the condition in Assumption~\ref{ass:1} (i) is sufficient for running SPG (option 1), and the condition in Assumption~\ref{ass:4} is sufficient for running SVRG.

\subsection{Basic Results with constant $\gamma$}
In this subsection, we present the basic results by using a constant regularization parameter $\gamma$. In next subsection, we present improved complexities by using an increasing sequence of regularization parameters. It will exhibit how an increasing sequence of $\gamma$ reduces the complexities.

We first present the results under the smoothness condition $h$, i.e., the following condition hold. 
\begin{ass}\label{ass:ghsm}
 Assume $h$ is $L$-smooth. 
\end{ass}
\begin{thm}\label{thm:ncns:reg}
Regarding the problem~(\ref{eqn:Pf}),  we have the following results:
\begin{itemize} 
\item [a.]
If Assumption~\ref{ass:ghsm}, Assumption~\ref{ass:last} (i) and Assumption~\ref{ass:1} (i) hold, then we can use Algorithm~\ref{alg:meta} with a constant $\gamma$ and with Algorithm~\ref{alg:sgd} (option 1)  to solve~(\ref{eqn:New}) with $\mu = \epsilon$, which returns a solution $\x_{\tau}$ after $K=O(1/\epsilon^4)$ stages satisfying 
\begin{align*}
\E[\|\x_\tau - \w_{\tau}\|]\leq O(\epsilon), \quad \E[\text{dist}(\nabla h(\w_\tau), \nabla g(\w_{\tau}) + \hat \partial r(\w_{\tau}))]\leq O(\epsilon),
\end{align*}
where $\w_\tau = \text{prox}_{\mu r}(\x_{\tau})$. 
\item [b.] If Assumption~\ref{ass:ghsm}, Assumption~\ref{ass:last} (ii) and Assumption~\ref{ass:1} (i) hold, then we can use Algorithm~\ref{alg:meta}  with a constant $\gamma$ and with Algorithm~\ref{alg:sgd} (option 1)  to solve~(\ref{eqn:New}) with $\mu = \epsilon^2$, which returns a solution $\x_{\tau}$ after $K=O(1/\epsilon^6)$ stages satisfying 
\begin{align*}
\E[\|\x_\tau - \w_{\tau}\|]\leq O(\epsilon), \quad \E[\text{dist}(\nabla h(\w_\tau), \nabla g(\w_{\tau}) + \hat \partial r(\w_{\tau}))]\leq O(\epsilon),
\end{align*}
where $\w_\tau = \text{prox}_{\mu r}(\x_{\tau})$. 
\item[c.]
If $g$ and $h$ have a finite-sum form, Assumption~\ref{ass:ghsm} and Assumption~\ref{ass:4} hold,  then we can use Algorithm~\ref{alg:meta}  with a constant $\gamma$ and with Algorithm~\ref{alg:SVRG}  to solve~(\ref{eqn:New}). We can  set  $\mu = \epsilon$ if Assumption~\ref{ass:last} (i) holds or $\mu=\epsilon^2$ if Assumption~\ref{ass:last} (ii) or (iii) holds.  The algorithm will return a solution $\x_{\tau}$ after $K=O(1/\epsilon^4)$ (corresponding to Assumption~\ref{ass:last} (i)) or $K=O(1/\epsilon^6)$ (corresponding to Assumption~\ref{ass:last} (ii) or (iii)) stages satisfying 
\begin{align*}
\E[\|\x_\tau - \w_{\tau}\|]\leq O(\epsilon), \quad \E[\text{dist}(\nabla h(\w_\tau), \nabla g(\w_{\tau}) + \hat \partial r(\w_{\tau}))]\leq O(\epsilon),
\end{align*}
where $\w_\tau = \text{prox}_{\mu r}(\x_{\tau})$. 
\end{itemize}
\end{thm}
{\bf Remark}: The above result establishes that $\w_\tau$ is an  $\epsilon$-critical point. Under Assumption~\ref{ass:last} (i),  the iteration complexity of using SPG is $O(1/\epsilon^8)$ and the gradient complexity of using SVRG is $\widetilde O(n/\epsilon^4)$; and under Assumption~\ref{ass:last} (ii),  the iteration complexity of using SPG  is $O(1/\epsilon^{12})$ and the gradient complexity of using SVRG is $\widetilde O(n/\epsilon^6)$. Under Assumption~\ref{ass:last} (iii), the only provable interesting result is by using SVRG, which gives  a gradient complexity of $\widetilde O(n/\epsilon^6)$. 


\begin{proof} In the following proof, $\hat\partial r(\x)$ means there exists $\v\in\hat\partial r(\x)$. 
By applying the stochastic algorithms for DC functions in last section, at each stage  the following problem is solved approximately 
\begin{align*}
\z_k = \arg\min_{\x\in\R^d} \hat g(\x) + \frac{\gamma}{2}\|\x - \x_k\|^2 - (\nabla  h(\x_k) + \frac{1}{\mu}\text{prox}_{\mu r}(\x_k))^{\top}(\x - \x_k).
\end{align*}
Then we have 
\begin{align*}
\E[\|\nabla  \hat g(\x_\tau)  -\nabla h(\x_\tau) - \frac{1}{\mu}\text{prox}_{\mu r}(\x_\tau)\|] &\leq \E[ \|\nabla \hat g(\z_\tau) - \nabla\hat g(\x_\tau)\| +  \gamma\|\z_\tau - \x_\tau\|]\\
&\leq \E[L\|\x_\tau - \z_\tau\| +  \frac{1}{\mu}\|\x_\tau - \z_\tau\|+ \gamma \|\z_\tau - \x_\tau\|],
   \end{align*}
   for any $\tau\in\{1,\ldots, K\}$. Denote by $\w_\tau =  \text{prox}_{\mu r}(\x_\tau)$. 
   It is notable that $\frac{1}{\mu}\left(\x_\tau - \text{prox}_{\mu r}(\x_\tau)\right)\in\hat \partial r(\w_\tau) $. Then we have 
   $\nabla  \hat g(\x_\tau)  -\nabla h(\x_\tau) - \frac{1}{\mu}\text{prox}_{\mu r}(\x_\tau)= \nabla g(\x_\tau) - \nabla h(\x_\tau) + \hat\partial r(\w_\tau)$ and 
\begin{align*}
\E[\|\nabla  g(\x_\tau) - \nabla h(\x_\tau)  + \hat\partial r(\w_\tau)\|] &\leq \E[L\|\x_\tau - \z_\tau\| +  \frac{1}{\mu}\|\x_\tau - \z_\tau\|+ \gamma \|\z_\tau - \x_\tau\|],
   \end{align*}
   which implies that 
\begin{align*}
\E[\|\nabla  g(\w_\tau) - \nabla h(\w_\tau)  + \hat\partial r(\w_\tau)\|] &\leq \E[(L+\gamma)\|\x_\tau - \z_\tau\| +  \frac{1}{\mu}\|\x_\tau - \z_\tau\| + 2L\|\x_\tau - \w_\tau\|],
   \end{align*}
   where uses the facts that $g$ and $h$ are smooth.
   
   Next, we need to show that $\E[\|\x_\tau - \w_\tau\|]$ is small. The argument will be different for part (a), part (b) and part (c). For part (a), using 
      \begin{align*}
 \frac{1}{2\mu}\|\x_\tau - \w_\tau\|^2 +  r(\w_\tau) \leq r(\x_\tau)
   \end{align*}
  we have  
   \begin{align*}
   \frac{1}{2\mu}\|\x_\tau - \w_\tau\|^2\leq r(\x_\tau) - r(\w_\tau)\leq G\|\x_\tau - \w_\tau\|\Rightarrow \|\x_\tau - \w_\tau\|\leq 2G\mu,
   \end{align*}
   where the second inequality with an appropriate $G>0$ follows the Lipchitz continuity of $r$. 
         Then 
      \begin{align*}
      \E[\|\nabla  g(\w_\tau) - \nabla h(\w_\tau)  + \hat\partial r(\w_\tau)\|] &\leq  \E[(\gamma + L) \|\x_\tau - \z_\tau\| +  \frac{1}{\mu}\|\x_\tau - \z_\tau\| + 4GL\mu ].
      \end{align*}
      By setting $\mu = \epsilon$ and $K= O(1/\epsilon^4)$ and $\tau$ randomly sampled, we have $\E[\|\x_\tau - \z_\tau\|]\leq \epsilon^2$ and hence
      \begin{align*}
      \E[\|\nabla  g(\w_\tau) - \nabla h(\w_\tau)  + \hat\partial r(\w_\tau)\|] &\leq  O(\epsilon).
      \end{align*}
      For part (b), using
   \begin{align*}
\frac{1}{2\mu}\|\x_\tau - \w_\tau\|^2 +  r(\w_\tau) \leq r(\x_\tau)
   \end{align*}
      we have
       \begin{align*}
   \|\x_\tau - \w_\tau\|\leq \sqrt{2\mu \left(r(\x_\tau)  - \min_{\x\in\R^d}r(\x)\right)}\leq \sqrt{2\mu M},
      \end{align*}
      where $M>0$ exists due to Assumption~\ref{ass:last}(ii).
             Then 
      \begin{align*}
      \E[\|\nabla  g(\w_\tau) - \nabla h(\w_\tau)  + \hat\partial r(\w_\tau)\|] &\leq  \E[(\gamma + L) \|\x_\tau - \z_\tau\| +  \frac{1}{\mu}\|\x_\tau - \z_\tau\| + 2L\sqrt{2\mu M}].
      \end{align*}
      By setting $\mu = \epsilon^2$ and $K= O(1/\epsilon^6)$ and $\tau$ randomly sampled, we have $\E[\|\x_\tau - \z_\tau\|]\leq \epsilon^3$ and hence
      \begin{align*}
      \E[\|\nabla  g(\w_\tau) - \nabla h(\w_\tau)  + \hat\partial r(\w_\tau)\|] &\leq  O(\epsilon).
      \end{align*}
      For part (c) under Assumption~(\ref{ass:last}) (iii), we take expectation over the above inequality giving
       \begin{align*}
  \E[ \|\x_\tau - \w_\tau\|]\leq \sqrt{ 2\mu \left(\E[r(\x_\tau)  - \min_{\x\in\R^d}r(\x)]\right)}.
   \end{align*}
   Since using the SVRG, we can show $\E[f(\x_\tau) + r_\mu(\x_\tau)]$ is bounded above, i.e., $\x_\tau$ is in a bounded set (in expectation), which together with the assumption $r$ is lower bounded implies that there exists a constant $M>0$ such that $\E[r(\x_\tau) - \min_{\x\in\R^d}r(\x) ]\leq M$ for $\tau=1,\ldots, K$.      Then 
      \begin{align*}
      \E[\|\nabla  g(\w_\tau) - \nabla h(\w_\tau)  + \hat\partial r(\w_\tau)\|] &\leq  \E[(\gamma + 3L) \|\x_\tau - \z_\tau\| +  \frac{1}{\mu}\|\x_\tau - \z_\tau\| + 2L\sqrt{2\mu M} ].
      \end{align*}
By setting $\mu = \epsilon^2$ and $K=O(1/\epsilon^6)$ and $\tau$ randomly sampled, we have $\E \|\x_\tau - \z_\tau\| \leq \epsilon^3$ and hence 
      \begin{align*}
      \E[\|\nabla  g(\w_\tau) - \nabla h(\w_\tau)  + \hat\partial r(\w_\tau)\|] &\leq  O(\epsilon).
      \end{align*}
       For part (c) under Assumption~(\ref{ass:last}) (i) and Assumption~(\ref{ass:last}) (ii), we can use similar analysis as for part (a) and (b). 
\end{proof}

Next, we extend the results to the case that $h$ has a H\"{o}lder continuous gradient, i.e, the following conditions hold. 

\begin{ass}\label{ass:ghsm2}
 Assume $h$ is differentiable and has $(L, \nu)$-H\"{o}lder continuous gradient for some $\nu\in(0,1]$. 
\end{ass}
\begin{thm}\label{thm:ncns:reg:Appendix}
Regarding the problem~(\ref{eqn:Pf}), we have the following results:
\begin{itemize} 
\item [a.]
If Assumption~\ref{ass:ghsm2}, Assumption~\ref{ass:last} (i)  and Assumption~\ref{ass:1} (i) hold, then we can use Algorithm~\ref{alg:meta} with Algorithm~\ref{alg:sgd} (option 1)  to solve~(\ref{eqn:New}) with $\mu = \epsilon$, which returns a solution $\x_{\tau}$ after $K=O(1/\epsilon^{2(1+\nu)})$ stages satisfying 
\begin{align*}
\E[\|\x_\tau - \w_{\tau}\|]\leq O(\epsilon), \quad \E[\text{dist}(\nabla h(\w_\tau), \nabla g(\w_{\tau}) + \hat \partial r(\w_{\tau}))]\leq O(\epsilon^\nu),
\end{align*}
where $\w_\tau = \text{prox}_{\mu r}(\x_{\tau})$. 
\item [b.] If Assumption~\ref{ass:ghsm2}, Assumption~\ref{ass:last} (ii) and Assumption~\ref{ass:1} (i) hold, then we can use Algorithm~\ref{alg:meta} with Algorithm~\ref{alg:sgd} (option 1)  to solve~(\ref{eqn:New}) with $\mu = \epsilon^2$, which returns a solution $\x_{\tau}$ after $K=O(1/\epsilon^{2(2+\nu)})$ stages satisfying 
\begin{align*}
\E[\|\x_\tau - \w_{\tau}\|]\leq O(\epsilon), \quad \E[\text{dist}(\nabla h(\w_\tau), \nabla g(\w_{\tau}) + \hat \partial r(\w_{\tau}))]\leq O(\epsilon^\nu),
\end{align*}
where $\w_\tau = \text{prox}_{\mu r}(\x_{\tau})$. 
\item[c.]
If $g$ and $h$ have a finite-sum form, Assumption~\ref{ass:ghsm2} and Assumption~\ref{ass:4} hod,  then we can use Algorithm~\ref{alg:meta} with Algorithm~\ref{alg:SVRG}  to solve~(\ref{eqn:New}). We can  set  $\mu = \epsilon$ if Assumption~\ref{ass:last} (i) holds or $\mu=\epsilon^2$ if Assumption~\ref{ass:last} (ii) or (iii) holds.  The algorithm will return a solution $\x_{\tau}$ after $K=O(1/\epsilon^{2(1+\nu)})$ (corresponding to Assumption~\ref{ass:last} (i)) or $K=O(1/\epsilon^{2(2+\nu)})$ (corresponding to Assumption~\ref{ass:last} (ii) or (iii)) stages satisfying 
\begin{align*}
\E[\|\x_\tau - \w_{\tau}\|]\leq O(\epsilon), \quad \E[\text{dist}(\nabla h(\w_\tau), \nabla g(\w_{\tau}) + \hat \partial r(\w_{\tau}))]\leq O(\epsilon^\nu),
\end{align*}
where $\w_\tau = \text{prox}_{\mu r}(\x_{\tau})$. 
\end{itemize}

\end{thm}
\begin{proof} 
In the following proof, $\hat\partial r(\x)$ means there exists $\v\in\hat\partial r(\x)$. 
By applying the stochastic algorithms for DC functions in last section, at each stage  the following problem is solved approximately 
\begin{align*}
\z_k = \arg\min_{\x\in\R^d} \hat g(\x) + \frac{\gamma}{2}\|\x - \x_k\|^2 - (\partial  h(\x_k) + \frac{1}{\mu}\text{prox}_{\mu r}(\x_k))^{\top}(\x - \x_k).
\end{align*}
Then we have 
\begin{align*}
\E[\|\nabla  \hat g(\x_\tau)  -\nabla  h(\x_\tau) - \frac{1}{\mu}\text{prox}_{\mu r}(\x_\tau)\|] &\leq \E[ \|\nabla \hat g(\z_\tau) - \nabla\hat g(\x_\tau)\| +  \gamma\|\z_\tau - \x_\tau\|]\\
&\leq \E[L\|\x_\tau - \z_\tau\| +  \frac{1}{\mu}\|\x_\tau - \z_\tau\|+ \gamma \|\z_\tau - \x_\tau\|],
   \end{align*}
   for any $\tau\in\{1,\ldots, K\}$. Denote by $\w_\tau =  \text{prox}_{\mu r}(\x_\tau)$. 
   It is notable that $\frac{1}{\mu}\left(\x_\tau - \text{prox}_{\mu r}(\x_\tau)\right)\in\hat \partial r(\w_\tau) $. Then we have $\nabla  \hat g(\x_\tau)  -\partial h(\x_\tau) - \frac{1}{\mu}\text{prox}_{\mu r}(\x_\tau)  = \nabla g(\x_\tau) - \partial h(\x_\tau) + \hat\partial r(\w_\tau)$ and 
\begin{align*}
\E[\|\nabla  g(\x_\tau) - \nabla  h(\x_\tau)  + \hat\partial r(\w_\tau)\|] &\leq \E[L\|\x_\tau - \z_\tau\| +  \frac{1}{\mu}\|\x_\tau - \z_\tau\|+ \gamma \|\z_\tau - \x_\tau\|],
   \end{align*}
   which implies that 
\begin{align*}
&\E[\|\nabla  g(\w_\tau) -\nabla h(\w_\tau)  + \hat\partial r(\w_\tau)\|] \\
\leq& \E[(L+\gamma)\|\x_\tau - \z_\tau\| +  \frac{1}{\mu}\|\x_\tau - \z_\tau\| + L\|\x_\tau - \w_\tau\| + L\|\x_\tau - \w_\tau\|^\nu ],
   \end{align*}
where uses the facts that $g$ is $L$-smooth and $h$ has $L$-H\"{o}lder-continuous gradient with parameter $\nu\in(0,1]$.

   Next, we need to show that $\E[\|\x_\tau - \w_\tau\|]$ is small. The argument will be different for part (a), part (b) and part (c).
  For part (a), using
      \begin{align*}
 \frac{1}{2\mu}\|\x_\tau - \w_\tau\|^2 +  r(\w_\tau) \leq r(\x_\tau)
   \end{align*}
  we have  
   \begin{align*}
   \frac{1}{2\mu}\|\x_\tau - \w_\tau\|^2\leq r(\x_\tau) - r(\w_\tau)\leq G\|\x_\tau - \w_\tau\|\Rightarrow \|\x_\tau - \w_\tau\|\leq 2G\mu,
   \end{align*}
   where the second inequality with an appropriate $G>0$ follows the Lipchitz continuity of $r$. 
Then 
      \begin{align*}
      \E[\|\nabla  g(\w_\tau) - \nabla h(\w_\tau)  + \hat\partial r(\w_\tau)\|] &\leq  \E[(\gamma + L) \|\x_\tau - \z_\tau\| +  \frac{1}{\mu}\|\x_\tau - \z_\tau\| + 2GL\mu  + L(2G\mu)^\nu ].
      \end{align*}
      By setting $\mu = \epsilon$ and $K= O(1/\epsilon^{2(1+\nu)})$ and $\tau$ randomly sampled, we have $\E[\|\x_\tau - \z_\tau\|]\leq \epsilon^{1+\nu}$ and hence
      \begin{align*}
      \E[\|\nabla  g(\w_\tau) - \nabla h(\w_\tau)  + \hat\partial r(\w_\tau)\|] &\leq  O(\epsilon^\nu).
      \end{align*}
      For part (b), using
   \begin{align*}
\frac{1}{2\mu}\|\x_\tau - \w_\tau\|^2 +  r(\w_\tau) \leq r(\x_\tau)
   \end{align*}
      we have
       \begin{align*}
   \|\x_\tau - \w_\tau\|\leq \sqrt{2\mu \left(r(\x_\tau)  - \min_{\x\in\R^d}r(\x)\right)}\leq \sqrt{2\mu M},
      \end{align*}
      where $M>0$ exists due to Assumption~\ref{ass:last}(ii). 
      Then 
      \begin{align*}
      \E[\|\nabla  g(\w_\tau) - \nabla h(\w_\tau)  + \hat\partial r(\w_\tau)\|] &\leq  \E[(\gamma + L) \|\x_\tau - \z_\tau\| +  \frac{1}{\mu}\|\x_\tau - \z_\tau\| + L\sqrt{2\mu M}  + L(2\mu M)^{\nu/2} ].
      \end{align*}
      By setting $\mu = \epsilon^2$ and $K= O(1/\epsilon^{2(2+\nu)})$ and $\tau$ randomly sampled, we have $\E[\|\x_\tau - \z_\tau\|]\leq \epsilon^{2+\nu}$ and hence
      \begin{align*}
      \E[\|\nabla  g(\w_\tau) - \nabla h(\w_\tau)  + \hat\partial r(\w_\tau)\|] &\leq  O(\epsilon^\nu).
      \end{align*} 
For part (c) under Assumption~\ref{ass:last} (iii),  we take expectation over the above inequality giving
       \begin{align*}
  \E[ \|\x_\tau - \w_\tau\|]\leq \sqrt{ 2\mu \left(\E[r(\x_\tau)  - \min_{\x\in\R^d}r(\x)]\right)}.
   \end{align*}
   Since using the SVRG, we can show $\E[f(\x_\tau) + r_\mu(\x_\tau)]$ is bounded above, i.e., $\x_\tau$ is in a bounded set (in expectation), which together with the assumption $r$ is lower bounded implies that there exists a constant $M>0$ such that $\E[r(\x_\tau) - \min_{\x\in\R^d}r(\x) ]\leq M$ for $\tau=1,\ldots, K$.      Then 
      \begin{align*}
      \E[\|\nabla  g(\w_\tau) - \nabla h(\w_\tau)  + \hat\partial r(\w_\tau)\|] &\leq  \E[(\gamma + L) \|\x_\tau - \z_\tau\| +  \frac{1}{\mu}\|\x_\tau - \z_\tau\| + L\sqrt{2\mu M}  + L(2\mu M)^{\nu/2} ].
      \end{align*}
By setting $\mu = \epsilon^2$ and $K=O(1/\epsilon^{2(2+\nu)})$ and $\tau$ randomly sampled, we have $\E \|\x_\tau - \z_\tau\| \leq \epsilon^{2+\nu}$ and hence 
      \begin{align*}
      \E[\|\nabla  g(\w_\tau) - \nabla h(\w_\tau)  + \hat\partial r(\w_\tau)\|] &\leq  O(\epsilon^\nu).
      \end{align*}
             For part (c) under Assumption~(\ref{ass:last}) (i) and Assumption~(\ref{ass:last}) (ii), we can use similar analysis as for part (a) and (b). 

\end{proof}

\subsection{Improved Complexities with Increasing $\gamma$}\label{sec:new:improve}
In this subsection, we present improved complexities by simply changing the value of $\gamma$ across stages. The key idea is to use an increasing sequence of $\gamma$ across stages.  In particular, at the $k$-stage, we use $\gamma_k = O(k^\beta)$ with $0 <\beta < 1$. Let us abuse the notations $F, F_k$ defined by
\begin{align*}
F & =g(\x) + \hat r(\x) - h(\x) - R_{\mu}(\x)  = g(\x)  - h(\x) + r_\mu(\x)\\
F_k &= g(\x) +\hat r(\x)+ \frac{\gamma_k}{2}\|\x - \x_k\|^2 - (\nabla  h(\x_k) + \frac{1}{\mu}\text{prox}_{\mu r}(\x_k))^{\top}(\x - \x_k),
\end{align*}
where $\hat r(\x) = \frac{1}{2\mu}\|\x\|^2$. Similar to Lemma~\ref{lem:DC:new1}, we have the following result for the convergence of $\|\x_\tau - \z_{\tau}\|$ when SPG (option 1) is employed at each stage. 
\begin{lemma}\label{lem:new1}
Suppose Assumption~\ref{ass:1} (i) holds and Algorithm~\ref{alg:sgd} is employed for solving $F_k$ with parameters given in Proposition~\ref{prop:sgd} and with $\gamma_k= 3Lk^{\beta}$ with $0\leq \beta < 1$ and  $T_k = 3Lk/\gamma_k+3$, and there exists $\Delta>0$ such that $\E[F(\x_k) - \min_{\x}F(\x)]\leq \Delta$ for all $k\in\{1,\ldots, K\}$, then with a total of $K$ stages Algorithm~\ref{alg:meta} guarantees 
\begin{align*}
\E\bigg[\|\z_{\tau} - \x_\tau\|^2\bigg]&\leq \frac{8 \Delta (\alpha +1)}{3LK^{1+\beta}} + \frac{32G^2(\alpha +1) }{3K^{1+\beta}L^2},
\end{align*}
    where $\tau\in\{1,\ldots, K\}$ is sampled according to probabilities $p(\tau=k) = \frac{k^\alpha}{\sum_{k=1}^Kk^\alpha}$ with $\alpha\geq 1$. 
\end{lemma}
\begin{thm}(Improved Complexities of SSDC-SPG)\label{thm:new1}
Regarding the problem~(\ref{eqn:Pf}),  we have the following results:
\begin{itemize} 
\item [a.]
If Assumption~\ref{ass:ghsm}, Assumption~\ref{ass:last} (i)  and Assumption~\ref{ass:1} (i) hold, then we can use Algorithm~\ref{alg:meta} with Algorithm~\ref{alg:sgd} (option 1)  to solve~(\ref{eqn:New}) with $\mu = \epsilon$, $\gamma_k= 3Lk^{1/3}$ and $T_k = 3Lk/\gamma_k+3$, which returns a solution $\x_{\tau}$ after $K=O(1/\epsilon^3)$ stages satisfying 
\begin{align*}
\E[\|\x_\tau - \w_{\tau}\|]\leq O(\epsilon), \quad \E[\text{dist}(\nabla h(\w_\tau), \nabla g(\w_{\tau}) + \hat \partial r(\w_{\tau}))]\leq O(\epsilon),
\end{align*}
where $\w_\tau = \text{prox}_{\mu r}(\x_{\tau})$. The total iteration complexity is $\sum_{k=1}^KT_k \le O(K^{5/3}) = O(1/\epsilon^5)$. 
\item [b.] If Assumption~\ref{ass:ghsm}, Assumption~\ref{ass:last} (ii) and Assumption~\ref{ass:1} (i) hold, then we can use Algorithm~\ref{alg:meta} with Algorithm~\ref{alg:sgd} (option 1)  to solve~(\ref{eqn:New}) with $\mu = \epsilon^2$, $\gamma_k= 3Lk^{1/2}$ and $T_k = 3Lk/\gamma_k+3$, which returns a solution $\x_{\tau}$ after $K=O(1/\epsilon^{4})$ stages satisfying 
\begin{align*}
\E[\|\x_\tau - \w_{\tau}\|]\leq O(\epsilon), \quad \E[\text{dist}(\nabla h(\w_\tau), \nabla g(\w_{\tau}) + \hat \partial r(\w_{\tau}))]\leq O(\epsilon),
\end{align*}
where $\w_\tau = \text{prox}_{\mu r}(\x_{\tau})$. The total iteration complexity is $\sum_{k=1}^KT_k \le O(K^{3/2}) = O(1/\epsilon^{6})$. 
\end{itemize}
\end{thm}
{\bf Remark:} We  note that after the first version of this paper was posted online, \cite{DBLP:journals/corr/abs/1901.08369} also considered a special setting of our problem~(\ref{eqn:Pf}) where $g$ is a smooth and non-convex, $h=0$, $r$ is  non-smooth non-convex and Lipchitz continuous, which is covered in part (a) of the above theorem by noting that a smooth function $g$ can be written as a DC decomposition of two smooth convex functions. In terms of iteration complexity, they obtained the same complexity of $O(1/\epsilon^5)$ but with a large mini-batch size equal to $O(1/\epsilon)$. When using  a mini-batch size of $1$, their algorithm has a worse complexity of $O(1/\epsilon^6)$. In contrast, our algorithm does not need a large mini-batch size. In addition, the step size in their algorithm is also small in the order of $O(\epsilon^2)$ or $(\epsilon^3)$ for finding an $\epsilon$-critical point, while the step size in our algorithm is decreasing in a stagewise manner.

\begin{proof} 
This proof is similar to the proof of Theorem~\ref{thm:ncns:reg}. Following the analysis of Theorem~\ref{thm:ncns:reg} we know
for any $\tau\in\{1,\ldots, K\}$,
\begin{align*}
\E[\|\nabla  g(\w_\tau) - \nabla h(\w_\tau)  + \hat\partial r(\w_\tau)\|] &\leq \E[(L+\gamma_\tau)\|\x_\tau - \z_\tau\| +  \frac{1}{\mu}\|\x_\tau - \z_\tau\| + 2L\|\x_\tau - \w_\tau\|].
\end{align*}
For part (a), by the setting of $\gamma_\tau= 3L\tau^{1/3}$, we know $\gamma_\tau \leq 3LK^{1/3}$ so that
\begin{align*}
\E[\|\nabla  g(\w_\tau) - \nabla h(\w_\tau)  + \hat\partial r(\w_\tau)\|] &\leq \E[(L+3LK^{1/3})\|\x_\tau - \z_\tau\| +  \frac{1}{\mu}\|\x_\tau - \z_\tau\| + 2L\|\x_\tau - \w_\tau\|].
 \end{align*}   
 By using 
 \begin{align*}
 \frac{1}{2\mu}\|\x_\tau - \w_\tau\|^2 +  r(\w_\tau) \leq r(\x_\tau)
\end{align*}
  we have  
   \begin{align*}
   \frac{1}{2\mu}\|\x_\tau - \w_\tau\|^2\leq r(\x_\tau) - r(\w_\tau)\leq G\|\x_\tau - \w_\tau\|\Rightarrow \|\x_\tau - \w_\tau\|\leq 2G\mu,
   \end{align*}
   where the second inequality with an appropriate $G>0$ follows the Lipchitz continuity of $r$. 
         Then 
      \begin{align*}
      \E[\|\nabla  g(\w_\tau) - \nabla h(\w_\tau)  + \hat\partial r(\w_\tau)\|] &\leq  \E[(L+3LK^{1/3}) \|\x_\tau - \z_\tau\| +  \frac{1}{\mu}\|\x_\tau - \z_\tau\| + 4GL\mu ].
      \end{align*}
      By setting $\mu = \epsilon$ and $K= O(1/\epsilon^3)$ and $\tau$ randomly sampled, we have $\E[\|\x_\tau - \z_\tau\|]\leq \epsilon^2$ by Lemma~\ref{lem:new1} with $\beta=1/3$ and hence
      \begin{align*}
     \E[\|\nabla  g(\w_\tau) - \nabla h(\w_\tau)  + \hat\partial r(\w_\tau)\|] \leq  O(\epsilon), \quad \E[\|\x_\tau - \w_{\tau}\|]\leq O(\epsilon). 
      \end{align*}   
For part (b), by the setting of $\gamma_\tau= 3L\tau^{1/2}$, we know $\gamma_\tau \leq 3LK^{1/2}$ so that
\begin{align*}
\E[\|\nabla  g(\w_\tau) - \nabla h(\w_\tau)  + \hat\partial r(\w_\tau)\|] &\leq \E[(L+3LK^{1/2})\|\x_\tau - \z_\tau\| +  \frac{1}{\mu}\|\x_\tau - \z_\tau\| + 2L\|\x_\tau - \w_\tau\|].
 \end{align*}   
By using
   \begin{align*}
\frac{1}{2\mu}\|\x_\tau - \w_\tau\|^2 +  r(\w_\tau) \leq r(\x_\tau)
   \end{align*}
      we have
       \begin{align*}
   \|\x_\tau - \w_\tau\|\leq \sqrt{2\mu \left(r(\x_\tau)  - \min_{\x\in\R^d}r(\x)\right)}\leq \sqrt{2\mu M},
      \end{align*}
      where $M>0$ exists due to Assumption~\ref{ass:last}(ii).
             Then 
      \begin{align*}
      \E[\|\nabla  g(\w_\tau) - \nabla h(\w_\tau)  + \hat\partial r(\w_\tau)\|] &\leq  \E[(L+3LK^{1/2}) \|\x_\tau - \z_\tau\| +  \frac{1}{\mu}\|\x_\tau - \z_\tau\| + 2L\sqrt{2\mu M}].
      \end{align*}
      By setting $\mu = \epsilon^2$ and $K= O(1/\epsilon^{4})$ and $\tau$ randomly sampled, we have $\E[\|\x_\tau - \z_\tau\|]\leq \epsilon^3$ by Lemma~\ref{lem:new1} with $\beta=1/2$ and hence
      \begin{align*}
      \E[\|\nabla  g(\w_\tau) - \nabla h(\w_\tau)  + \hat\partial r(\w_\tau)\|] &\leq  O(\epsilon), \quad \E[\|\x_\tau - \w_{\tau}\|]\leq O(\epsilon).
      \end{align*}
\end{proof}

We can also solve the subproblem by using SVRG (Algorithm~\ref{alg:SVRG}) when $g$ and $h$ are of finite-sum form. Similar to Lemma~\ref{lem:DC:new3}, we have the following result for the convergence of $\|\x_\tau - \z_{\tau}\|$.
\begin{lemma}\label{lem:new2}
    Suppose Assumption~\ref{ass:4} holds and  there exists $\Delta>0$ such that $\E[F(\x_1) - \min_{\x}F(\x)]\leq \Delta$,  and Algorithm~\ref{alg:SVRG} is employed for solving $F_k$ with  $\gamma_k= ck^{\beta}$ with $0\leq \beta < 1$ and $c>0$, $\eta_{k} = 0.05/L$, $T_k\geq \max(2, 200L/\gamma_k)$, $S_k =\lceil \log_2(k)\rceil$, then with a total of $K$ stages Algorithm~\ref{alg:meta} guarantees  
\begin{align*}
     \E\bigg[\|\z_{\tau} - \x_\tau\|^2\bigg]&\leq \frac{12\Delta (\alpha +1)}{cK^{1+\beta}},
\end{align*}
    where $\tau\in\{1,\ldots, K\}$ is sampled according to probabilities $p(\tau=k) = \frac{k^\alpha}{\sum_{k=1}^Kk^\alpha}$ with $\alpha\geq 1$. 
\end{lemma}
\begin{thm}(Improved Complexities of SSDC-SVRG)\label{thm:new2}
Regarding the problem~(\ref{eqn:Pf}), we have the following results:
\begin{itemize} 
\item [a.]
If $g$ and $h$ have a finite-sum form, Assumption~\ref{ass:ghsm} and Assumption~\ref{ass:4} hold, and Assumption~\ref{ass:last} (i)  holds,  then we can use Algorithm~\ref{alg:meta} with Algorithm~\ref{alg:SVRG}  to solve~(\ref{eqn:New}). We can  set  $\mu = \epsilon$, $\gamma_k= ck^{1/3}$, $T_k\geq \max(2, 200L/\gamma_k)$, $S_k =\lceil \log_2(k)\rceil$.  The algorithm will return a solution $\x_{\tau}$ after $K=O(1/\epsilon^3)$ stages satisfying 
\begin{align*}
\E[\|\x_\tau - \w_{\tau}\|]\leq O(\epsilon), \quad \E[\text{dist}(\nabla h(\w_\tau), \nabla g(\w_{\tau}) + \hat \partial r(\w_{\tau}))]\leq O(\epsilon),
\end{align*}
where $\w_\tau = \text{prox}_{\mu r}(\x_{\tau})$. The total gradient complexity is $\widetilde O(n/\epsilon^3)$.
\item [b.]
If $g$ and $h$ have a finite-sum form, Assumption~\ref{ass:ghsm} and Assumption~\ref{ass:4} hold, and Assumption~\ref{ass:last} (ii) or (iii) holds,  then we can use Algorithm~\ref{alg:meta} with Algorithm~\ref{alg:SVRG}  to solve~(\ref{eqn:New}). We can  set  $\mu=\epsilon^2$, $\gamma_k= ck^{1/2}$, $T_k\geq \max(2, 200L/\gamma_k)$, $S_k =\lceil \log_2(k)\rceil$.  The algorithm will return a solution $\x_{\tau}$ after $K=O(1/\epsilon^{4})$ stages satisfying 
\begin{align*}
\E[\|\x_\tau - \w_{\tau}\|]\leq O(\epsilon), \quad \E[\text{dist}(\nabla h(\w_\tau), \nabla g(\w_{\tau}) + \hat \partial r(\w_{\tau}))]\leq O(\epsilon),
\end{align*}
where $\w_\tau = \text{prox}_{\mu r}(\x_{\tau})$. The total gradient complexity is $\widetilde O(n/\epsilon^{4})$.
\end{itemize}
\end{thm}
{\bf Remark:} \cite{DBLP:journals/corr/abs/1901.08369} also considered a special setting of our problem~(\ref{eqn:Pf}) where $g$ is a finite-sum smooth and non-convex function, $h=0$, $r$ is  non-smooth non-convex and Lipchitz continuous, which is covered in part (a) of the above theorem. In terms of iteration complexity, they obtained a gradient complexity of $O(n^{2/3}/\epsilon^3)$ but with a large mini-batch size equal to $n^{2/3}$. Their gradient complexity is better than the result of part (a) in the above theorem. However, their algorithm needs to use a large mini-batch size and a small step size. 
In contrast, our algorithm does not need a large mini-batch size, and the step size in SVRG is a constant. In addition, we also emphasize that our algorithm SSDC is a general framework, which allows one to employ any suitable stochastic algorithms for smooth and strongly convex functions. It therefore gives us much more flexibility in practice. However, it is an interesting question whether one can obtain a better dependence on $n$ in our framework. 


\begin{proof} 
Similar to the proof of Theorem~\ref{thm:new1}, we know
for any $\tau\in\{1,\ldots, K\}$,
\begin{align*}
\E[\|\nabla  g(\w_\tau) - \nabla h(\w_\tau)  + \hat\partial r(\w_\tau)\|] &\leq \E[(L+\gamma_K)\|\x_\tau - \z_\tau\| +  \frac{1}{\mu}\|\x_\tau - \z_\tau\| + 2L\|\x_\tau - \w_\tau\|].
 \end{align*}   
 For part (a), using 
 \begin{align*}
 \frac{1}{2\mu}\|\x_\tau - \w_\tau\|^2 +  r(\w_\tau) \leq r(\x_\tau)
\end{align*}
  we have  
   \begin{align*}
   \frac{1}{2\mu}\|\x_\tau - \w_\tau\|^2\leq r(\x_\tau) - r(\w_\tau)\leq G\|\x_\tau - \w_\tau\|\Rightarrow \|\x_\tau - \w_\tau\|\leq 2G\mu,
   \end{align*}
   where the second inequality with an appropriate $G>0$ follows the Lipchitz continuity of $r$. 
         Then by the setting of $\gamma_K = 3LK^{1/3}$,
      \begin{align*}
      \E[\|\nabla  g(\w_\tau) - \nabla h(\w_\tau)  + \hat\partial r(\w_\tau)\|] &\leq  \E[(L+3LK^{1/3}) \|\x_\tau - \z_\tau\| +  \frac{1}{\mu}\|\x_\tau - \z_\tau\| + 4GL\mu ].
      \end{align*}
      By setting $\mu = \epsilon$ and $K= O(1/\epsilon^3)$ and $\tau$ randomly sampled, we have $\E[\|\x_\tau - \z_\tau\|]\leq \epsilon^2$ by Lemma~\ref{lem:new2} with $\beta=1/3$ and hence
      \begin{align*}
     \E[\|\nabla  g(\w_\tau) - \nabla h(\w_\tau)  + \hat\partial r(\w_\tau)\|] \leq  O(\epsilon), \quad \E[\|\x_\tau - \w_{\tau}\|]\leq O(\epsilon). 
      \end{align*}
      For part (b), if Assumption~\ref{ass:last}(ii) holds, then using
   \begin{align*}
\frac{1}{2\mu}\|\x_\tau - \w_\tau\|^2 +  r(\w_\tau) \leq r(\x_\tau)
   \end{align*}
      we have
       \begin{align*}
   \|\x_\tau - \w_\tau\|\leq \sqrt{2\mu \left(r(\x_\tau)  - \min_{\x\in\R^d}r(\x)\right)}\leq \sqrt{2\mu M}.
      \end{align*}
      If Assumption~\ref{ass:last}(iii) holds, we take expectation over the above inequality giving
       \begin{align*}
  \E[ \|\x_\tau - \w_\tau\|]\leq \sqrt{ 2\mu \left(\E[r(\x_\tau)  - \min_{\x\in\R^d}r(\x)]\right)}.
   \end{align*}
   Since using the SVRG, we can show $\E[f(\x_\tau) + r_\mu(\x_\tau)]$ is bounded above, i.e., $\x_\tau$ is in a bounded set (in expectation), which together with the assumption $r$ is lower bounded implies that there exists a constant $M>0$ such that $\E[r(\x_\tau) - \min_{\x\in\R^d}r(\x) ]\leq M$ for $\tau=1,\ldots, K$. 
   For both cases, by the setting of $\gamma_K = 3LK^{1/2}$ we have 
      \begin{align*}
      \E[\|\nabla  g(\w_\tau) - \nabla h(\w_\tau)  + \hat\partial r(\w_\tau)\|] &\leq  \E[(3L+3LK^{1/2}) \|\x_\tau - \z_\tau\| +  \frac{1}{\mu}\|\x_\tau - \z_\tau\| + 2L\sqrt{2\mu M} ].
      \end{align*}
By setting $\mu = \epsilon^2$ and $K=O(1/\epsilon^{4})$ and $\tau$ randomly sampled, we have $\E \|\x_\tau - \z_\tau\| \leq \epsilon^3$ by Lemma~\ref{lem:new2}  and hence 
      \begin{align*}
      \E[\|\nabla  g(\w_\tau) - \nabla h(\w_\tau)  + \hat\partial r(\w_\tau)\|] &\leq  O(\epsilon), \quad \E[\|\x_\tau - \w_{\tau}\|]\leq O(\epsilon).
      \end{align*}
\end{proof}

Next, we present the results under the H\"{o}lder continuous gradient condition of $h$. These results are simple extension of that in Theorems~\ref{thm:new1}, and Theorem~\ref{thm:new2}.
\begin{thm}\label{thm:new1:Appendix}
Regarding the problem~(\ref{eqn:Pf}), we have the following results:
\begin{itemize} 
\item [a.]
If Assumption~\ref{ass:ghsm2}, Assumption~\ref{ass:last} (i)  and Assumption~\ref{ass:1} (i) hold, then we can use Algorithm~\ref{alg:meta} with Algorithm~\ref{alg:sgd} (option 1)  to solve~(\ref{eqn:New}) with $\mu = \epsilon$, $\gamma_k= 3Lk^{1/3}$ and $T_k = 3Lk/\gamma_k+3$, which returns a solution $\x_{\tau}$ after $K=O(1/\epsilon^{\frac{3(1+\nu)}{2}})$ stages satisfying 
\begin{align*}
\E[\|\x_\tau - \w_{\tau}\|]\leq O(\epsilon), \quad \E[\text{dist}(\nabla h(\w_\tau), \nabla g(\w_{\tau}) + \hat \partial r(\w_{\tau}))]\leq O(\epsilon^\nu),
\end{align*}
where $\w_\tau = \text{prox}_{\mu r}(\x_{\tau})$. The total iteration complexity  for such convergence guarantee  is $\sum_{k=1}^KT_k \le O(K^{5/3}) = O(1/\epsilon^{\frac{5(1+\nu)}{2}})$. 
\item [b.] If Assumption~\ref{ass:ghsm2}, Assumption~\ref{ass:last} (ii) and Assumption~\ref{ass:1} (i) hold, then we can use Algorithm~\ref{alg:meta} with Algorithm~\ref{alg:sgd} (option 1)  to solve~(\ref{eqn:New}) with $\mu = \epsilon^2$, $\gamma_k= 3Lk^{1/2}$ and $T_k = 3Lk/\gamma_k+3$, which returns a solution $\x_{\tau}$ after $K=O(1/\epsilon^{\frac{4(2+\nu)}{3}})$ stages satisfying 
\begin{align*}
\E[\|\x_\tau - \w_{\tau}\|]\leq O(\epsilon), \quad \E[\text{dist}(\nabla h(\w_\tau), \nabla g(\w_{\tau}) + \hat \partial r(\w_{\tau}))]\leq O(\epsilon^\nu),
\end{align*}
where $\w_\tau = \text{prox}_{\mu r}(\x_{\tau})$. The total iteration complexity for such convergence guarantee is $\sum_{k=1}^KT_k \le O(K^{3/2}) = O(1/\epsilon^{ 2(2+\nu)})$. 
\item [c.]
If $g$ and $h$ have a finite-sum form, Assumption~\ref{ass:ghsm2},  Assumption~\ref{ass:4} and Assumption~\ref{ass:last} (i)  hold,  then we can use Algorithm~\ref{alg:meta} with Algorithm~\ref{alg:SVRG}  to solve~(\ref{eqn:New}). We can  set  $\mu = \epsilon$, $\gamma_k= ck^{1/3}$, $T_k\geq \max(2, 200L/\gamma_k)$, $S_k =\lceil \log_2(k)\rceil$.  The algorithm will return a solution $\x_{\tau}$ after $K=O(1/\epsilon^{\frac{3(1+\nu)}{2}})$ stages satisfying 
\begin{align*}
\E[\|\x_\tau - \w_{\tau}\|]\leq O(\epsilon), \quad \E[\text{dist}(\nabla h(\w_\tau), \nabla g(\w_{\tau}) + \hat \partial r(\w_{\tau}))]\leq O(\epsilon^\nu),
\end{align*}
where $\w_\tau = \text{prox}_{\mu r}(\x_{\tau})$. The total gradient complexity  for such convergence guarantee  is $\widetilde O(n/\epsilon^{\frac{3(1+\nu)}{2}})$.
\item [d.]
If $g$ and $h$ have a finite-sum form, Assumption~\ref{ass:ghsm2},  Assumption~\ref{ass:4} and  Assumption~\ref{ass:last} (ii) or (iii) hold,  then we can use Algorithm~\ref{alg:meta} with Algorithm~\ref{alg:SVRG}  to solve~(\ref{eqn:New}). We can  set  $\mu=\epsilon^2$, $\gamma_k= ck^{1/2}$, $T_k\geq \max(2, 200L/\gamma_k)$, $S_k =\lceil \log_2(k)\rceil$.  The algorithm will return a solution $\x_{\tau}$ after $K=O(1/\epsilon^{\frac{4(2+\nu)}{3}})$ stages satisfying 
\begin{align*}
\E[\|\x_\tau - \w_{\tau}\|]\leq O(\epsilon), \quad \E[\text{dist}(\nabla h(\w_\tau), \nabla g(\w_{\tau}) + \hat \partial r(\w_{\tau}))]\leq O(\epsilon^\nu),
\end{align*}
where $\w_\tau = \text{prox}_{\mu r}(\x_{\tau})$. The total gradient complexity  for such convergence guarantee is $\widetilde O(n/\epsilon^{\frac{4(2+\nu)}{3}})$.
\end{itemize}
\end{thm}
{\bf Remark:} When deriving the total gradient complexity for ensuring $ \E[\text{dist}(\nabla h(\w_\tau), \nabla g(\w_{\tau}) + \hat \partial r(\w_{\tau}))]\leq O(\epsilon)$ we can simply replace $\epsilon$ by $\epsilon^{1/\nu}$ in the above results. 

\begin{proof} 
This proof is similar to the proof of Theorem~\ref{thm:ncns:reg:Appendix} and Theorems~\ref{thm:new1} and~\ref{thm:new2}. Following the analysis of Theorem~\ref{thm:ncns:reg:Appendix} we know
for any $\tau\in\{1,\ldots, K\}$,
\begin{align*}
&\E[\|\nabla  g(\w_\tau) - \partial h(\w_\tau)  + \hat\partial r(\w_\tau)\|] \\
\leq &\E[(L+\gamma_\tau)\|\x_\tau - \z_\tau\| +  \frac{1}{\mu}\|\x_\tau - \z_\tau\| + L\|\x_\tau - \w_\tau\|+ L\|\x_\tau - \w_\tau\|^\nu].
\end{align*}
For part (a), by the setting of $\gamma_\tau= 3L\tau^{1/3}$, we know $\gamma_\tau \leq 3LK^{1/3}$ so that
\begin{align*}
&\E[\|\nabla  g(\w_\tau) - \partial h(\w_\tau)  + \hat\partial r(\w_\tau)\|] \\
\leq& \E[(L+3LK^{1/3})\|\x_\tau - \z_\tau\| +  \frac{1}{\mu}\|\x_\tau - \z_\tau\| + L\|\x_\tau - \w_\tau\|+ L\|\x_\tau - \w_\tau\|^\nu].
 \end{align*}   
 By using 
 \begin{align*}
 \frac{1}{2\mu}\|\x_\tau - \w_\tau\|^2 +  r(\w_\tau) \leq r(\x_\tau)
\end{align*}
  we have  
   \begin{align*}
   \frac{1}{2\mu}\|\x_\tau - \w_\tau\|^2\leq r(\x_\tau) - r(\w_\tau)\leq G\|\x_\tau - \w_\tau\|\Rightarrow \|\x_\tau - \w_\tau\|\leq 2G\mu,
   \end{align*}
   where the second inequality with an appropriate $G>0$ follows the Lipchitz continuity of $r$. 
         Then 
      \begin{align*}
      \E[\|\nabla  g(\w_\tau) - \nabla h(\w_\tau)  + \hat\partial r(\w_\tau)\|] &\leq  \E[(L+3LK^{1/3}) \|\x_\tau - \z_\tau\| +  \frac{1}{\mu}\|\x_\tau - \z_\tau\| + 2GL\mu + L (2G\mu)^\nu].
      \end{align*}
      By setting $\mu = \epsilon$ and $K= O(1/\epsilon^{\frac{3(1+\nu)}{2}})$ and $\tau$ randomly sampled, we have $\E[\|\x_\tau - \z_\tau\|]\leq \epsilon^{(1+\nu)}$ by Lemma~\ref{lem:new1} with $\beta=1/3$ and hence
      \begin{align*}
     \E[\|\nabla  g(\w_\tau) - \nabla h(\w_\tau)  + \hat\partial r(\w_\tau)\|] \leq  O(\epsilon^\nu), \quad \E[\|\x_\tau - \w_{\tau}\|]\leq O(\epsilon). 
      \end{align*}
      For part (b), 
      by the setting of $\gamma_\tau= 3L\tau^{1/2}$, we know $\gamma_\tau \leq 3LK^{1/2}$ so that
\begin{align*}
&\E[\|\nabla  g(\w_\tau) - \partial h(\w_\tau)  + \hat\partial r(\w_\tau)\|] \\
\leq& \E[(L+3LK^{1/2})\|\x_\tau - \z_\tau\| +  \frac{1}{\mu}\|\x_\tau - \z_\tau\| + L\|\x_\tau - \w_\tau\|+ L\|\x_\tau - \w_\tau\|^\nu].
 \end{align*}   
      By using
   \begin{align*}
\frac{1}{2\mu}\|\x_\tau - \w_\tau\|^2 +  r(\w_\tau) \leq r(\x_\tau)
   \end{align*}
      we have
       \begin{align*}
   \|\x_\tau - \w_\tau\|\leq \sqrt{2\mu \left(r(\x_\tau)  - \min_{\x\in\R^d}r(\x)\right)}\leq \sqrt{2\mu M},
      \end{align*}
      where $M>0$ exists due to Assumption~\ref{ass:last}(ii).
             Then 
      \begin{align*}
      \E[\|\nabla  g(\w_\tau) - \nabla h(\w_\tau)  + \hat\partial r(\w_\tau)\|] &\leq  \E[(L+3LK^{1/2}) \|\x_\tau - \z_\tau\| +  \frac{1}{\mu}\|\x_\tau - \z_\tau\| + 2L(2\mu M)^{\nu/2}].
      \end{align*}
      By setting $\mu = \epsilon^2$ and $K= O(1/\epsilon^{\frac{4(2+\nu)}{3}})$ and $\tau$ randomly sampled, we have $\E[\|\x_\tau - \z_\tau\|]\leq \epsilon^{2+\nu}$ by Lemma~\ref{lem:new1} with $\beta=1/2$ and hence
      \begin{align*}
      \E[\|\nabla  g(\w_\tau) - \nabla h(\w_\tau)  + \hat\partial r(\w_\tau)\|] &\leq  O(\epsilon^\nu), \quad \E[\|\x_\tau - \w_{\tau}\|]\leq O(\epsilon).
      \end{align*}
Part (c) and Part (d) can be proved similarly as in Theorem~\ref{thm:new2}. 
\end{proof}

\subsection{Improved Complexities when $\nu$ of $h$ is known}\label{sec:new:improve}
Similar to the case when $r$ is convex, we can also improve the complexity for solving problems with non-smooth non-convex $r$ and $h$ with a H\"{o}lder continuous gradient under the condition that $\nu$ is known. 
\begin{thm}\label{thm:new1:Appendix2}
Regarding the problem~(\ref{eqn:Pf}), we have the following results:
\begin{itemize} 
\item [a.]
If Assumption~\ref{ass:ghsm2}, Assumption~\ref{ass:last} (i)  and Assumption~\ref{ass:1} (i) hold, then we can use Algorithm~\ref{alg:meta} with Algorithm~\ref{alg:sgd} (option 1)  to solve~(\ref{eqn:New}) with $\mu = \epsilon$, $\gamma_k= 3Lk^{1/(1+2\nu)}$ and $T_k = 3Lk/\gamma_k+3$, which returns a solution $\x_{\tau}$ after $K=O(1/\epsilon^{1+2\nu})$ stages satisfying 
\begin{align*}
\E[\|\x_\tau - \w_{\tau}\|]\leq O(\epsilon), \quad \E[\text{dist}(\nabla h(\w_\tau), \nabla g(\w_{\tau}) + \hat \partial r(\w_{\tau}))]\leq O(\epsilon^\nu),
\end{align*}
where $\w_\tau = \text{prox}_{\mu r}(\x_{\tau})$. The total iteration complexity is $\sum_{k=1}^KT_k= O(1/\epsilon^{1+4\nu})$. 
\item [b.] If Assumption~\ref{ass:ghsm2}, Assumption~\ref{ass:last} (ii) and Assumption~\ref{ass:1} (i) hold, then we can use Algorithm~\ref{alg:meta} with Algorithm~\ref{alg:sgd} (option 1)  to solve~(\ref{eqn:New}) with $\mu = \epsilon^2$, $\gamma_k= 3Lk^{1/(1+\nu)}$ and $T_k = 3Lk/\gamma_k+3$, which returns a solution $\x_{\tau}$ after $K=O(1/\epsilon^{2(1+\nu)})$ stages satisfying 
\begin{align*}
\E[\|\x_\tau - \w_{\tau}\|]\leq O(\epsilon), \quad \E[\text{dist}(\nabla h(\w_\tau), \nabla g(\w_{\tau}) + \hat \partial r(\w_{\tau}))]\leq O(\epsilon^\nu),
\end{align*}
where $\w_\tau = \text{prox}_{\mu r}(\x_{\tau})$. The total iteration complexity for such convergence guarantee is $\sum_{k=1}^KT_k = O(1/\epsilon^{ 2(1+2\nu)})$. 
\item [c.]
If $g$ and $h$ have a finite-sum form, Assumption~\ref{ass:ghsm2},  Assumption~\ref{ass:4}  and Assumption~\ref{ass:last} (i)  holds,  then we can use Algorithm~\ref{alg:meta} with Algorithm~\ref{alg:SVRG}  to solve~(\ref{eqn:New}). We can  set  $\mu = \epsilon$, $\gamma_k= ck^{1/(1+2\nu)}$, $T_k\geq \max(2, 200L/\gamma_k)$, $S_k =\lceil \log_2(k)\rceil$.  The algorithm will return a solution $\x_{\tau}$ after $K=O(1/\epsilon^{1+2\nu})$ stages satisfying 
\begin{align*}
\E[\|\x_\tau - \w_{\tau}\|]\leq O(\epsilon), \quad \E[\text{dist}(\nabla h(\w_\tau), \nabla g(\w_{\tau}) + \hat \partial r(\w_{\tau}))]\leq O(\epsilon^\nu),
\end{align*}
where $\w_\tau = \text{prox}_{\mu r}(\x_{\tau})$. The total gradient complexity is $\widetilde O(n/\epsilon^{1+2\nu})$.
\item [d.]
 If $g$ and $h$ have a finite-sum form, Assumption~\ref{ass:ghsm2},  Assumption~\ref{ass:4}  and Assumption~\ref{ass:last} (ii) or (iii) holds,  then we can use Algorithm~\ref{alg:meta} with Algorithm~\ref{alg:SVRG}  to solve~(\ref{eqn:New}). We can  set  $\mu=\epsilon^2$, $\gamma_k= ck^{1/(1+\nu)}$, $T_k\geq \max(2, 200L/\gamma_k)$, $S_k =\lceil \log_2(k)\rceil$.  The algorithm will return a solution $\x_{\tau}$ after $K=O(1/\epsilon^{2(1+\nu)})$ stages satisfying 
\begin{align*}
\E[\|\x_\tau - \w_{\tau}\|]\leq O(\epsilon), \quad \E[\text{dist}(\nabla h(\w_\tau), \nabla g(\w_{\tau}) + \hat \partial r(\w_{\tau}))]\leq O(\epsilon^\nu),
\end{align*}
where $\w_\tau = \text{prox}_{\mu r}(\x_{\tau})$. The total gradient complexity is $\widetilde O(n/\epsilon^{2(1+\nu)})$.
\end{itemize}
\end{thm}
{\bf Remark:} The above results can be easily proved as Theorem~\ref{thm:new1:Appendix}. 

\paragraph{Deterministic Methods for~(\ref{eqn:Pf}) with a Non-smooth Non-convex Regularizer.}
Finally, we note that for the problem~(\ref{eqn:Pf}) with a non-smooth non-convex regularizer, a non-asymptotic convergence result for a deterministic optimization method is novel and interesting of its own. In particular, SVRG can be replaced by deterministic gradient based methods (e.g., accelerated gradient methods~\citep{citeulike:9501961,Beck:2009:FIS:1658360.1658364,Composite}) to enjoy a similar non-asymptotic convergence in terms of $\epsilon$. From this perspective, we can also generalize the results to the case when both $g$ and $h$ have H\"{o}lder continuous gradients by using a universal gradient method~\citep{Nesterov2015} that can adapts to the  actual level of smoothness of the objective function for solving each subproblem. 


\section{Numerical Experiments}
In this section, we perform some experiments for solving different tasks to demonstrate effectiveness of proposed algorithms by comparing with different baselines. 
We use very large-scale datasets from libsvm website in experiments, including real-sim (n = 72309) and rcv1 (n=20242) for classification, million songs (n = 463715) for regression. 
For all algorithms, the initial stepsizes are tuned in the range of $\{10^{-6:1:4}\}$, and the same initial solution with all zero entries is used.
 The initial iteration number $T_0$ of SSDC-SPG is tuned in $ \{10^{1:1:4}\}$. For SSDC algorithms, we use a constant parameter $\gamma$ and also tune its value in $\{10^{-7:1:3}\}$. 
 

First, we compare SSDC algorithms with SDCA~\citep{pmlr-v70-thi17a}, SGD~\citep{sgdweakly18}, SVRG~\citep{reddi2016proximal}, GIST~\citep{DBLP:conf/icml/GongZLHY13} and GPPA~\citep{doi:10.1080/02331934.2016.1253694} for learning with a DC regularizer: minimizing  logistic loss with a SCAD regularizer for classification and huber loss with a MCP regularizer for regression. The parameter in Huber loss is set to be $1$. The value of regularization parameter is set to be $10^{-4}$.  Note that since these regularizers are weakly convex, SGD with step size $\eta_0/\sqrt{t}$ is applicable~\citep{sgdweakly18}. We set the inner iteration number of SVRG as $n$ following ~\citep{reddi2016proximal} and the same value is used as the inner iteration number $T$ of SSDC-SVRG. We set the values of parameters in GIST with their suggested  BB rule~\citep{DBLP:conf/icml/GongZLHY13}. 
Similar to~\citep{pmlr-v70-thi17a}, we tune the batch size of SDCA in a wide range and choose the one with the best performance. 
 GIST and GPPA are two deterministic algorithms that use the all data points in each iteration. For fairness of comparison, we plot the objective in log scale versus the number of gradient computations in Figure~\ref{fig01}. 

\begin{figure}[t]
\centering
\centering
{\includegraphics[scale=.33]{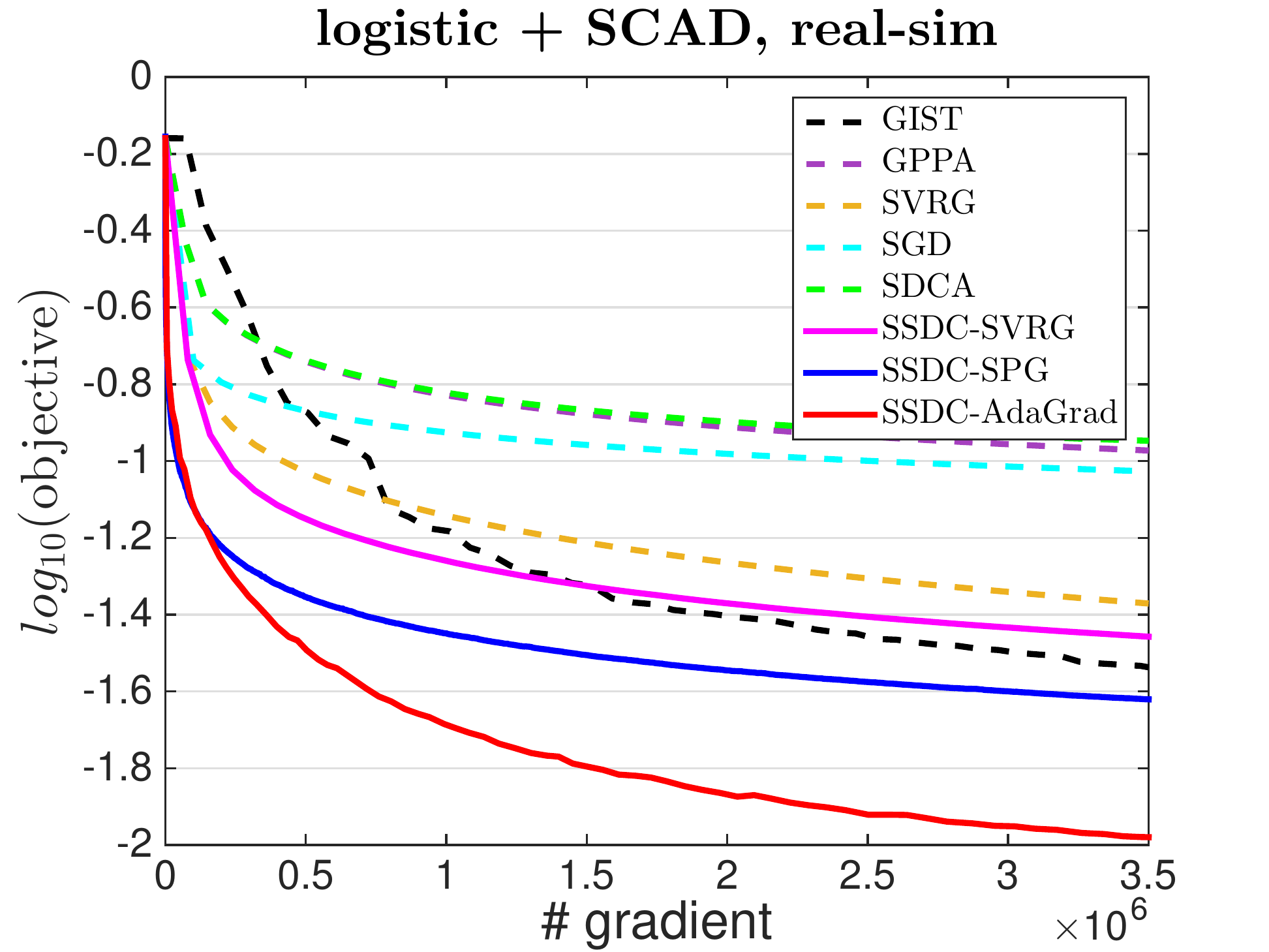}}
{\includegraphics[scale=.33]{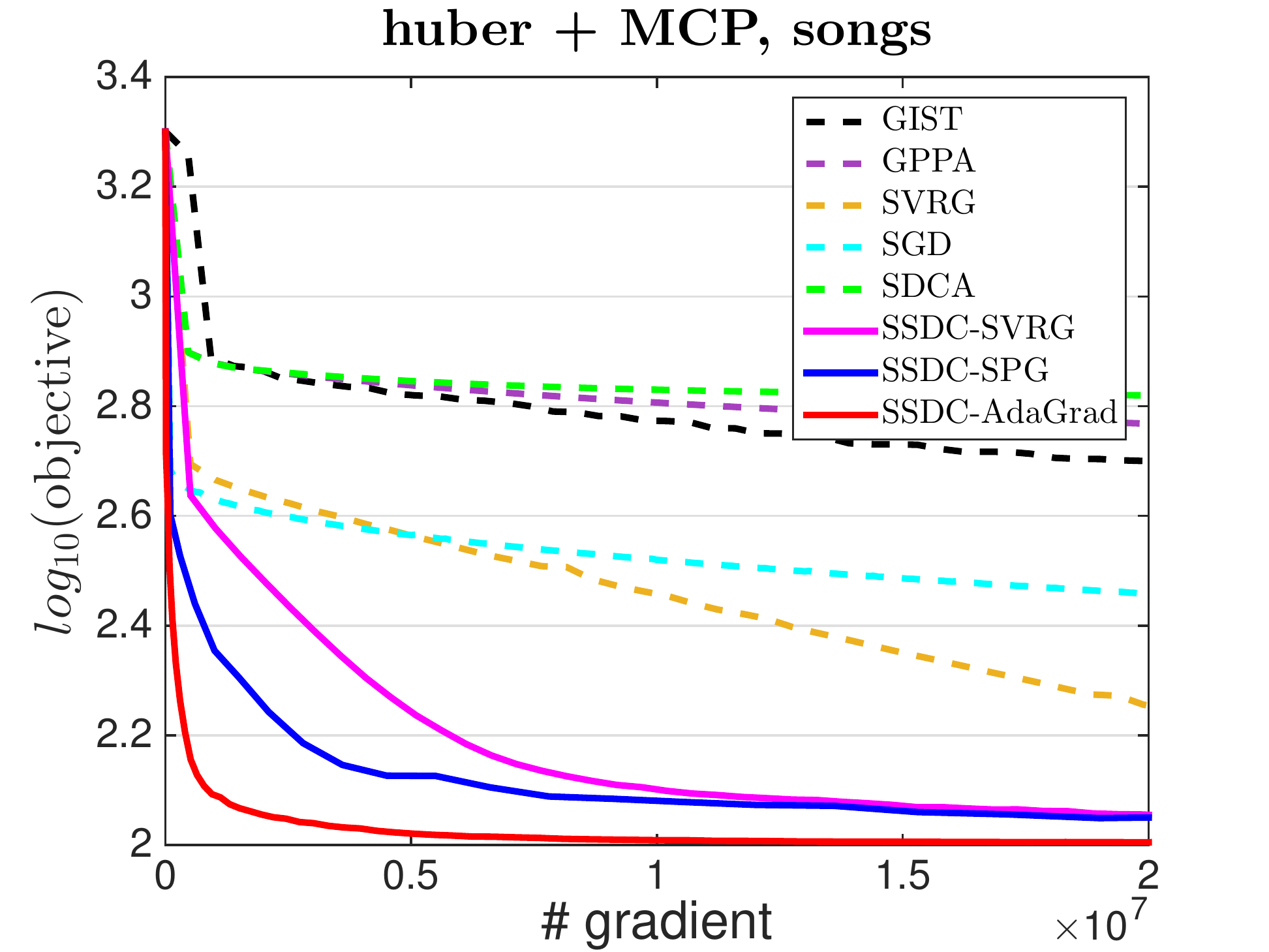}}
\caption{Learning with DC regularizers on different datasets for classification and regression.}\label{fig01}
\end{figure}
\begin{figure}[t]
\centering
{\includegraphics[scale=.33]{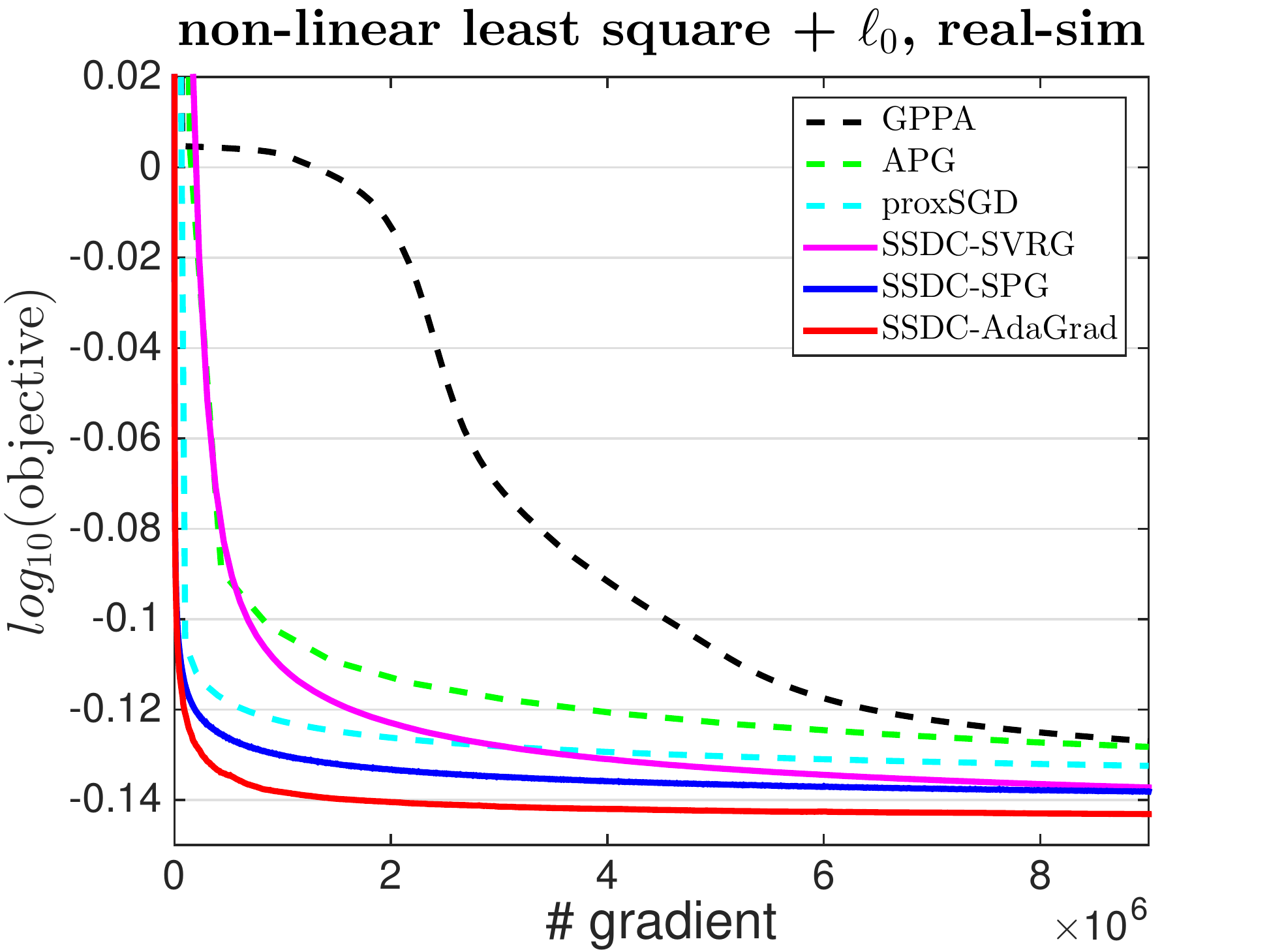}}
{\includegraphics[scale=.33]{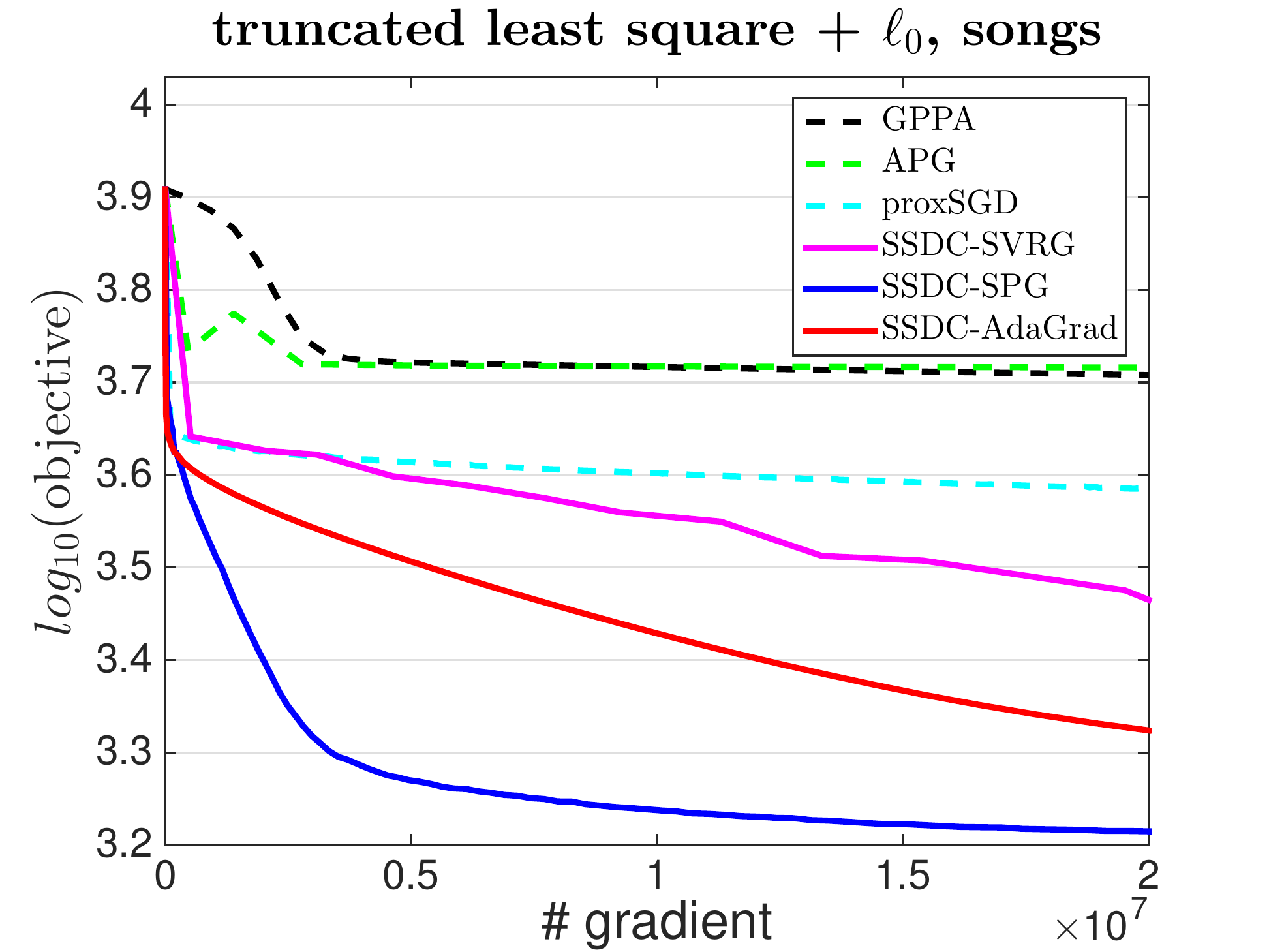}}
\caption{Learning with non-DC regularizers (right two) on different datasets for classification and regression.}\label{fig02}
\end{figure}
\begin{figure}[t]
\centering
{\includegraphics[scale=.33]{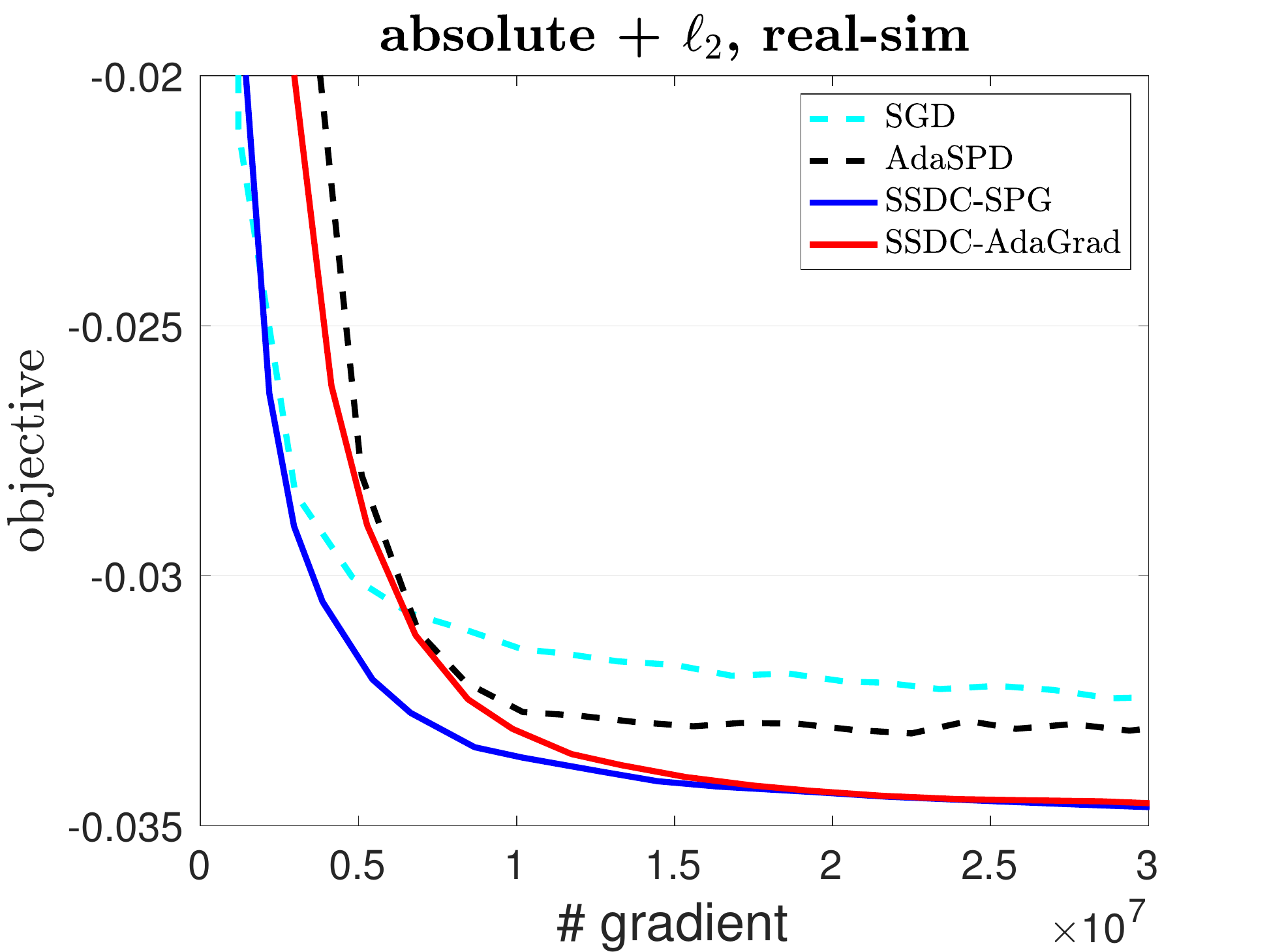}}
{\includegraphics[scale=.33]{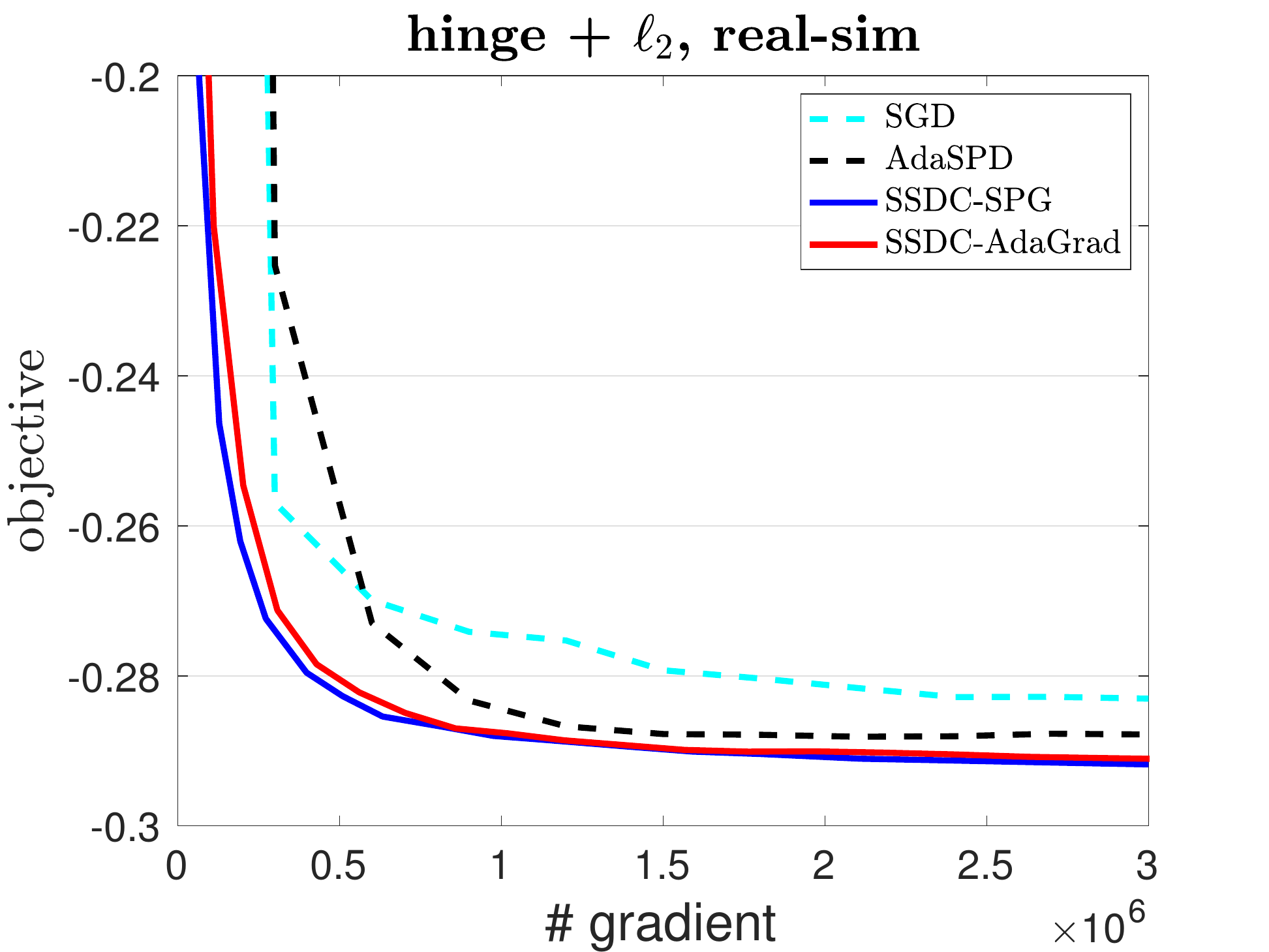}}

{\includegraphics[scale=.33]{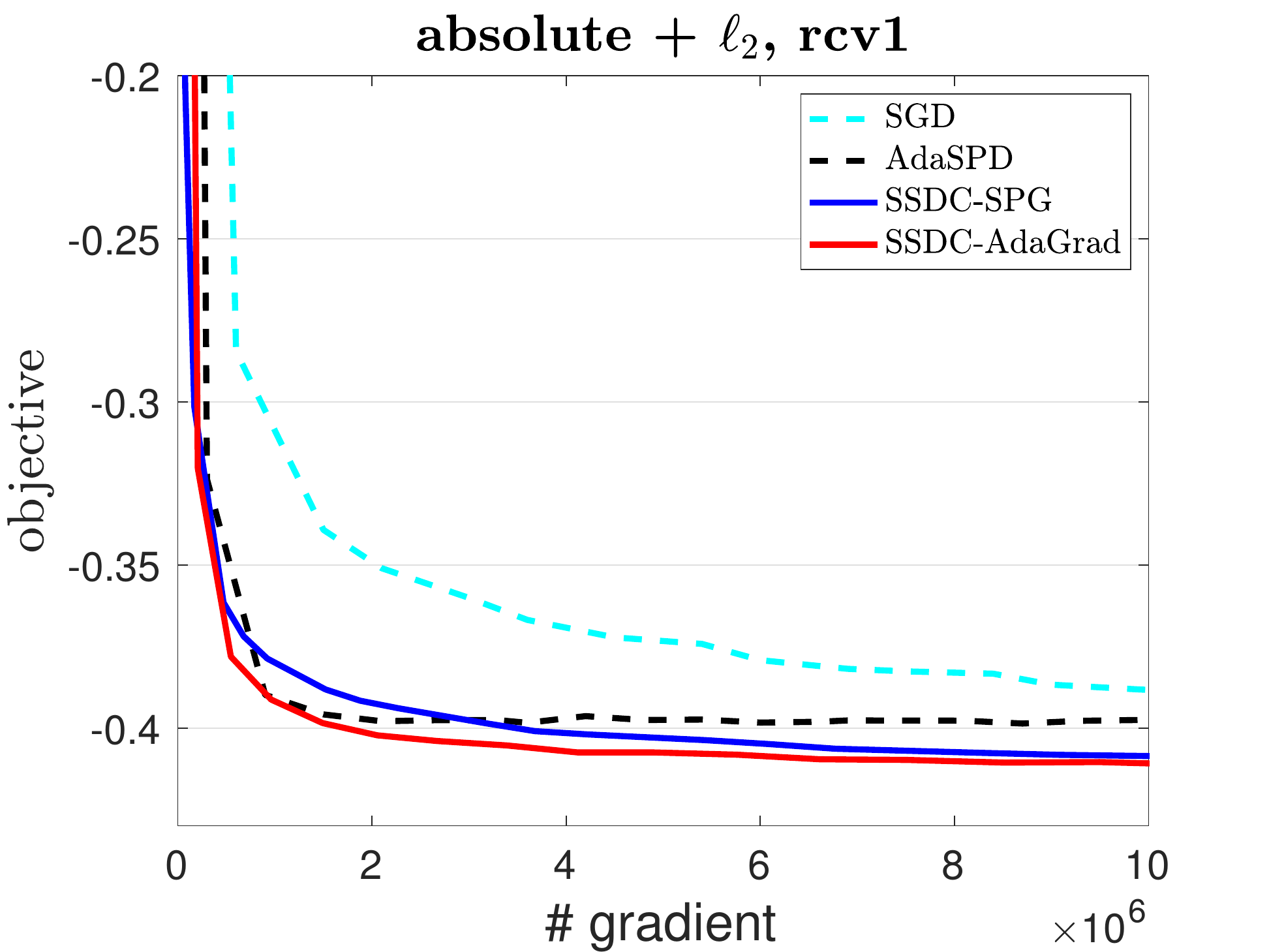}}
{\includegraphics[scale=.33]{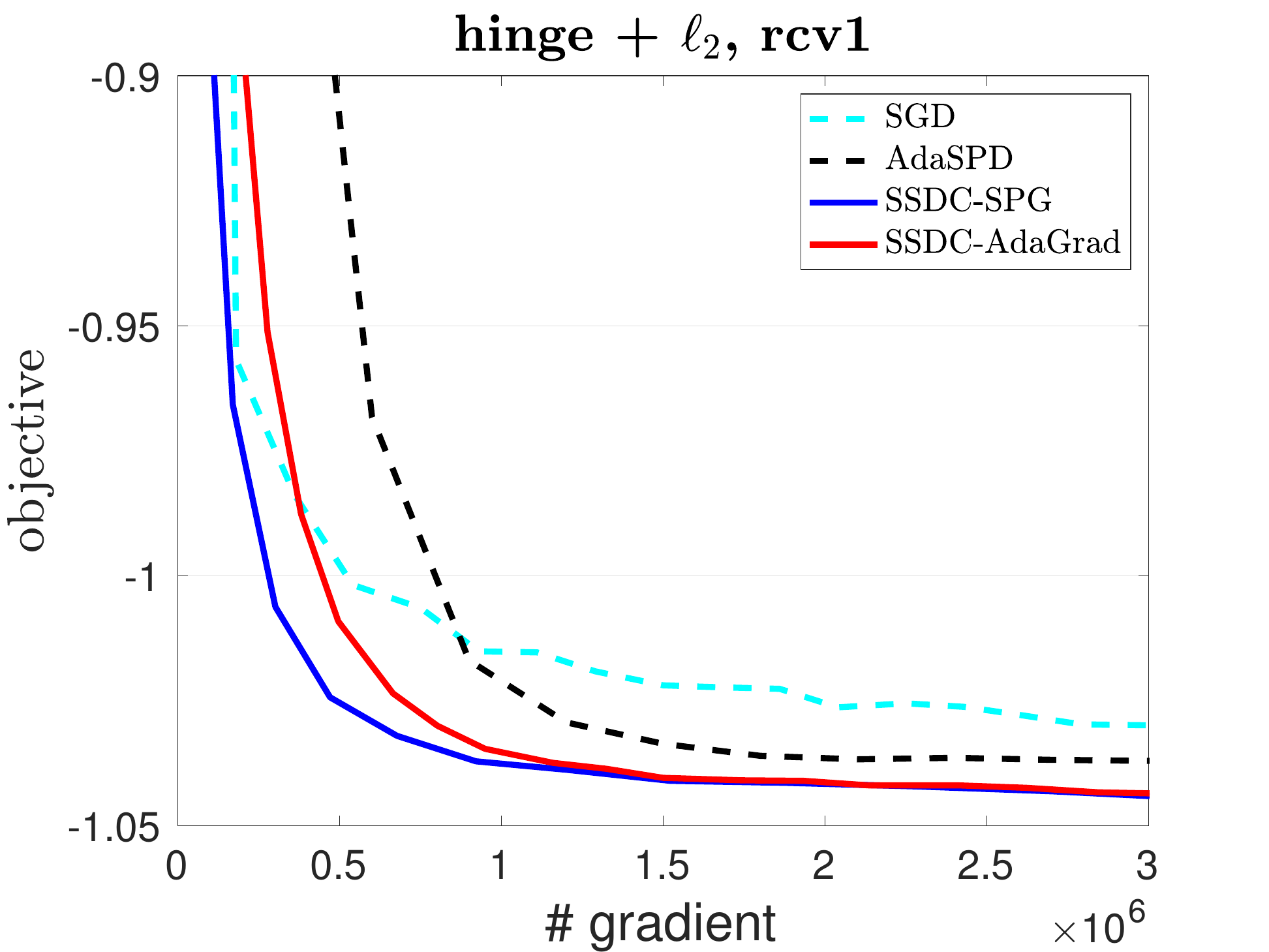}}
\caption{PU learning with different non-smooth loss functions on different datasets.}
\label{fig03}
\end{figure}

Second, we consider minimizing $\ell_0$ regularized non-linear least square loss function $\frac{1}{n}\sum_{i=1}^{n} (y_i - \sigma(\w^\top\x_i))^2 + \lambda\|\w\|_0$ with a sigmod function $\sigma(s) =  \frac{1}{1+e^{-s}}$ for classification and $\ell_0$ regularized truncated least square loss function $\frac{1}{2n}\sum_{i=1}^{n} \alpha \log(1+(y_i - \w^\top\x_i)^2/\alpha) + \lambda\|\w\|_0$~\citep{DBLP:journals/corr/abs-1805-07880} for regression. We compare the proposed algorithms with GPPA, APG~\citep{Li:2015:APG:2969239.2969282} and  proximal version of SGD (proxSGD), where GPPA and APG are  deterministic algorithms. We fix the truncation value as $\alpha = \sqrt{10n}$.  The loss function in these two tasks are smooth and non-convex.  The value of regularization parameter is fixed as $10^{-6}$. For APG, we implement both monotone and non-monotone versions following~\citep{Li:2015:APG:2969239.2969282}, and then the better one is reported. Although the convergence guarantee of proxSGD remains unclear for the considered problems, we still include it for comparison. The results on two data sets are plotted in Figure~\ref{fig02}.

The results of these two experiments indicate that the proposed stochastic algorithms outperform all deterministic baselines (GITS, GPPA, APG) on all tasks, which verify the necessity of using stochastic algorithms on large datasets.  In addition, our algorithms especially SSDC-AdaGrad and SSDC-SPG also converge faster than stochastic algorithms SGD, SDCA, and non-convex SVRG verifying that our stochastic algorithms are more practical for the considered problems. 
We also see that in most cases SSDC-AdaGrad is more effective than SSDC-SPG and SSDC-SVRG. 

Finally, we compare SSDC algorithms with two baselines AdaSPD~\citep{pmlr-v54-nitanda17a} and SGD~\citep{DBLP:journals/corr/abs-1804-07795} for solving two $\ell_2$ regularized positive-unlabeled (PU) learning problems~\citep{du2015convex} with non-smooth losses, i.e., hinge loss and absolute loss. The $\ell_2$ regularization parameter is  set to be $10^{-4}$.  For SGD, we  use the standard stepsize $\eta = \eta_0/\sqrt{t}$~\citep{DBLP:journals/siamjo/GhadimiL13a} with $\eta_0$ tuned. The mini-batch size and the number of iterations of each stage of AdaSPD are simply set as $10^4$. The results on two classification datasets are plotted in Figure~\ref{fig03}, which show that SSDC-SPG and SSDC-AdaGrad outperforms SGD and AdaSPD. 

\section{Conclusions}
In this paper, we have presented stochastic optimization algorithms for solving a broad family of DC functions with convex or non-convex non-smooth regularizer. We consider several stochastic algorithms and presented their convergence results in terms of finding an approximate critical point. The proposed stochastic optimization algorithms and their analysis for solving DC functions are improvements over an existing work by making the algorithms more efficient and practical and the theories more broad and general. For the first time, we provide non-asymptotic convergence for solving non-convex problems with non-smooth non-convex regularizer. 

\appendix 
\section{Proof of Proposition~\ref{prop:sgd}}
\paragraph{\bf Option 1.} Let us first prove the case of smooth $g$. Let $f(\x) = g(\x) - \partial h(\x_1)^{\top}(\x- \x_1)$ and $\hat r (\x) = r(\x) +  \frac{\gamma}{2}\|\x - \x_1\|^2$. Then $f(\x)$ is $L$-smooth and $\hat r(\x)$ is $\gamma$-strongly convex.  A stochastic gradient of $f(\x)$ is given by $\partial f(\x; \xi, \varsigma) = \nabla g(\x; \xi) - \partial h(\x_1, \varsigma) $, which has a variance of $G^2$ according to the assumption. Let $\eta_t= 3/(\gamma(t+1))\leq 1/L$ due to $\gamma\geq 3L$. In our proof, we first need the following lemma, which is attributed to \cite{DBLP:conf/icml/ZhaoZ15}. \begin{lemma}\label{lem:1:prop2}
Under the same assumptions in Proposition~\ref{prop:sgd}, we have
\begin{align*}
\E[f(\x_{t+1}) + \hat r(\x_{t+1}) - f(\x) - \hat r(\x)] \leq & \frac{\|\x_t - \x \|^2}{2\eta_t} - \frac{\|\x - \x_{t+1} \|^2}{2\eta_t} - \frac{\gamma}{2}\|\x - \x_{t+1}\|^2  + \eta_tG^2.
\end{align*}
\end{lemma}
The proof of this lemma is similar to the analysis to proof of Lemma 1 in~\citep{DBLP:conf/icml/ZhaoZ15}. For completeness, we will include its proof later in this section. 

Let's continue the proof by following Lemma~\ref{lem:1:prop2} and letting $w_t = t$, then
\begin{align*}
&\sum_{t=1}^Tw_{t+1}(f(\x_{t+1}) + \hat r(\x_{t+1}) - f(\x) - \hat r(\x))\\
& \leq\sum_{t=1}^T\left(\frac{w_{t+1}}{2\eta_t}\|\x - \x_t\|^2 - \frac{w_{t+1}}{2\eta_t}\|\x - \x_{t+1}\|^2 -\frac{\gamma w_{t+1}}{2}\|\x - \x_{t+1}\|^2  \right) + \sum_{t=1}^T\eta_t w_{t+1} G^2\\
& \leq \sum_{t=1}^T\bigg(  \frac{w_{t+1}}{2\eta_t}  - \frac{w_{t}}{2\eta_{t-1}}- \frac{\gamma w_{t}}{2}\bigg)\|\x - \x_t\|^2 + \frac{w_1/\eta_0 + \gamma w_1}{2} \|\x - \x_1\|^2   + \sum_{t=1}^T\eta_t w_{t+1} G^2\\
& \leq \frac{4\gamma/3}{2} \|\x - \x_1\|^2  + \sum_{t=1}^T3G^2/\gamma
\end{align*}
where the last inequality uses the fact $ \frac{w_{t+1}}{\eta_t}  - \frac{w_{t}}{\eta_{t-1}}- \gamma w_{t} =\frac{\gamma(t+1)^2}{3} - \frac{\gamma t^2}{3} - \gamma t \leq 0$. 
Then we have
\begin{align*}
f(\xh_T)  +\hat  r(\xh_T) - f(\x) - \hat r(\x)\leq \frac{4\gamma\|\x - \x_1\|^2}{3T(T+3)} + \frac{6G^2}{(T+3)\gamma}.
\end{align*}

\paragraph{\bf Option 2.}  Next, we prove the case when $g$ is non-smooth. First, we need to show that $\x_* = \arg\min F^{\gamma}_{\x_1}(\x)$ in the set $\|\x - \x_1\|\leq 3G/\gamma $. By the optimality condition of $\x_*$ we have
\begin{align*}
(\partial g(\x_*)  + \partial r(\x_*)- \partial h(\x_1)  + \gamma (\x_* - \x_1))^{\top}(\x - \x_*)\geq 0, \forall \x\in\text{dom}(r)
\end{align*}
Plugging $\x = \x_1$ into the above inequality, we have
\begin{align*}
\gamma \|\x_1 - \x_*\|^2\leq 3G\|\x_1 -\x_*\|\Rightarrow \|\x_1 - \x_*\|\leq 3G/\gamma,
\end{align*}
where the first inequality uses Assumption~\ref{ass:1} (ii).
Let us recall the update 
\[
\x_{t+1} =\arg\min_{\|\x - \x_1\|\leq 3G/\gamma} \x^{\top}\partial f(\x_t; \xi_t, \varsigma_t) +\hat r(\x)  + \frac{1}{2\eta_t}\|\x- \x_{t}\|^2
\]
By the optimality condition of $\x_{t+1}$ and the strong convexity of above problem, we have for any $\x\in\Omega=\{\x\in\text{dom}(r); \|\x - \x_1\|\leq 3G/\gamma\}$
\begin{align*}
&\x^{\top}\partial f(\x_t; \xi_t, \varsigma_t) +\hat r(\x) +   \frac{1}{2\eta_t}\|\x- \x_{t}\|^2\\
&\geq \x_{t+1}^{\top}\partial f(\x_t; \xi_t, \varsigma_t) +\hat r(\x_{t+1}) + \frac{1}{2\eta_t}\|\x_{t+1}- \x_{t}\|^2 + \frac{1/\eta_t + \gamma}{2}\|\x - \x_{t+1}\|^2
\end{align*}
Rearranging the terms, we have
\begin{align*}
&(\x_t - \x)^{\top}\partial f(\x_t; \xi_t, \varsigma_t)+\hat r(\x_{t+1})   - \hat r(\x)\\
&\leq  (\x_t - \x_{t+1})^{\top}\partial f(\x_t; \xi_t, \varsigma_t)- \frac{1}{2\eta_t}\|\x_{t+1}- \x_{t}\|^2 +     \frac{1}{2\eta_t}\|\x- \x_{t}\|^2 -  \frac{1/\eta_t + \gamma}{2}\|\x - \x_{t+1}\|^2\\
&\leq \frac{\eta_t \|\partial f(\x_t; \xi_t, \varsigma_t)\|^2}{2}+     \frac{1}{2\eta_t}\|\x- \x_{t}\|^2 -  \frac{1/\eta_t + \gamma}{2}\|\x - \x_{t+1}\|^2
\end{align*}
Taking expectation on both sides, we have
\begin{align*}
&\E[f(\x_{t}) - f(\x)+\hat r(\x_{t+1})   - \hat r(\x)]\\
&\leq \frac{\eta_t\E[ \|\partial f(\x_t; \xi_t, \varsigma_t)\|^2]}{2}+\E\bigg[\frac{1}{2\eta_t}\|\x- \x_{t}\|^2 -  \frac{1/\eta_t + \gamma}{2}\|\x - \x_{t+1}\|^2\bigg].
\end{align*}
Multiplying both sides $w_{t}$ and taking summation over $t=1,\ldots, T$ and expectation, we have
\begin{align*}
&\E\bigg[\sum_{t=1}^Tw_{t}(f(\x_t) - f(\x)+\hat r(\x_{t+1})   - \hat r(\x))\bigg]\\
&\leq \sum_{t=1}^T2G^2w_{t}\eta_t+ \E\bigg[\sum_{t=1}^T \frac{w_{t}}{2\eta_t}\|\x- \x_{t}\|^2 -  \frac{w_{t}/\eta_t + w_{t}\gamma}{2}\|\x - \x_{t+1}\|^2\bigg].
\end{align*}
Thus, 
\begin{align*}
&\E\bigg[\sum_{t=1}^Tw_{t}(f(\x_{t}) - f(\x)+\hat r(\x_{t})   - \hat r(\x))\bigg]\\
&\leq \E\bigg[\sum_{t=1}^Tw_{t}(r(\x_{t}) - r(\x_{t+1}))\bigg]+ \sum_{t=1}^T2G^2w_{t}\eta_t+    \E\bigg[\sum_{t=1}^T \left(\frac{w_{t}}{2\eta_t}-\frac{w_{t-1}/\eta_{t-1} + w_{t-1}\gamma}{2}\right)\|\x- \x_{t}\|^2\bigg]\\
&\leq w_0r(\x_1) - w_{T}r(\x_{T+1}) + \E\bigg[\sum_{t=1}^T(w_{t} - w_{t-1})r(\x_{t})\bigg]+ \sum_{t=1}^T2G^2w_{t}\eta_t \\
   &+\E\bigg[\sum_{t=1}^T \left(\frac{w_{t}}{2\eta_t}-\frac{w_{t-1}/\eta_{t-1} + w_{t-1}\gamma}{2}\right)\|\x- \x_{t}\|^2\bigg]\\
   &\leq  \E\bigg[\sum_{t=1}^T(w_{t} - w_{t-1})(r(\x_{t}) - r(\x_{T+1}))\bigg]+ \sum_{t=1}^T2G^2w_{t}\eta_t \\
   &+\E\bigg[\sum_{t=1}^T \left(\frac{w_{t}}{2\eta_t}-\frac{w_{t-1}/\eta_{t-1} + w_{t-1}\gamma}{2}\right)\|\x- \x_{t}\|^2\bigg]\\
   &\leq  \E\bigg[\sum_{t=1}^T(w_{t} - w_{t-1})G\|\x_t - \x_{T+1}\|\bigg]+ \sum_{t=1}^T2G^2w_{t}\eta_t \\
   &+\E\bigg[\sum_{t=1}^T \left(\frac{w_{t}}{2\eta_t}-\frac{w_{t-1}/\eta_{t-1} + w_{t-1}\gamma}{2}\right)\|\x- \x_{t}\|^2\bigg]\\
&\leq  \E\bigg[\sum_{t=1}^T(w_{t} - w_{t-1})6G^2/\gamma\bigg]+ \sum_{t=1}^T2G^2w_{t}\eta_t +\E\bigg[\sum_{t=1}^T \left(\frac{w_{t}}{2\eta_t}-\frac{w_{t-1}/\eta_{t-1} + w_{t-1}\gamma}{2}\right)\|\x- \x_{t}\|^2\bigg]\\
\end{align*}
Plugging the value $w_t =t, \eta_t = 4/(\gamma t)$, and using the fact $\frac{w_{t}}{2\eta_t}-\frac{w_{t-1}/\eta_{t-1} + w_{t-1}\gamma}{2}\leq 0, \forall t\geq 2$, we have
\begin{align*}
\E\bigg[\sum_{t=1}^Tw_{t}(f(\x_{t}) - f(\x)+\hat r(\x_{t})   - \hat r(\x))\bigg]\leq \frac{w_{T}6G^2}{\gamma} + \frac{8G^2T}{\gamma} + \frac{\gamma\|\x - \x_1\|^2}{8}
\end{align*}
Thus,
\begin{align*}
\E\bigg[f(\xh_T) - f(\x)+\hat r(\xh_T)   - \hat r(\x))\bigg]\leq \frac{28G^2}{\gamma (T+1)}+ \frac{\gamma\|\x - \x_1\|^2}{4T(T+1)}.
\end{align*}

\subsection{Proof of Lemma~\ref{lem:1:prop2}}
First, we need the following lemma, which is attributed to \cite{DBLP:conf/icml/ZhaoZ15}. 
\begin{lemma}\label{lem:2:prop2}
If $\hat r(\x)$ is convex and 
\begin{align*}
\widehat \u =\arg\min_{\x} \x^\top \g_u + \hat r(\x) + \frac{1}{2\eta}\|\x-\z\|^2,~\widehat \v =\arg\min_{\x} \x^\top \g_v + \hat r(\x) + \frac{1}{2\eta}\|\x-\z\|^2, 
\end{align*}
then we have
\begin{align*}
\|\widehat \u - \widehat \v \| \leq \eta\|\g_u - \g_v\|.
\end{align*}
\end{lemma}
\begin{proof}{\bf of Lemma~\ref{lem:2:prop2}.}
The proof can be found in the analysis of Lemma 1 in~\citep{DBLP:conf/icml/ZhaoZ15}, but we still include it here for completeness. 
By the optimilaty of $\widehat\u$ and $\widehat\v$ we have
\begin{align*}
\a := & \frac{\z -  \widehat\u}{\eta} - \g_u \in \partial \hat r(\widehat\u) \\  
\b := & \frac{\z -  \widehat\v}{\eta} - \g_v \in  \partial \hat r(\widehat\v). 
\end{align*}
Since $\hat r(\x)$ is convex, then
\begin{align*}
0 \leq \langle \a - \b, \widehat\u - \widehat\v \rangle =  \frac{1}{\eta}  \langle  \eta \g_v - \eta \g_u + \widehat\v - \widehat\u, \widehat\u - \widehat\v \rangle,
\end{align*}
which implies
\begin{align*}
1/\eta \|\widehat\u - \widehat\v\|^2 \leq \langle   \g_v -  \g_u, \widehat\u - \widehat\v \rangle \leq \|\g_v -  \g_u\| \| \widehat\u - \widehat\v \|.
\end{align*}
Thus
\begin{align*}
  \|\widehat\u - \widehat\v\| \leq \eta\|\g_v -  \g_u\|.
\end{align*}
\end{proof}
Then, let's start the proof of Lemma~\ref{lem:1:prop2}.
Since $f(\x) = g(\x) - \partial h(\x_1)^{\top}(\x- \x_1)$ and $\hat r (\x) = r(\x) +  \frac{\gamma}{2}\|\x - \x_1\|^2$, then $f(\x)$ is $L$-smooth and $\hat r(\x)$ is $\gamma$-strongly convex. Recall that a stochastic gradient of $f(\x)$ is given by $\partial f(\x; \xi, \varsigma) = \nabla g(\x; \xi) - \partial h(\x_1, \varsigma) $, which has a variance of $G^2$ according to the assumption, and $\eta_t= 1/(L(t+1))\leq 1/L$. 
By the convexity of $f(\x)$ and strong convexity of $\hat r(\x)$ we have
\begin{align*}
f(\x) + \hat r(\x) \geq f(\x_t) + \langle \partial f(\x_t), \x - \x_t \rangle + \hat r(\x_{t+1}) + \langle \partial \hat r(\x_{t+1}), \x - \x_{t+1} \rangle + \frac{\gamma}{2}\|\x - \x_{t+1}\|^2.
\end{align*}
By the smoothness of $f(\x)$, we also have
\begin{align*}
f(\x_t)  \geq f(\x_{t+1}) - \langle \partial f(\x_t), \x_{t+1} - \x_t \rangle - \frac{L}{2}\|\x_t - \x_{t+1}\|^2.
\end{align*}
Combing the above two inequalities, we have
\begin{align}\label{lem:1:ineq1}
\nonumber & f(\x_{t+1}) + \hat r(\x_{t+1}) - f(\x) - \hat r(\x) \\
\nonumber \leq & \langle \partial f(\x_t) +  \partial \hat r(\x_{t+1}), \x_{t+1} - \x \rangle   - \frac{\gamma}{2}\|\x - \x_{t+1}\|^2  + \frac{L}{2}\|\x_t - \x_{t+1}\|^2\\
\nonumber = &  \langle \partial f(\x_t) - \partial f(\x_t;\xi_t,\varsigma_t), \x_{t+1} - \x \rangle  +  \frac{1}{\eta_t} \langle \x_t - \x_{t+1}, \x_{t+1} - \x \rangle\\
\nonumber &  - \frac{\gamma}{2}\|\x - \x_{t+1}\|^2  + \frac{L}{2}\|\x_t - \x_{t+1}\|^2\\
\nonumber = &  \langle \partial f(\x_t) - \partial f(\x_t;\xi_t,\varsigma_t), \x_{t+1} - \x \rangle  +  \frac{\|\x_t - \x \|^2}{2\eta_t} - \frac{\|\x_t - \x_{t+1} \|^2}{2\eta_t} - \frac{\|\x_{t+1} - \x \|^2}{2\eta_t}\\
&  - \frac{\gamma}{2}\|\x - \x_{t+1}\|^2  + \frac{L}{2}\|\x_t - \x_{t+1}\|^2,
\end{align}
where the last second equality is due to the update and optimilaty of $\x_{t+1}$ (Option 1 in Algorithm~\ref{alg:sgd}); 
the last equality uses the fact that $2 \langle \x - \y, \y - \z \rangle = \|\x-\z\|^2 - \|\x-\y\|^2 - \|\y-\z\|^2$. To deal with the term $\langle \partial f(\x_t) - \partial f(\x_t;\xi_t,\varsigma_t), \x_{t+1} - \x \rangle$, we define $\widehat \x_{t+1} = \arg\min_{\x} \x^\top \partial f(\x_t) + \hat r(\x) + \frac{1}{2\eta_t}\|\x-\x_t\|^2$, which is independent of $\partial f(\x_t;\xi_t,\varsigma_t)$. Taking expectation over $\xi_t$ and $\varsigma_t$ over this term we get
\begin{align*}
&\E[\langle \partial f(\x_t) - \partial f(\x_t;\xi_t,\varsigma_t), \x_{t+1} - \x \rangle] \\
=& \E[\langle \partial f(\x_t) - \partial f(\x_t;\xi_t,\varsigma_t), \x_{t+1} - \widehat \x_{t+1} + \widehat \x_{t+1}- \x \rangle] \\
=& \E[\langle \partial f(\x_t) - \partial f(\x_t;\xi_t,\varsigma_t), \x_{t+1} - \widehat \x_{t+1}  \rangle] + \E[\langle \partial f(\x_t) - \partial f(\x_t;\xi_t,\varsigma_t), \widehat \x_{t+1}- \x \rangle]\\
=& \E[\langle \partial f(\x_t) - \partial f(\x_t;\xi_t,\varsigma_t), \x_{t+1} - \widehat \x_{t+1}  \rangle] \\
\leq & \E[\| \partial f(\x_t) - \partial f(\x_t;\xi_t,\varsigma_t)\|  \|\x_{t+1} - \widehat \x_{t+1}  \| ] \\
\leq & \eta_t \E[\| \partial f(\x_t) - \partial f(\x_t;\xi_t,\varsigma_t)\|^2 ] \leq  \eta_t G^2,
\end{align*}
where the third equality is due to $\E[\langle \partial f(\x_t) - \partial f(\x_t;\xi_t,\varsigma_t), \widehat \x_{t+1}- \x \rangle | \x_t] = 0$; the last third inequality uses Cauchy-Schwartz inequality; the last second inequality is due to Lemma~\ref{lem:2:prop2}; the last inequality uses Assumption~\ref{ass:1} (i).
With above inequality, taking the expectation on both sides of (\ref{lem:1:ineq1}) and using the fact that $\eta_t \le 1/L$, we get
\begin{align*}
\E[f(\x_{t+1}) + \hat r(\x_{t+1}) - f(\x) - \hat r(\x)] \leq & \frac{\|\x_t - \x \|^2}{2\eta_t} - \frac{\|\x - \x_{t+1} \|^2}{2\eta_t} - \frac{\gamma}{2}\|\x - \x_{t+1}\|^2  + \eta_tG^2.
\end{align*}

\section{Proof of Theorem~\ref{thm:sgd}}
Let us use the same notations as in the proof of Theorem~\ref{thm:3} and prove the case of smooth $g$. By Proposition~\ref{prop:sgd}, we have 
\begin{align*}
\E[ F_k(\x_{k+1}) - F_k(\z_k)]\leq \frac{4\gamma\|\x - \x_1\|^2}{3T_k(T_k+3)} + \frac{6G^2}{(T_k+3)\gamma}. 
\end{align*}
To continue the analysis, we have
\begin{align*}
\E[f_k(\x_{k+1})   + \frac{\gamma}{2}\|\x_{k+1} - \x_k\|^2]&\leq F_k(\z_k) + \frac{4\gamma\|\x_k - \z_k\|^2}{3T_k(T_k+3)} + \frac{6G^2}{(T_k+3)\gamma}\\
&\leq F_k(\x_k) - \frac{\gamma}{2}\|\x_k - \z_k\|^2 + \frac{4\gamma \|\x_k - \z_k\|^2}{3T_k(T_k+3)} + \frac{6G^2}{(T_k+3)\gamma}\\
& \leq  g(\x_k)  + r(\x_k)  +  \frac{6G^2}{(T_k+3)\gamma}
\end{align*}
where we use $F_k(\x_k)\geq F_k(\z_k) + \frac{\gamma}{2}\|\x_k - \z_k\|^2$ due to the strong convexity of $F(\x)$, and $T_k\geq 2$.
On the other hand, we have that 
\begin{align*}
    \|\x_{k+1}-\x_{k}\|^2 =& \|\x_{k+1}-\z_{k}+\z_{k}-\x_{s}\|^2\\
    =& \|\x_{k+1}-\z_{k}\|^2+\|\z_{k}-\x_{k}\|^2 + 2\langle \x_{k+1}-\z_{k}, \z_{k}- \x_{k}\rangle\\
    \geq& (1-\hat\alpha^{-1})\|\x_{k+1}-\z_{k}\|^2 + (1-\hat\alpha)\|\x_{k}-\z_{k}\|^2
\end{align*}
where the inequality follows from the Young's inequality with $0<\hat\alpha<1$. Thus we have that
\begin{align*}
 \E\bigg[\frac{\gamma(1-\hat\alpha)}{2}\|\z_{k} - \x_k\|^2\bigg]&\leq \E[g(\x_k) + r(\x_k)  - f_k(\x_{k+1})] + \frac{\gamma(\hat\alpha^{-1}-1)}{2}\E[\|\x_{k+1} - \z_k\|^2] \\
 &   + \frac{6G^2}{(T_k+3)\gamma}.
\end{align*}
On the other hand, by the convexity of $h(\cdot)$ we have
\begin{align*}
&\E[g(\x_k)  + r(\x_k) - f_k(\x_{k+1})]\\
 &\leq \E[g(\x_k) + r(\x_k) - g(\x_{k+1}) - r(\x_{k+1}) + \partial h(\x_k)^{\top}(\x_{k+1} - \x_k)]\\
 &\leq \E[g(\x_k)  + r(\x_k )- g(\x_{k+1}) -r(\x_{k+1}) + h(\x_{k+1}) - h(\x_k)]  \\
 &=  \E[F(\x_k) - F(\x_{k+1})],
\end{align*}
By the strong convexity of $F_k(\x)$, we also have 
\begin{align*}
\frac{\gamma}{2}\E[\|\x_{k+1} - \z_k\|^2]\leq \E[F_k(\x_{k+1}) - F_k(\z_k)]\leq \frac{4\gamma\|\x_k - \z_k\|^2}{3T_k(T_k+3)}  + \frac{6G^2}{(T_k+3)\gamma}
\end{align*}
Then we have
\begin{align*}
 \E\bigg[\frac{\gamma(1-\hat\alpha)}{2}\|\z_{k} - \x_k\|^2\bigg]&\leq \E[F(\x_k) - F(\x_{k+1})] + (\hat\alpha^{-1}-1)\left( \frac{4\gamma\|\x_k - \z_k\|^2}{3T_k(T_k+3)}  + \frac{6G^2}{(T_k+3)\gamma}\right) \\
 &  +\frac{6G^2}{(T_k+3)\gamma}
\end{align*}
Let $\hat\alpha=1/2$, we have
\begin{align*}
 \E\bigg[\frac{\gamma}{4}\|\z_{k} - \x_k\|^2\bigg]&\leq \E[F(\x_k) - F(\x_{k+1})] + \frac{4\gamma\|\x_k - \z_k\|^2}{3T_k(T_k+3)}+  \frac{12G^2}{(T_k+3)\gamma}\\
 & \leq \E[F(\x_k) - F(\x_{k+1})] + \frac{\gamma\|\x_k - \z_k\|^2}{8}+  \frac{12G^2}{(T_k+3)\gamma}
\end{align*}
where we use the fac  $T_k\geq 3$.  It then gives us 
\begin{align*}
 \E\bigg[\frac{\gamma}{8}\|\z_{k} - \x_k\|^2\bigg]&\leq \E[F(\x_k) - F(\x_{k+1})] + \frac{12G^2}{(T_k+3)\gamma}
\end{align*}
By setting $T_k = 3Lk/\gamma+3$, we have
\begin{align*}
 \E\bigg[\frac{\gamma}{8}\|\z_{k} - \x_k\|^2\bigg]&\leq \E[F(\x_k) - F(\x_{k+1})] + \frac{4G^2}{kL}
\end{align*}
Following similar analysis to the proof of Theorem~\ref{thm:3}, we can finish the proof. 
For completeness, we include the remaining analysis here. Multiplying both sides by $w_k = k^\alpha$ and taking summation over $k=1,\ldots, K$, we have
\begin{align}\label{eqn:c2}
 \E\bigg[\frac{\gamma}{8}\sum_{k=1}^Kw_k\|\z_{k} - \x_k\|^2\bigg]&\leq \E\bigg[ \sum_{k=1}^Kw_k(F(\x_k) - F(\x_{k+1}))\bigg] +\sum_{k=1}^Kw_k \frac{4G^2}{kL},
\end{align}
Similar to proof of of Theorem~\ref{thm:3}, we have
\begin{align*}
 \E\bigg[ \sum_{k=1}^Kw_k(F(\x_k) - F(\x_{k+1}))\bigg] \leq  \sum_{k=1}^{K}(w_k - w_{k-1})\E[(F(\x_k) - \min_{\x}F(\x))]\leq \Delta w_{K}
\end{align*}
Then, by $\sum_{k=1}^Kk^{\alpha}\geq \int_0^{K}x^\alpha d x= \frac{K^{\alpha+1}}{\alpha+1}$ and $\sum_{k=1}^Kk^{\alpha-1}\leq K^\alpha,~(\alpha\geq 1)$, we have
\begin{align*}
 \E\bigg[\frac{\gamma}{8}\|\z_{\tau} - \x_\tau\|^2\bigg]\leq  \frac{\Delta K^\alpha}{\sum_{k=1}^{K} k^\alpha} + \frac{4G^2 \sum_{k=1}^{K} k^{\alpha-1} }{L\sum_{k=1}^{K} k^\alpha} \leq \frac{\Delta (\alpha+1)}{K} + \frac{4G^2(\alpha+1)}{LK},
\end{align*}
which can complete the proof by multiplying both sides by $8\gamma$.

Similarly, we can prove the case of non-smooth $g(\x)$. For completeness, we include the details here. By Proposition~\ref{prop:sgd} we have
\begin{align*}
\E[f_k(\x_{k+1})   + \frac{\gamma}{2}\|\x_{k+1} - \x_k\|^2]&\leq F_k(\z_k) + \frac{\gamma\|\x_k - \z_k\|^2}{4T_k(T_k+1)} + \frac{28G^2}{\gamma(T_k+1)}\\
&\leq F_k(\x_k) - \frac{\gamma}{2}\|\x_k - \z_k\|^2 + \frac{\gamma\|\x_k - \z_k\|^2}{4T_k(T_k+1)} + \frac{28G^2}{\gamma(T_k+1)}\\
& =  g(\x_k)  + r(\x_k)  +  \frac{28G^2}{\gamma(T_k+1)}
\end{align*}
where we use $F_k(\x_k)\geq F_k(\z_k) + \frac{\gamma}{2}\|\x_k - \z_k\|^2$ due to the strong convexity of $F(\x)$, and $T_k\geq 1$.
On the other hand, we have that 
\begin{align*}
    \|\x_{k+1}-\x_{k}\|^2 =& \|\x_{k+1}-\z_{k}+\z_{k}-\x_{s}\|^2\\
    =& \|\x_{k+1}-\z_{k}\|^2+\|\z_{k}-\x_{k}\|^2 + 2\langle \x_{k+1}-\z_{k}, \z_{k}- \x_{k}\rangle\\
    \geq& (1-\hat\alpha^{-1})\|\x_{k+1}-\z_{k}\|^2 + (1-\hat\alpha)\|\x_{k}-\z_{k}\|^2
\end{align*}
where the inequality follows from the Young's inequality with $0<\hat\alpha<1$. Thus we have that
\begin{align*}
 \E\bigg[\frac{\gamma(1-\hat\alpha)}{2}\|\z_{k} - \x_k\|^2\bigg]&\leq \E[g(\x_k) + r(\x_k)  - f_k(\x_{k+1})] + \frac{\gamma(\hat\alpha^{-1}-1)}{2}\E[\|\x_{k+1} - \z_k\|^2] \\
 &   + \frac{28G^2}{\gamma(T_k+1)}.
\end{align*}
On the other hand, by the convexity of $h(\cdot)$ we have
\begin{align*}
&\E[g(\x_k)  + r(\x_k) - f_k(\x_{k+1})]\\
 &\leq \E[g(\x_k) + r(\x_k) - g(\x_{k+1}) - r(\x_{k+1}) + \partial h(\x_k)^{\top}(\x_{k+1} - \x_k)]\\
 &\leq \E[g(\x_k)  + r(\x_k )- g(\x_{k+1}) -r(\x_{k+1}) + h(\x_{k+1}) - h(\x_k)]  \\
 &=  \E[F(\x_k) - F(\x_{k+1})],
\end{align*}
By the strong convexity of $F_k(\x)$, we also have 
\begin{align*}
\frac{\gamma}{2}\E[\|\x_{k+1} - \z_k\|^2]\leq \E[F_k(\x_{k+1}) - F_k(\z_k)]\leq  \frac{\gamma\|\x_k - \z_k\|^2}{4T_k(T_k+1)} + \frac{28G^2}{\gamma(T_k+1)}
\end{align*}
Then we have
\begin{align*}
 \E\bigg[\frac{\gamma(1-\hat\alpha)}{2}\|\z_{k} - \x_k\|^2\bigg]&\leq \E[F(\x_k) - F(\x_{k+1})] + (\hat\alpha^{-1}-1)\left(  \frac{\gamma\|\x_k - \z_k\|^2}{4T_k(T_k+1)} + \frac{28G^2}{\gamma(T_k+1)}\right) \\
 &  +  \frac{28G^2}{\gamma(T_k+1)}
\end{align*}
Let $\hat\alpha=1/2$, we have
\begin{align*}
 \E\bigg[\frac{\gamma}{4}\|\z_{k} - \x_k\|^2\bigg]&\leq \E[F(\x_k) - F(\x_{k+1})] + \frac{\gamma\|\x_k - \z_k\|^2}{4T_k(T_k+1)} + \frac{56G^2}{\gamma(T_k+1)}\\
 & \leq \E[F(\x_k) - F(\x_{k+1})] + \frac{\gamma\|\x_k - \z_k\|^2}{8}+ \frac{56G^2}{\gamma(T_k+1)}
\end{align*}
where we use the fact $T_k \geq 1$.  It then gives us 
\begin{align*}
 \E\bigg[\frac{\gamma}{8}\|\z_{k} - \x_k\|^2\bigg]&\leq \E[F(\x_k) - F(\x_{k+1})] + \frac{56G^2}{\gamma(T_k+1)}
\end{align*}
By setting $T_k = k/\gamma + 1$, we have
\begin{align*}
 \E\bigg[\frac{\gamma}{8}\|\z_{k} - \x_k\|^2\bigg]&\leq \E[F(\x_k) - F(\x_{k+1})] + \frac{56 G^2}{k}
\end{align*}
Multiplying both sides by $w_k = k^\alpha$ and taking summation over $k=1,\ldots, K$, we have
\begin{align}\label{eqn:c2}
 \E\bigg[\frac{\gamma}{8}\sum_{k=1}^Kw_k\|\z_{k} - \x_k\|^2\bigg]&\leq \E\bigg[ \sum_{k=1}^Kw_k(F(\x_k) - F(\x_{k+1}))\bigg] +\sum_{k=1}^Kw_k  \frac{56 G^2}{k},
\end{align}
Similar to proof of of Theorem~\ref{thm:3}, we have
\begin{align*}
 \E\bigg[ \sum_{k=1}^Kw_k(F(\x_k) - F(\x_{k+1}))\bigg] \leq  \sum_{k=1}^{K}(w_k - w_{k-1})\E[(F(\x_k) - \min_{\x}F(\x))]\leq \Delta w_{K}
\end{align*}
Then, by $\sum_{k=1}^Kk^{\alpha}\geq \int_0^{K}x^\alpha d x= \frac{K^{\alpha+1}}{\alpha+1}$ and $\sum_{k=1}^Kk^{\alpha-1}\leq K^\alpha,~(\alpha\geq 1)$, we have
\begin{align*}
 \E\bigg[\frac{\gamma}{8}\|\z_{\tau} - \x_\tau\|^2\bigg]\leq  \frac{\Delta K^\alpha}{\sum_{k=1}^{K} k^\alpha} + \frac{56 G^2 \sum_{k=1}^{K} k^{\alpha-1} }{\sum_{k=1}^{K} k^\alpha} \leq \frac{\Delta (\alpha+1)}{K} +   \frac{56 G^2(\alpha+1)}{K},
\end{align*}
which can complete the proof by multiplying both sides by $8\gamma$.

\section{Proof of Proposition~\ref{lem:adagrad}}
The convergence analysis of using AdaGrad is build on the following proposition about the convergence AdaGrad for minimizing $F^\gamma_{\x}$, which is attributed to~\cite{SadaGrad18}.
For completeness, we include the proof here. First, we need to show that $\x_* = \arg\min F^{\gamma}_{\x_1}(\x)$ in the set $\|\x - \x_1\|\leq \frac{2G+G_r}{\gamma} $. By the optimality condition of $\x_*$ we have
\begin{align*}
(\partial g(\x_*)  + \partial r(\x_*)- \partial h(\x_1)  + \gamma (\x_* - \x_1))^{\top}(\x - \x_*)\geq 0, \forall \x\in\text{dom}(r)
\end{align*}
Plugging $\x = \x_1$ into the above inequality, we have
\begin{align*}
\gamma \|\x_1 - \x_*\|^2\leq (2G+G_r)\|\x_1 -\x_*\|\Rightarrow \|\x_1 - \x_*\|\leq \frac{2G+G_r}{\gamma},
\end{align*}
where the first inequality uses Assumption~\ref{ass:new}.

Let $\psi_0(\x) = 0$ and $\|\x\|_{H}=\sqrt{\x^{\top}H\x}$, then we can see that $\psi_{t+1}(\x)\geq \psi_t(\x)$ for any $t\geq 0$. 
Let $f(\x) = g(\x) - \partial h(\x_1)^\top(\x-\x_1)$ and $\hat r(\x) = r(\x) + \frac{\gamma}{2}\|\x-\x_1\|^2$. Then $f(\x)$ is convex and $\hat r(\x)$ is $\gamma$-strongly convex.
Let  $\z_t  = \sum_{\tau=1}^t\g_t$, $\Delta_t  = (\partial f(\x_t)- \g_t)^{\top}(\x_t - \x)$ and 
\begin{align*}
\psi^*_t(g) = \sup_{\x\in\Omega}g^{\top}\x - \frac{1}{\eta}\psi_t(\x) - t \hat r(\x)
\end{align*}
By the convexity of $f(\x)$, and then taking the summation over all iterations, we get
\begin{align}\label{lemma1:ieq1}
\nonumber &\sum_{t=1}^T(f(\x_t) - f(\x) + \hat r(\x_t) - \hat r(\x)) \\
\nonumber\leq& \sum_{t=1}^T(\partial f(\x_t)^{\top}(\x_t - \x)+ \hat r(\x_t) - \hat r(\x)) =  \sum_{t=1}^T\g_t^{\top}(\x_t - \x) + \sum_{t=1}^T\Delta_t + \sum_{t=1}^T(\hat r(\x_t) - \hat r(\x))\\
\nonumber =&\sum_{t=1}^T\g_t^{\top}\x_t - \sum_{t=1}^T\g_t^{\top}\x - \frac{1}{\eta}\psi_T(\x) - T\hat r(\x) + \frac{1}{\eta}\psi_T(\x) +  \sum_{t=1}^T\Delta_t + \sum_{t=1}^T\hat r(\x_t) \\
\nonumber \leq& \sum_{t=1}^T\g_t^{\top}\x_t + \sup_{\x\in\Omega}\left\{ - \sum_{t=1}^T\g_t^{\top}\x  - \frac{1}{\eta}\psi_T(\x) - T\hat r(\x) \right\} +  \frac{1}{\eta}\psi_T(\x)   + \sum_{t=1}^T\Delta_t+ \sum_{t=1}^T\hat r(\x_t)\\
= & \sum_{t=1}^T\g_t^{\top}\x_t + \psi_T^*(-\z_T)  + \frac{1}{\eta}\psi_T(\x) +  \sum_{t=1}^T\Delta_t+ \sum_{t=1}^T\hat r(\x_t)
\end{align}
On the other hand, 
\begin{align*}
\psi_T^*( - \z_T)  &=  - \sum_{t=1}^T\g_t^{\top}\x_{T+1} - \frac{1}{\eta}\psi_T(\x_{T+1}) - T\hat r(\x_{T+1}) \\
&\leq -\sum_{t=1}^T\g_t^{\top}\x_{T+1} - \frac{1}{\eta}\psi_{T-1}(\x_{T+1}) - (T-1)\hat r(\x_{T+1})- \hat r(\x_{T+1}) \\
&\leq \sup_{\x\in\Omega}\left \{- \z_T^{\top}\x - \frac{1}{\eta}\psi_{T-1}(\x) - (T-1)\hat r(\x)\right\} - \hat r(\x_{T+1}) \\
& =  \psi_{T-1}^*(-\z_T) - \hat r(\x_{T+1}) \\
&\leq \psi_{T-1}^*(-\z_{T-1}) - \g_T^{\top}\nabla\psi_{T-1}^*(-\z_{T-1}) + \frac{\eta}{2}\|\g_T\|_{\psi^*_{T-1}}^2- \hat r(\x_{T+1})
\end{align*}
where the last inequality is due to $\psi_t(\x)$ is $1$-strongly convex w.r.t $\|\cdot\|_{\psi_t} = \|\cdot\|_{H_t}$ and consequentially $\psi_t^*(\x)$ is $\eta$-smooth w.r.t. $\|\cdot\|_{\psi^*_t} = \|\cdot\|_{H_t^{-1}}$.
Then 
\begin{align*}
 &\sum_{t=1}^T\g_t^{\top}\x_t + \psi_T^*(-\z_T)\\
 \leq & \sum_{t=1}^T\g_t^{\top}\x_t + \psi_{T-1}^*(-\z_{T-1}) - \g_T^{\top}\nabla\psi_{T-1}^*(-\z_{T-1}) + \frac{\eta}{2}\|\g_T\|_{\psi^*_{T-1}}^2 - \hat r(\x_{T+1}) \\
 =& \sum_{t=1}^{T-1}\g_t^{\top}\x_t + \psi^*_{T-1}(-\z_{T-1})  + \frac{\eta}{2}\|\g_T\|_{\psi^*_{T-1}}^2 - \hat r(\x_{T+1})
\end{align*}
By repeating this process, we get
\begin{align}\label{lemma1:ieq2}
\nonumber \sum_{t=1}^T\g_t^{\top}\x_t + \psi_T^*(-\z_T) &\leq \psi_0^*(-\z_0) +\frac{\eta}{2} \sum_{t=1}^T\|\g_t\|^2_{\psi^*_{t-1}} - \sum_{t=1}^{T}\hat r(\x_{t+1})\\
&=  \frac{\eta}{2}\sum_{t=1}^T\|\g_t\|^2_{\psi^*_{t-1}} -  \sum_{t=1}^{T}\hat r(\x_{t+1})
\end{align}
Plugging the inequality (\ref{lemma1:ieq2}) in the inequality (\ref{lemma1:ieq1}), we have
\begin{align*}
&\sum_{t=1}^T(f(\x_t) - f(\x) + \hat r(\x_t) - \hat r(\x))
\leq \frac{1}{\eta}\psi_T(\x) +\frac{\eta}{2}\sum_{t=1}^T\|\g_t\|^2_{\psi^*_{t-1}} +\sum_{t=1}^T\Delta_t+ \hat r(\x_1) - \hat r(\x_{T+1}).
\end{align*}
It is known from the analysis in~\citep{duchi2011adaptive} that 
\begin{align*}
\sum_{t=1}^T\|\g_t\|_{\psi^*_{t-1}}^2\leq 2\sum_{i=1}^d\|\g_{1:T, i}\|_2
\end{align*}
Thus
\begin{align*}
&\sum_{t=1}^{T}(f(\x_t) - f(\x) + \hat r(\x_t) - \hat r(\x))\\
&\leq\frac{2 G\|\x - \x_1\|_2^2}{2\eta} + \frac{(\x - \x_1)^{\top}\diag(\s_T)(\x-\x_1)}{2\eta} + \eta\sum_{i=1}^d\|g_{1:T, i}\|_2 + \sum_{t=1}^T\Delta_t +\hat r(\x_1) - \hat r(\x_{T+1}) \\
&\leq \frac{2 G  +\max_i \|g_{1:T,i}\|_2}{2\eta}\|\x - \x_1\|_2^2  + \eta\sum_{i=1}^d\|g_{1:T,i}\|_2  +\sum_{t=1}^T\Delta_t + (\partial \hat r(\x_{1}))^\top( \x_1-\x_{T+1}) \\
&\leq \frac{2 G  +\max_i \|g_{1:T,i}\|_2}{2\eta}\|\x - \x_1\|_2^2  + \eta\sum_{i=1}^d\|g_{1:T,i}\|_2  +\sum_{t=1}^T\Delta_t + G_r\| \x_1 - \x_{T+1} \|_2
\end{align*}
where the last inequality holds by using the fact that $\|\partial \hat r(\x_{1})\|  = \| \partial r(\x_{1})\|\leq G_r $. 
Dividing by $T$ and taking the expecation on both sides, then by using the convexity of $f(\x) + \hat r(\x)$ and $\E[\sum_{t=1}^T\Delta_t/T ] = 0 $ according to the stopping time argument~\citep{SadaGrad18}[Lemma 1, Supplement] we get
\begin{align*}
& \E[ F^\gamma_{\x_1}(\xh_T)-F^\gamma_{\x_1}(\x_*)] \\
\leq &\E\bigg[\frac{2 G  +\max_i \|g_{1:T,i}\|_2}{2\eta T}\|\x - \x_1\|_2^2  + \frac{\eta}{T}\sum_{i=1}^d\|g_{1:T,i}\|_2  + \frac{G_r\| \x_1 - \x_{T+1} \|_2}{T}\bigg]\\
\leq & \frac{1}{2aM\eta}\|\x_1-\x_*\|^2 +\frac{(a+1)\eta}{M},
\end{align*}
where the last inequality is due to $T\geq M \max\{a(2G+ \max_i\|g_{1:T,i}\|),   \sum_{i=1}^d\|g_{1:T,i}\|/a, G_r\|\x_1 - \x_{T+1}\|/\eta \}$.

\section{Proof of Theorem~\ref{thm:nadagrad}}
Let us use the same notations as in the proof of Theorem~\ref{thm:3}. By Proposition~\ref{lem:adagrad}, we have 
\begin{align*}
\E[ F_k(\x_{k+1}) - F_k(\z_k)] \leq \frac{ \|\x_k - \z_k\|^2}{2aM_k \eta_k } + \frac{(a+1)\eta_k }{M_k}
\end{align*}
To continue the analysis, we have
\begin{align*}
\E[f_k(\x_{k+1})   + \frac{\gamma}{2}\|\x_{k+1} - \x_k\|^2]&\leq F_k(\z_k) + \frac{\|\x_k - \z_k\|^2}{2aM_k \eta_k } + \frac{(a+1)\eta_k }{M_k}\\
&\leq F_k(\x_k) - \frac{\gamma}{2}\|\x_k - \z_k\|^2 + \frac{ \|\x_k - \z_k\|^2}{2aM_k \eta_k } + \frac{(a+1)\eta_k }{M_k}\\
& \leq g(\x_k)  + r(\x_k)  +   \frac{(a+1)\eta_k }{M_k}
\end{align*}
where we use $F_k(\x_k)\geq F_k(\z_k) + \frac{\gamma}{2}\|\x_k - \z_k\|^2$ due to the strong convexity of $F(\x)$, and $M_k\eta_k \geq  \frac{8a}{\gamma}$.
On the other hand, we have that 
\begin{align*}
    \|\x_{k+1}-\x_{k}\|^2 =& \|\x_{k+1}-\z_{k}+\z_{k}-\x_{s}\|^2\\
    =& \|\x_{k+1}-\z_{k}\|^2+\|\z_{k}-\x_{k}\|^2 + 2\langle \x_{k+1}-\z_{k}, \z_{k}- \x_{k}\rangle\\
    \geq& (1-\hat\alpha^{-1})\|\x_{k+1}-\z_{k}\|^2 + (1-\hat\alpha)\|\x_{k}-\z_{k}\|^2
\end{align*}
where the inequality follows from the Young's inequality with $0<\hat\alpha<1$. Thus we have that
\begin{align*}
 \E\bigg[\frac{\gamma(1-\hat\alpha)}{2}\|\z_{k} - \x_k\|^2\bigg]&\leq \E[g(\x_k) + r(\x_k)  - f_k(\x_{k+1})] + \frac{\gamma(\hat\alpha^{-1}-1)}{2}\E[\|\x_{k+1} - \z_k\|^2] \\
 &   + \frac{(a+1)\eta_k }{M_k}\\
 &\leq \E[F(\x_k) - F(\x_{k+1})] + \frac{\gamma(\hat\alpha^{-1}-1)}{2}\E[\|\x_{k+1} - \z_k\|^2]+ \frac{(a+1)\eta_k }{M_k}.
\end{align*}
By the strong convexity of $F_k(\x)$, we also have 
\begin{align*}
\frac{\gamma}{2}\E[\|\x_{k+1} - \z_k\|^2]\leq \E[F_k(\x_{k+1}) - F_k(\z_k)]\leq \frac{\|\x_k - \z_k\|^2}{2aM_k \eta_k } +\frac{(a+1)\eta_k }{M_k}
\end{align*}
Then we have
\begin{align*}
 \E\bigg[\frac{\gamma(1-\hat\alpha)}{2}\|\z_{k} - \x_k\|^2\bigg]&\leq \E[F(\x_k) - F(\x_{k+1})] + (\hat\alpha^{-1}-1)\left( \frac{\|\x_k - \z_k\|^2}{2aM_k \eta_k } +\frac{(a+1)\eta_k }{M_k}\right) \\
 &  +\frac{(a+1)\eta_k }{M_k}
\end{align*}
Let $\hat\alpha=1/2$, we have
\begin{align*}
 \E\bigg[\frac{\gamma}{4}\|\z_{k} - \x_k\|^2\bigg]&\leq \E[F(\x_k) - F(\x_{k+1})] + \frac{ \|\x_k - \z_k\|^2}{2aM_k \eta_k } + \frac{4\eta_k }{aM_k}\\
 & \leq \E[F(\x_k) - F(\x_{k+1})] + \frac{\gamma\|\x_k - \z_k\|^2}{8}+  \frac{2(a+1)\eta_k }{M_k}
\end{align*}
where we use the fact $M_k\eta_k \geq  \frac{4}{a\gamma}$.  It then gives us 
\begin{align*}
 \E\bigg[\frac{\gamma}{8}\|\z_{k} - \x_k\|^2\bigg]&\leq \E[F(\x_k) - F(\x_{k+1})] + \frac{2(a+1)\eta_k }{M_k}
\end{align*}
By setting $\eta_k = c/\sqrt{k}$, we have
\begin{align*}
 \E\bigg[\frac{\gamma}{8}\|\z_{k} - \x_k\|^2\bigg]&\leq \E[F(\x_k) - F(\x_{k+1})] + \frac{a(a+1)\gamma c^2}{2 k }
\end{align*}
Multiplying both sides by $w_k = k^\alpha$ and taking summation over $k=1,\ldots, K$, we have
\begin{align}\label{eqn:c2}
 \E\bigg[\frac{\gamma}{8}\sum_{k=1}^Kw_k\|\z_{k} - \x_k\|^2\bigg]&\leq \E\bigg[ \sum_{k=1}^Kw_k(F(\x_k) - F(\x_{k+1}))\bigg] +\sum_{k=1}^Kw_k \frac{a(a+1)\gamma c^2}{2 k },
\end{align}
Similar to proof of of Theorem~\ref{thm:3}, we have
\begin{align*}
 \E\bigg[ \sum_{k=1}^Kw_k(F(\x_k) - F(\x_{k+1}))\bigg] \leq  \sum_{k=1}^{K}(w_k - w_{k-1})\E[(F(\x_k) - \min_{\x}F(\x))]\leq \Delta w_{K}
\end{align*}
Then, by $\sum_{k=1}^Kk^{\alpha}\geq \int_0^{K}x^\alpha d x= \frac{K^{\alpha+1}}{\alpha+1}$ and $\sum_{k=1}^Kk^{\alpha-1}\leq K^\alpha,~(\alpha\geq 1)$, we have
\begin{align*}
 \E\bigg[\frac{\gamma}{8}\|\z_{\tau} - \x_\tau\|^2\bigg]\leq  \frac{\Delta K^\alpha}{\sum_{k=1}^{K} k^\alpha} + \frac{\gamma c^2 \sum_{k=1}^{K} k^{\alpha-1} }{2a^2\sum_{k=1}^{K} k^\alpha} \leq \frac{\Delta (\alpha+1)}{K} + \frac{a(a+1)\gamma c^2(\alpha+1)}{2 K},
\end{align*}
which can complete the proof by multiplying both sides by $8\gamma$.

\section{Proof of Theorem~\ref{thm:svrg}}
Recall that 
$F(\x)=\frac{1}{n_1}\sum_{i=1}^{n_1} g_i(\x) + r(\x) - \frac{1}{n_2}\sum_{j=1}^{n_2}h_j(\x)$
and
$F_k(\x)=\frac{1}{n_1}\sum_{i=1}^{n_1} g_i(\x) + r(\x)  - h(\x_k)- \frac{1}{n_2}\sum_{j=1}^{n_2}\partial h_j(\x_k)^{\top}(\x - \x_k) + \frac{\gamma}{2}\|\x - \x_k\|^2$.
For any $\x$, by the convexity of $h(\x)$ we know $F_k(\x) \geq F(\x) + \frac{\gamma}{2}\|\x-\x_k\|^2$.
By applying the result in Proposition~\ref{prop:svrg} to the $k$-th stage, we have
\begin{align*}
\E_k[F_k(\x_{k+1}) - F_k(\z_k)]\leq 0.5^{S_k} \E[F_k(\x_k) - F_k(\z_k)],
\end{align*}
where $\z_k = \arg\min_{\x} F_k(\x)$. Since $F_k(\x_{k+1}) - F_k(\z_k) \geq 0$ and $F_k(\x_{k}) - F_k(\z_k) \geq 0$, then $\E_k[F_k(\x_{k+1}) - F_k(\z_k)]\leq  \E[F_k(\x_k) - F_k(\z_k)]$, which implies $\E_k[F_k(\x_{k+1})]\leq  \E[F_k(\x_k)]$. 
Due to $F(\x_{k+1})\leq F_k(\x_{k+1})$ and $F_k(\x_k) = F(\x_k)$, we have $\E_k[F(\x_{k+1})]\leq F(\x_k)$. Hence $\E[F(\x_k) - F(\x_*)]\leq F(\x_0) - F(\x_*)\leq\Delta$ for all $k$. 
Since 
\begin{align*}
F_k(\x_k) - F_k(\z_k)\leq F(\x_k) - F(\z_k)\leq F(\x_k ) - F(\x_*),
\end{align*}
as a result, we have 
\begin{align*}
\E_k[F_k(\x_{k+1}) - F_k(\z_k)]\leq 0.5^{S_k} [F(\x_k) - F(\x_*)]. 
\end{align*}
To continue the analysis, we have
\begin{align*}
\E[F(\x_{k+1})   + \frac{\gamma}{2}\|\x_{k+1} - \x_k\|^2]&\leq F_k(\z_k) +0.5^{S_k} [F(\x_k) - F(\x_*)]\\
&\leq F_k(\x_k) - \frac{\gamma}{2}\|\x_k - \z_k\|^2 +0.5^{S_k} [F(\x_k) - F(\x_*)]\\
& \leq F(\x_k) +0.5^{S_k} [F(\x_k) - F(\x_*)],
\end{align*}
where we use $F_k(\x_k)\geq F_k(\z_k) + \frac{\gamma}{2}\|\x_k - \z_k\|^2$ due to the strong convexity of $F(\x)$, and $F_k(\x_k)  = F(\x_k) $.
On the other hand, we have that 
\begin{align*}
    \|\x_{k+1}-\x_{k}\|^2 =& \|\x_{k+1}-\z_{k}+\z_{k}-\x_{s}\|^2\\
    =& \|\x_{k+1}-\z_{k}\|^2+\|\z_{k}-\x_{k}\|^2 + 2\langle \x_{k+1}-\z_{k}, \z_{k}- \x_{k}\rangle\\
    \geq& (1-\hat\alpha^{-1})\|\x_{k+1}-\z_{k}\|^2 + (1-\hat\alpha)\|\x_{k}-\z_{k}\|^2
\end{align*}
where the inequality follows from the Young's inequality with $0<\hat\alpha<1$. Thus we have that
\begin{align*}
 \E\bigg[\frac{\gamma(1-\hat\alpha)}{2}\|\z_{k} - \x_k\|^2\bigg]&\leq \E[F(\x_k) - F_k(\x_{k+1})] + \frac{\gamma(\hat\alpha^{-1}-1)}{2}\E[\|\x_{k+1} - \z_k\|^2] \\
 &   +0.5^{S_k} [F(\x_k) - F(\x_*)].
\end{align*}
On the other hand, by the strong convexity of $F_k(\x)$, we also have 
\begin{align*}
\frac{\gamma}{2}\E[\|\x_{k+1} - \z_k\|^2]\leq \E[F_k(\x_{k+1}) - F_k(\z_k)]\leq 0.5^{S_k} [F(\x_k) - F(\x_*)]
\end{align*}
Then we have
\begin{align*}
 \E\bigg[\frac{\gamma(1-\hat\alpha)}{2}\|\z_{k} - \x_k\|^2\bigg]&\leq \E[F(\x_k) - F(\x_{k+1})] + (\hat\alpha^{-1}-1)\left(0.5^{S_k} [F(\x_k) - F(\x_*)]\right) \\
 &  + 0.5^{S_k} [F(\x_k) - F(\x_*)]
\end{align*}
Let $\hat\alpha=1/2$, we have
\begin{align*}
 \E\bigg[\frac{\gamma}{4}\|\z_{k} - \x_k\|^2\bigg]&\leq \E[F(\x_k) - F(\x_{k+1})] + 2 \times 0.5^{S_k} [F(\x_k) - F(\x_*)].
 \end{align*}
Multiplying both sides by $w_k = k^\alpha$ and taking summation over $k=1,\ldots, K$, we have
\begin{align*}
 \E\bigg[\frac{\gamma}{4}\sum_{k=1}^Kw_k\|\z_{k} - \x_k\|^2\bigg]&\leq \E\bigg[ \sum_{k=1}^Kw_k(F(\x_k) - F(\x_{k+1}))\bigg] +2\E\bigg[\sum_{k=1}^Kw_k 0.5^{S_k} [F(\x_k) - F(\x_*)]\bigg],
\end{align*}
Then following the similar analysis as the proof of Theorem~\ref{thm:3},  we have
\begin{align*}
 \E\bigg[\frac{\gamma}{4}\|\z_{\tau} - \x_\tau\|^2\bigg]&\leq \frac{\Delta w_K}{\sum_{k=1}^Kw_k} + \frac{2\Delta\sum_{k=1}^Kw_k0.5^{S_k}}{\sum_{k=1}^Kw_k},
\end{align*}
By noting that $0.5^{S_k}\leq 1/k$,
\begin{align*}
 \E\bigg[\frac{\gamma}{4}\|\z_{\tau} - \x_\tau\|^2\bigg]&\leq  \frac{3\Delta (\alpha+1)}{K},
\end{align*}
which can complete the proof by multiplying both sides by $4\gamma$.

\section{Proof of Theorem~\ref{thm:6}}
The key is to connect the convergence in terms of $\|G_\gamma(\x)\|$ to the convergence in terms of (sub)gradient. 
By Jensen's inequality we know for any random variable $X$ and a convex function $g(x)$,
\begin{align*}
\E[g(X)] \geq g(\E[X]).
\end{align*}
Let $X = \|G_\gamma(\x_\tau)\|^v$ and $g(x) = x^{2/v}$, then $g(x)$ is a convex function since $0<v\le 1$. Therefore, we have
\begin{align*}
(\E[\|G_\gamma(\x_\tau)\|^v])^{2/v} \leq \E[\|G_\gamma(\x_\tau)\|^2] \leq O(1/K),
\end{align*}
which implies
\begin{align}\label{eqn1:thm6}
\E[\|G_\gamma(\x_\tau)\|^v]  \leq O(1/K^{v/2}),~(0<v\leq 1).
\end{align}
We finish the proof by combining inequality (\ref{eqn1:thm6}) and the results in Proposition~\ref{prop:5}.



\section{Proof of Lemma~\ref{lem:DC:new1}}
The proof can be obtained by slightly changing the proof of Theorem~\ref{thm:sgd}.

For the case of smooth $g(\x)$, from the proof of Theorem~\ref{thm:sgd}, we have
\begin{align*}
 \E\bigg[\frac{\gamma_k}{8}\|\z_{k} - \x_k\|^2\bigg]&\leq \E[F(\x_k) - F(\x_{k+1})] + \frac{4G^2}{kL}
\end{align*}
Multiplying by $w_k/\gamma_k$ and then summing over $k=1,\ldots, K$, we have
\begin{align*}
 \E\bigg[\frac{1}{8}\sum_{k=1}^Kw_k\|\z_{k} - \x_k\|^2\bigg]&\leq \E\bigg[ \sum_{k=1}^K\frac{w_k}{\gamma_k}(F(\x_k) - F(\x_{k+1}))\bigg] +\sum_{k=1}^K\frac{w_k}{\gamma_k} \frac{4G^2}{kL},
\end{align*}
According to the setting of $w_k = k^\alpha$ and $\gamma_k = 3Lk^{\frac{1-\nu}{ 1+ \nu}}$ with $\alpha\geq 1$, we have
\begin{align*}
 \E\bigg[\frac{1}{8}\|\z_{\tau} - \x_\tau\|^2\bigg]&\leq \frac{K^{\alpha - {\frac{1-\nu}{ 1+ \nu}}} \Delta (\alpha +1)}{3LK^{\alpha+1}} + \frac{4G^2K^{\alpha - {\frac{1-\nu}{ 1+ \nu}}}(\alpha +1) }{K^{\alpha +1}3L^2},
\end{align*}
which implies that 
\begin{align*}
 \E\bigg[\|\z_{\tau} - \x_\tau\|^2\bigg]&\leq \frac{8\Delta (\alpha +1)}{3LK^{\frac{2}{ 1+ \nu}}} + \frac{32G^2(\alpha +1) }{K^{\frac{2}{ 1+ \nu}}3L^2}.
\end{align*}

For the case of non-smooth $g(\x)$, from the proof of Theorem~\ref{thm:sgd}, we have
\begin{align*}
 \E\bigg[\frac{\gamma_k}{8}\|\z_{k} - \x_k\|^2\bigg]&\leq \E[F(\x_k) - F(\x_{k+1})] + \frac{56G^2}{k}
\end{align*}
Multiplying by $w_k/\gamma_k$ and then summing over $k=1,\ldots, K$, we have
\begin{align*}
 \E\bigg[\frac{1}{8}\sum_{k=1}^Kw_k\|\z_{k} - \x_k\|^2\bigg]&\leq \E\bigg[ \sum_{k=1}^K\frac{w_k}{\gamma_k}(F(\x_k) - F(\x_{k+1}))\bigg] +\sum_{k=1}^K\frac{w_k}{\gamma_k} \frac{56G^2}{k},
\end{align*}
According to the setting of $w_k = k^\alpha$ and $\gamma_k = k^{\frac{1-\nu}{ 1+ \nu}}$ with $\alpha\geq 1$, we have
\begin{align*}
 \E\bigg[\frac{1}{8}\|\z_{\tau} - \x_\tau\|^2\bigg]&\leq \frac{K^{\alpha - {\frac{1-\nu}{ 1+ \nu}}} \Delta (\alpha +1)}{K^{\alpha+1}} + \frac{56G^2K^{\alpha - {\frac{1-\nu}{ 1+ \nu}}}(\alpha +1) }{K^{\alpha +1}},
\end{align*}
which implies that 
\begin{align*}
 \E\bigg[\|\z_{\tau} - \x_\tau\|^2\bigg]&\leq \frac{8\Delta (\alpha +1)}{K^{\frac{2}{ 1+ \nu}}} + \frac{448G^2(\alpha +1) }{K^{\frac{2}{ 1+ \nu}}}.
\end{align*}


\section{Proof of Lemma~\ref{lem:DC:new2}}
From the proof of Theorem~\ref{thm:nadagrad}, we have
\begin{align*}
 \E\bigg[\frac{\gamma_k}{4}\|\z_{k} - \x_k\|^2\bigg]&\leq \E[F(\x_k) - F(\x_{k+1})] + \frac{ \|\x_k - \z_k\|^2}{2aM_k \eta_k } + \frac{4\eta_k }{aM_k}\\
 & \leq \E[F(\x_k) - F(\x_{k+1})] + \frac{\gamma_k\|\x_k - \z_k\|^2}{8}+  \frac{2(a+1)\eta_k }{M_k}
\end{align*}
where we use the fact $M_k\eta_k \geq  \frac{4}{a\gamma_k}$.  Then we have  
\begin{align*}
 \E\bigg[\frac{\gamma_k}{8}\|\z_{k} - \x_k\|^2\bigg]&\leq \E[F(\x_k) - F(\x_{k+1})] + \frac{2(a+1)\eta_k }{M_k}
\end{align*}
By setting $\eta_k = c/\sqrt{\gamma_k k}$ and $M_k\eta_k \geq \frac{4}{a\gamma_k}$, we have
\begin{align*}
 \E\bigg[\frac{\gamma_k}{8}\|\z_{k} - \x_k\|^2\bigg]&\leq \E[F(\x_k) - F(\x_{k+1})] + \frac{a(a+1)c^2}{2 k }
\end{align*}
Multiplying by $w_k/\gamma_k$ and then summing over $k=1,\ldots, K$, we have
\begin{align*}
 \E\bigg[\frac{1}{8}\sum_{k=1}^Kw_k\|\z_{k} - \x_k\|^2\bigg]&\leq \E\bigg[ \sum_{k=1}^K\frac{w_k}{\gamma_k}(F(\x_k) - F(\x_{k+1}))\bigg] +\E\bigg[\sum_{k=1}^K\frac{w_k}{\gamma_k} \frac{a(a+1)c^2}{2k}\bigg],
\end{align*}
According to the setting of $w_k = k^\alpha$ and $\gamma_k = k^{\frac{1-\nu}{ 1+ \nu}}$ with $\alpha\geq 1$, we have
\begin{align*}
 \E\bigg[\frac{1}{8}\|\z_{\tau} - \x_\tau\|^2\bigg]&\leq \frac{K^{\alpha - {\frac{1-\nu}{ 1+ \nu}}} \Delta (\alpha +1)}{K^{\alpha+1}} + \frac{a(a+1)c^2K^{\alpha - {\frac{1-\nu}{ 1+ \nu}}}(\alpha +1) }{2K^{\alpha +1}},
\end{align*}
which implies that 
\begin{align*}
 \E\bigg[\|\z_{\tau} - \x_\tau\|^2\bigg]&\leq \frac{8\Delta (\alpha +1)}{K^{\frac{2}{ 1+ \nu}}} + \frac{4a(a+1)c^2(\alpha +1) }{K^{\frac{2}{ 1+ \nu}}}.
\end{align*}


\section{Proof of Lemma~\ref{lem:DC:new3}}
From the proof of Theorem~\ref{thm:svrg}, we have
\begin{align*}
 \E\bigg[\frac{\gamma_k}{4}\|\z_{k} - \x_k\|^2\bigg]&\leq \E[F(\x_k) - F(\x_{k+1})] +2\E\bigg[\sum_{k=1}^K0.5^{S_k} [F(\x_k) - F(\x_*)]\bigg]
\end{align*}
Multiplying by $w_k/\gamma_k$ and then summing over $k=1,\ldots, K$, we have
\begin{align*}
 \E\bigg[\frac{1}{4}\sum_{k=1}^Kw_k\|\z_{k} - \x_k\|^2\bigg]&\leq \E\bigg[ \sum_{k=1}^K\frac{w_k}{\gamma_k}(F(\x_k) - F(\x_{k+1}))\bigg] +2\E\bigg[\sum_{k=1}^K\frac{w_k}{\gamma_k} 0.5^{S_k} [F(\x_k) - F(\x_*)]\bigg],
\end{align*}
According to the setting of $w_k = k^\alpha$ and $\gamma_k = ck^{\frac{1-\nu}{ 1+ \nu}}$ with $\alpha\geq 1$, and by noting that $0.5^{S_k}\leq 1/k$, we have
\begin{align*}
 \E\bigg[\frac{1}{4}\|\z_{\tau} - \x_\tau\|^2\bigg]&\leq \frac{K^{\alpha - {\frac{1-\nu}{ 1+ \nu}}} \Delta (\alpha +1)}{cK^{\alpha+1}} + \frac{2K^{\alpha - {\frac{1-\nu}{ 1+ \nu}}}\Delta (\alpha +1) }{cK^{\alpha +1}},
\end{align*}
which implies that 
\begin{align*}
 \E\bigg[\|\z_{\tau} - \x_\tau\|^2\bigg]&\leq \frac{12\Delta (\alpha +1)}{cK^{\frac{2}{ 1+ \nu}}}.
\end{align*}


\section{Proof of Lemma~\ref{lem:prox}}
The proof of the first fact can be found in~\citep{Liu2018}[Eqn. 7], and the second fact follows~\citep{RockWets98}[Theorem 10.1 and Exercise 8.8].
Since $r(\cdot)$ is nonnegative proper closed, then $R_{\mu}(\x)$ is convex continuous. By the definition of $r_{\mu}(x)$ and $\text{prox}_{\mu r}(\x)$ we know the supremum in $R_{\mu}(\x)$ is attained at any point in $\text{prox}_{\mu r}(\x)$. Let $\v\in \text{prox}_{\mu r}(\x)$, then for any $\w$ we get
\begin{align*}
R_{\mu}(\w) - R_{\mu}(\x) = & \max_{y\in\R^d} \left\{\frac{1}{\mu} \y^\top\w - \frac{1}{2\mu} \|\y\|^2 - r(\y)\right\}  - \max_{y\in\R^d} \left\{\frac{1}{\mu} \y^\top\x - \frac{1}{2\mu} \|\y\|^2 - r(\y) \right\} \\
\geq &  \left\{\frac{1}{\mu} \v^\top\w - \frac{1}{2\mu} \|\v\|^2 - r(\v)\right\}  -  \left\{\frac{1}{\mu} \v^\top\x - \frac{1}{2\mu} \|\v\|^2 - r(\v) \right\} \\
= &  \frac{1}{\mu} \v^\top(\w - \x),
\end{align*}
which implies $\frac{1}{\mu}\text{prox}_{\mu r}(\x) \subseteq \partial R_{\mu}(\x)$. 
By Theorem 1.25 of~\citep{RockWets98}, the set $\text{prox}_{\mu r}(\x) := \text{Arg}\min_{\y\in\R^d}\left\{\frac{1}{2\mu} \|\y-\x\|^2 + r(\y) \right\} $ is always nonempty since $r$ is proper lower-semicontinuous and bounded below. 
Let $ \v \in \text{prox}_{\mu r}(\x)$, then by Exercise 8.8 (c) and Theorem 10.1 of~\citep{RockWets98} we have 
\begin{align*}
\frac{1}{\mu}(\x- \v) \in \hat \partial r(\v).
\end{align*}

\section{DC decomposition of non-convex sparsity-promoting regularizers}

We present the details of DC decomposition for several regularizers. The following examples are from~\citep{Wen2018,DBLP:conf/icml/GongZLHY13}. 

{\bf Example 1.} The DC decomposition of log-sum penalty (LSP)~\citep{Candades2008} is given by
\begin{align*}
r(\x) := \lambda \sum_{i=1}^{d}\log(|\x_i|+\theta) =  \lambda\frac{\|\x\|_1}{\theta} - \underbrace{\lambda\sum_{i=1}^{d} \left(\frac{|\x_i|}{\theta} - \log(|\x_i|+\theta)\right)}\limits_{r_2(\x)},
\end{align*}
where $\lambda >0$ and $\theta>0$.
It was shown that $r_2(\x)$ is convex and smooth with smoothness parameter $\frac{\lambda}{\theta^2}$.

{\bf Example 2.} The DC decomposition of minimax concave penalty (MCP)~\citep{cunzhang10} is given by
\begin{align*}
r(\x) := \lambda \sum_{i=1}^{d} \int_{0}^{|\x_i|} \left[1- \frac{z}{\theta\lambda}\right]_+ d z = \lambda \|\x\|_1 - \underbrace{\lambda \sum_{i=1}^{d} \int_{0}^{|\x_i|} \min\left\{ 1, \frac{z}{\theta\lambda}\right\} dz}\limits_{r_2(\x)},
\end{align*}
where $\lambda >0$, $\theta>0$, $[z]_+ = \max\{0,z\}$,  
\begin{align*}
\lambda \int_{0}^{|\x_i|} \left[1- \frac{z}{\theta\lambda}\right]_+ dz =     \left\{ \begin{array}{rl}
        \lambda |\x_i| -  \frac{\x_i^2 }{2\theta} & \mbox{if}~ |\x_i| \leq \theta\lambda\\ 
         \frac{\theta\lambda^2}{2}& \mbox{if}~|\x_i| > \theta\lambda
                \end{array}\right.
\end{align*}
and
\begin{align*}
\lambda  \int_{0}^{|\x_i|} \min\left\{ 1, \frac{z}{\theta\lambda}\right\} dz =     \left\{ \begin{array}{rl}
       \frac{\x_i^2 }{2\theta} & \mbox{if}~ |\x_i| \leq \theta\lambda\\ 
         \lambda|\x_i| - \frac{\theta\lambda^2}{2}& \mbox{if}~|\x_i| > \theta\lambda
                \end{array}\right.
\end{align*}
Then $r_2(\x)$ is a convex and smooth function, and the smoothness parameter $\frac{1}{\theta}$.

{\bf Example 3.} The DC decomposition of smoothly clipped absolute deviation (SCAD)~\citep{CIS-172933} is given by
\begin{align*}
r(\x) = \lambda \sum_{i=1}^{d} \int_{0}^{|\x_i|}\min\left\{1, \frac{[\theta\lambda - z]_+}{(\theta-1)\lambda} \right\} dz = \lambda \|\x\|_1 - \underbrace{\lambda \sum_{i=1}^{d} \int_{0}^{|\x_i|} \frac{[\min\{\theta\lambda, z\} - \lambda]_+ }{(\theta-1)\lambda}dz}\limits_{r_2(\x)},
\end{align*}
where $\lambda >0$, $\theta>2$,
\begin{align*}
\lambda \int_{0}^{|\x_i|}\min\left\{1, \frac{[\theta\lambda - z]_+}{(\theta-1)\lambda} \right\} dz =     \left\{ \begin{array}{rcl}
         \lambda |\x_i| & \mbox{if} & |\x_i| \leq \lambda \\ 
         \frac{-\x_i^2 + 2\theta\lambda |\x_i|-\lambda^2}{2(\theta-1)}  & \mbox{if} & \lambda < |\x_i| \leq \theta\lambda\\
         \frac{(\theta+1)\lambda^2}{2}& \mbox{if} &  |\x_i| > \theta\lambda
                \end{array}\right.
\end{align*}
and
\begin{align*}
\lambda \int_{0}^{|\x_i|} \frac{[\min\{\theta\lambda, z\} - \lambda]_+ }{(\theta-1)\lambda}dz =     \left\{ \begin{array}{rcl}
         0& \mbox{if} & |\x_i| \leq \lambda \\ 
         \frac{\x_i^2 - 2\lambda |\x_i|+\lambda^2}{2(\theta-1)}  & \mbox{if} & \lambda < |\x_i| \leq \theta\lambda\\
       \lambda|\x_i| -  \frac{(\theta+1)\lambda^2}{2}& \mbox{if} &  |\x_i| > \theta\lambda
                \end{array}\right.
\end{align*}
Then $r_2(\x)$ was shown to be convex and smooth with modulus $\frac{1}{\theta-1}$.

{\bf Example 4.} The DC decomposition of transformed $\ell_1$ norm~\cite{zhang2018minimization} is given by
\begin{align*}
r(\x) :=  \sum_{i=1}^{d}\frac{(\theta+1)|\x_i|}{\theta+|\x_i|} =  \frac{(1+\theta)\|\x\|_1}{\theta} - \underbrace{\sum_{i=1}^{d}\left[\frac{(\theta+1)|\x_i|}{\theta}  - \frac{(\theta+1)|\x_i|}{\theta+|\x_i|} \right]}\limits_{r_2(\x)},
\end{align*}
where $\theta>0$. The function $r_2(\x)$ is smooth with parameter $\frac{2(1+\theta)}{\theta^2}$.

{\bf Example 5.} The DC decomposition of capped $\ell_1$ penalty~\citep{Zhang:2010:AMC:1756006.1756041} is given by
\begin{align*}
r(\x) :=  \lambda \sum_{i=1}^{d}\min\{|\x_i|,\theta \} =  \lambda \|\x\|_1 - \underbrace{\lambda \sum_{i=1}^{d}[|\x_i|-\theta]_+}\limits_{r_2(\x)},
\end{align*}
where $\lambda >0$, $\theta>0$.

\bibliography{all}
\end{document}